\documentclass[11pt]{amsart}
\usepackage[dvipsnames]{xcolor}
\usepackage[margin= 1in]{geometry}
\usepackage{graphicx,dsfont}
\usepackage{amssymb}
\usepackage{mathtools}
\usepackage{tikz-cd}

\usepackage[title]{appendix}
\usepackage{amsmath, mathrsfs, stmaryrd, color,slashed,bbm}
\usepackage{hyperref}
\hypersetup{colorlinks = true, urlcolor = blue,linkcolor=BrickRed, citecolor=Orchid!80!Black, bookmarksopen = true}
\usepackage{float}
\usepackage{verbatim}
\usepackage{comment}
\usepackage[utf8]{inputenc}
\usepackage[english]{babel}
\usepackage{enumitem}
\usepackage{booktabs}

\newtheorem{theorem}{Theorem}
\newtheorem{remark}[theorem]{Remark}

\newtheorem{definition}[theorem]{Definition}
\newtheorem{proposition}[theorem]{Proposition}
\newtheorem{corollary}[theorem]{Corollary}
\newtheorem{lemma}[theorem]{Lemma}

\newcommand{\bP}{\mathbb{P}}

\numberwithin{theorem}{section}
\numberwithin{equation}{section}

\begin{document}
\title[]{Determinantal structures for Bessel fields}
\author[L. Benigni]{Lucas Benigni}
\address{Department of Mathematics, University of Chicago, 
Chicago IL, 60637}
 \email{lbenigni@uchicago.edu}

\author[P.-K. Hung]{Pei-Ken Hung}
\address{School of Mathematics, University of Minnesota, 
Minneapolis, MN 55455}
 \email{pkhung@umn.edu}

\author[X. Wu]{Xuan Wu}
\address{Department of Mathematics, University of Chicago, 
Chicago IL, 60637} 
\email{xuanw@uchicago.edu}
\begin{abstract}
A Bessel field $\mathcal{B}=\{\mathcal{B}(\alpha,t), \alpha\in\mathbb{N}_0, t\in\mathbb{R}\}$ is a two-variable random field such that for every $(\alpha,t)$, $\mathcal{B}(\alpha,t)$ has the law of a Bessel point process with index $\alpha$. The Bessel fields arise as hard edge scaling limits of the Laguerre field, a natural extension of the classical Laguerre unitary ensemble. It is recently proved in \cite{LW} that for fixed $\alpha$, $\{\mathcal{B}(\alpha,t), t\in\mathbb{R}\}$ is a squared Bessel Gibbsian line ensemble. In this paper, we discover rich integrable structures for the Bessel fields: along a time-like or a space-like path,  $\mathcal{B}$ is a determinantal point process with an explicit correlation kernel; for fixed $t$, $\{\mathcal{B}(\alpha,t),\alpha\in\mathbb{N}_0\}$ is an exponential Gibbsian line ensemble. 
\end{abstract}
\maketitle

\section{Introduction and main results}

The Laguerre Unitary Ensemble (LUE) or complex Wishart ensemble has been extensively studied since its introduction in \cite{Wis} which is considered to be the first occurrence of random matrices in the scientific literature. Over the last six decades, this model has been shown to be exactly solvable and many statistics such as the joint eigenvalue distribution and its scaling limit can be computed. A rich behavior has been unearthed as the \emph{Airy}, \emph{Sine} and \emph{Bessel} kernels arise respectively for the largest eigenvalues (and possibly the smallest in the \emph{soft edge} scaling), for bulk eigenvalues and for the smallest eigenvalues under the \emph{hard edge} scaling. A dynamical model of the LUE has been introduced in \cite{bru1, bru} (though for real variables) to study dynamically the Principal Component Analysis in high-dimensional multi-variate statistics. It has later been showed in \cite{KO} that this stochastic process has the same law as $N$ independent squared Bessel processes conditioned to never collide, a similar representation of a dynamical eigenvalue process as the Dyson Brownian motion \cite{Dys} for the Gaussian Unitary Ensemble.

The LUE has also been related to combinatorial stochastic processes not only asymptotically, as the kernels mentioned above are ubiquitous within the Kardar--Parisi--Zhang universality class, but also at any finite dimension since it was shown in the seminal work \cite{johansson} that the largest eigenvalue of the LUE has exactly the same distribution as the \emph{exponential last passage percolation} (ELPP). In this paper, we introduce a process on $\mathbb{N}_0\times \mathbb{R}_+$ which we call the \emph{Laguerre field} as well as its renormalization under the hard edge scaling: the \emph{Bessel field}. This model encompasses both dynamical models such as Wishart processes and combinatorial models such as non-intersecting exponential random walks. It is a Markov process which enjoys many properties such as a determinantal structure, a form of Gibbs property (which we call \emph{exponential} Gibbs property) and a scaling limit for its correlation kernels. The scaling limit of its actual trajectories has also been studied and the convergence to the \emph{Airy line ensemble} or the \emph{Bessel line ensemble} has been proved \cite{dauvergne2019uniform, LW} along certain paths in the process.    
\subsection{The Laguerre field}
Let $B_{ij}(t)$, $i,j\in\mathbb{N}$, be a collection of i.i.d. complex Brownian motions whose real and imaginary parts both have mean zero and variance $t/2$. That is,
\begin{align*}
B_{ij}(t)=\frac{1}{\sqrt{2}}b_{ij}(t)+\frac{\textrm{i}}{\sqrt{2}}\tilde{b}_{ij}(t), 
\end{align*}
where $b_{ij}(t)$ and $\tilde{b}_{ij}(t)$ are independent Brownian motions with mean zero and diffusive parameter one. For $N\in\mathbb{N}$, $\alpha\in \mathbb{N}_0$ and $t\geq 0$, let $A^{N}(\alpha, t)$ be the $(N+\alpha)\times N$ matrix defined by
\begin{equation}\label{def:A}
 A^N(\alpha,t)\coloneqq \{B_{ij}(t)\}_{1\leq i\leq N+\alpha,1\leq j\leq N}.   
\end{equation}

Consider the Hermitian matrix $\big(A^{N}(\alpha,t) \big)^{*}A^{N}(\alpha,t)$, with $A^{*}$ being the conjugate transpose of $A$. Let $0\leq X_{1}^N(\alpha,t)\leq X_2^N(\alpha,t)\leq \cdots \leq X_N^N(\alpha,t)$ be the ordered eigenvalues of$\big(A^{N}(\alpha,t) \big)^{*}A^{N}(\alpha,t)$. We denote 
\begin{equation}\label{def:X}
 \vec{X}^N(\alpha,t):=\{X^N_i(\alpha,t)\}_{j=1}^N. 
\end{equation}
For fixed $\alpha \in\mathbb{N}_0$ and $t=1$, $\vec{X}^N(\alpha,1)$ is the Laguerre ensemble (or complex Wishart ensemble). In this paper, we are interested in the dynamical and asymptotic behavior of $\vec{X}^N(\alpha,t)$ when both $\alpha$ and $t$ vary. We refer to it as the {\em Laguerre field}.

\begin{figure}[!ht]
	\centering
	\includegraphics[width=.6\linewidth]{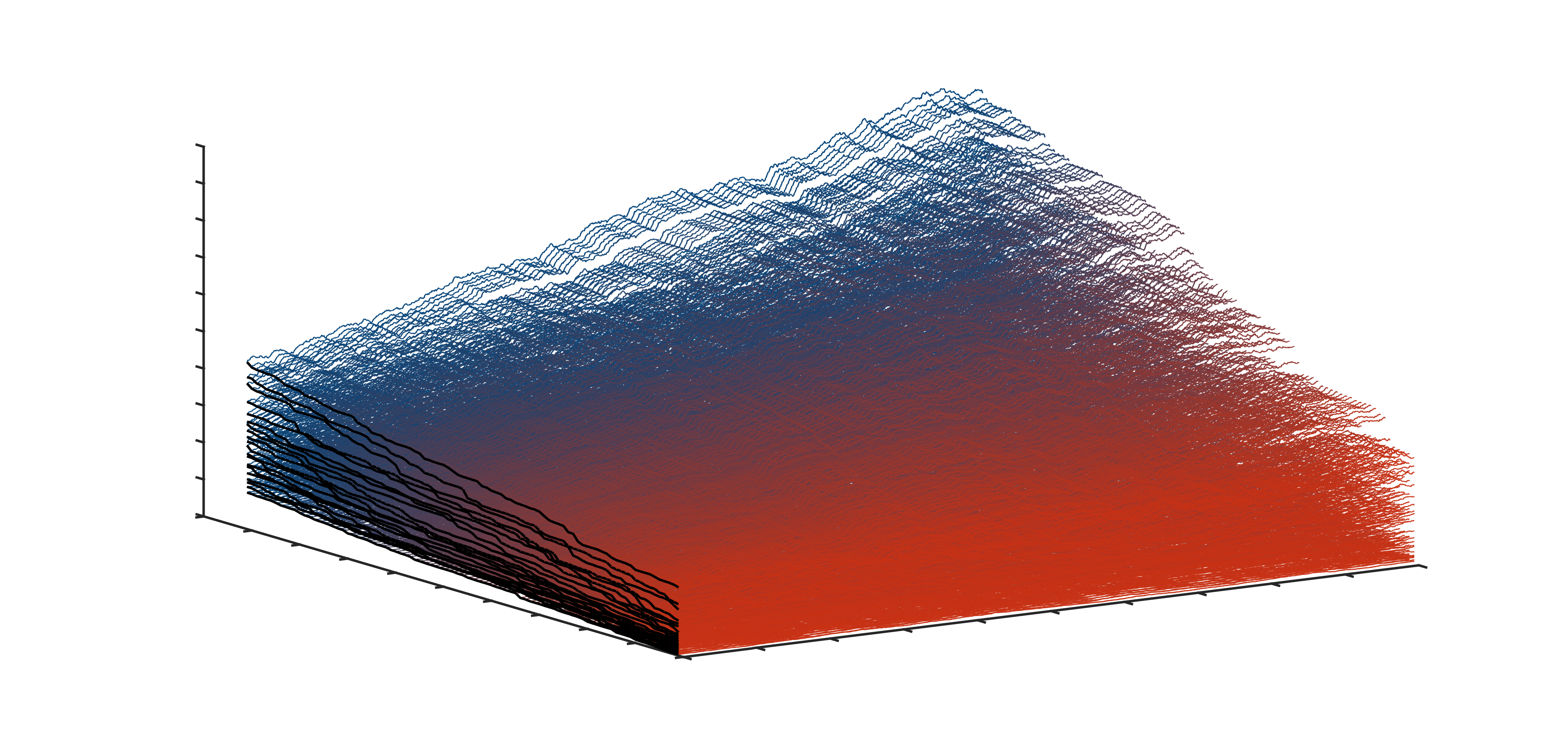}\\[-3ex]
	\begin{minipage}{.4\linewidth}
		\centering
		\includegraphics[width=.8\linewidth]{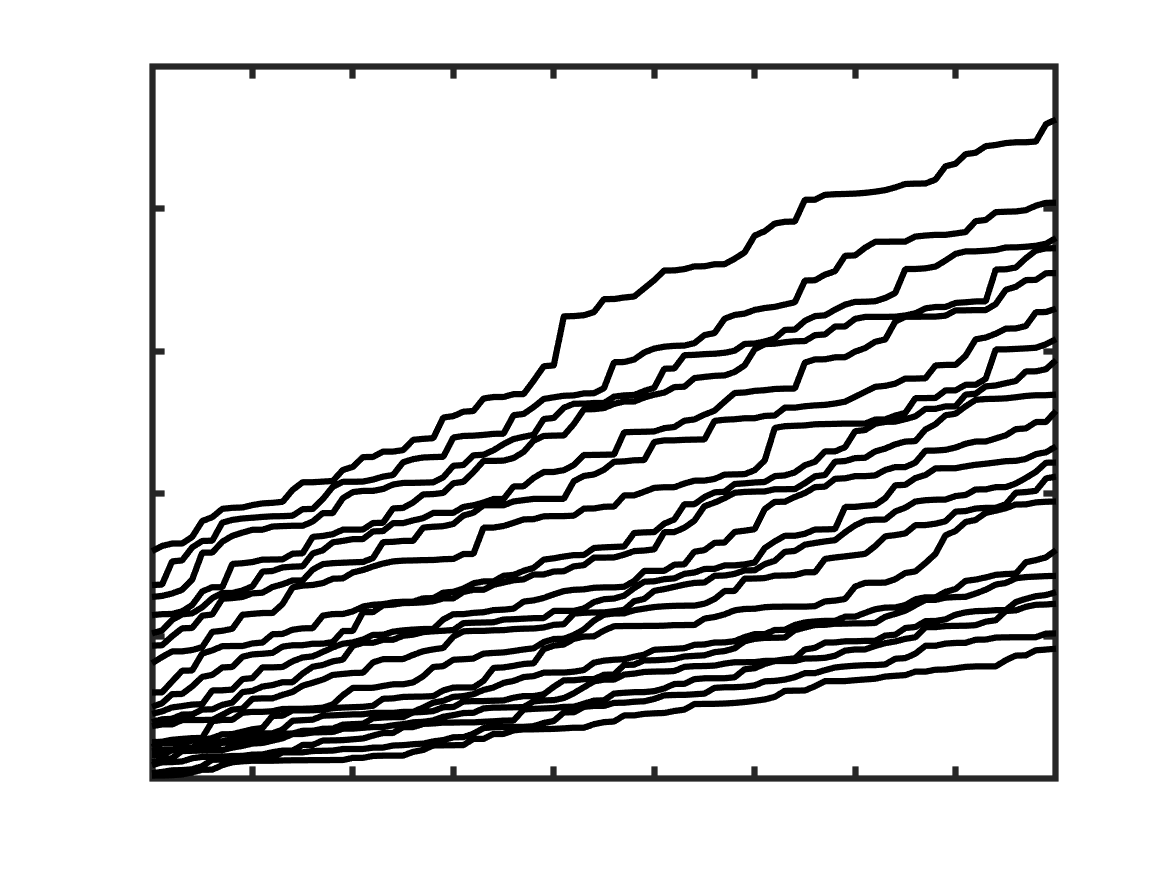}
	\end{minipage}
\begin{minipage}{.4\linewidth}
	\centering
	\includegraphics[width=.8\linewidth]{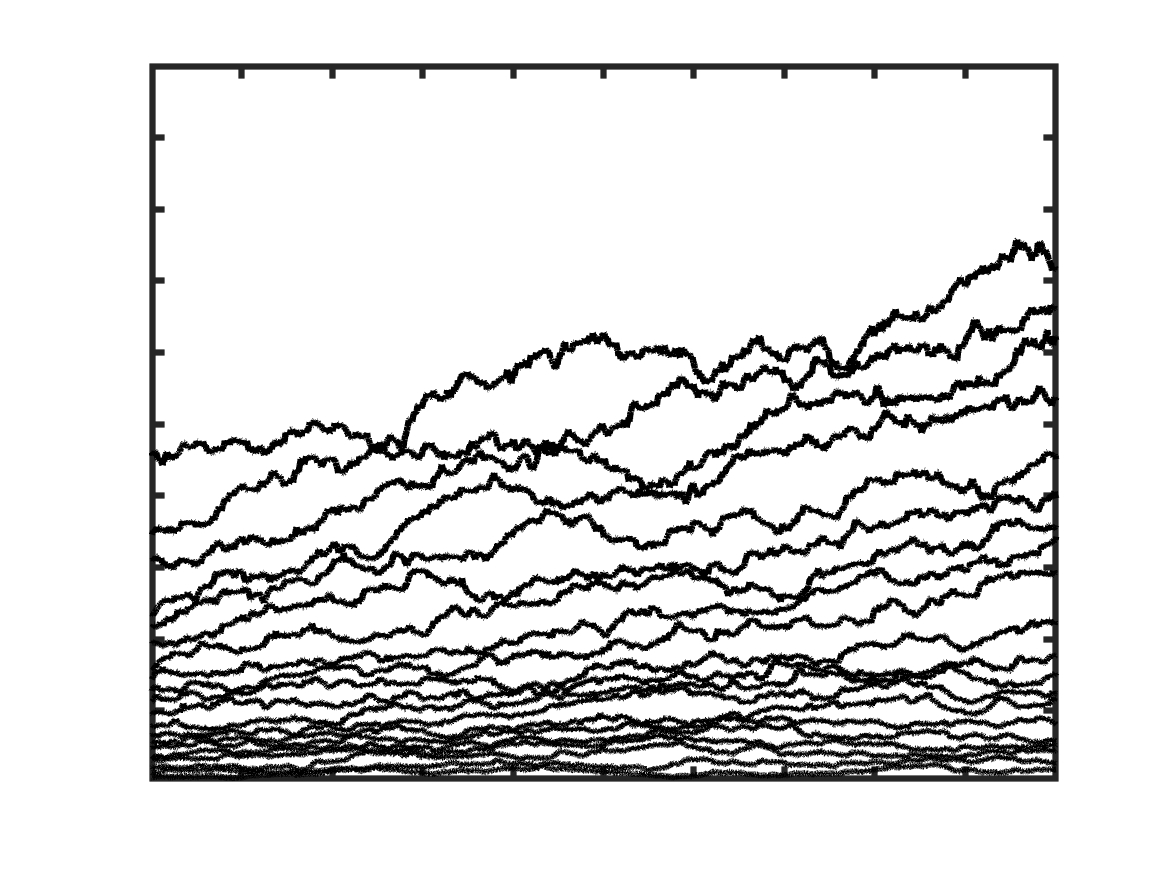}
\end{minipage}
\caption{Illustration of the Laguerre field for $N=20$, $\alpha\in[10,100]$ and $t\in[0,1]$. For $t$ fixed (bottom left), the process consists in non-intersecting exponential random walks while for $\alpha$ fixed (bottom right), the process consists in $N$ non-intersecting squared Bessel processes}
\end{figure}

\subsection{The hard edge scaling}
The study of the asymptotic behavior of the eigenvalues of the Laguerre ensemble and its
variants has been an important topic in the random matrix community, see for instance \cite{For93, For94, FT, JV, TW}. There are three natural scaling regimes, the {\em bulk scaling}, the {\em soft edge} scaling
and the {\em hard edge} scaling. 

In this paper we investigate the Laguerre field $\vec{X}^N(\alpha,t)$ under the hard edge scaling. For $\alpha\in\mathbb{N}_0$ and $t\in [-4N, \infty)$, define
\begin{align}\label{eqn:scaledBfield}
\mathcal{B}^N_i(\alpha,t):= 4N\cdot X^N_i\left(\alpha,1+\frac{t}{4N}\right).
\end{align}

Define the {\em scaled Laguerre field} $\mathcal{B}^N:=\left\{\mathcal{B}^N_i(\alpha,t), i\in\mathbb{N},  \alpha\in\mathbb{N}_0, t\in\mathbb{R}\right\}$. We first examine some known results about the asymptotic behavior of $\mathcal{B}^N$. For fixed $\alpha\in\mathbb{N}_0, t\in\mathbb{R}$, \cite{For93} showed that the correlation kernel of $\mathcal{B}^N(\alpha,t)$ converges locally uniformly to the Bessel kernel. This implies the convergence to the Bessel point process, for which the distribution of the smallest particle was studied in \cite{TW}. When $\alpha$ is being fixed, a very recent result by the third author and Lawler \cite{LW}, constructed the Bessel line ensemble with index $\alpha$ through taking a functional limit along hard edge scaling. The key to their argument is the squared Bessel Gibbs property.

Let us remark that none of the above results have investigated the dynamics in $\alpha$ and $t$ simultaneously. In this paper, we are interested in understanding the limit of $\mathcal{B}^N$ to the full generality, i.e. as a sequence of random fields in both of $\alpha\in\mathbb{N}_0$ and $t\in\mathbb{R}$. For fixed $t$, we identify the evolution of the Laguerre field $\vec{X}(\cdot,t)$ as the evolution of interlacing exponential random walks. We expect this is well-known by experts in the field but we did not find a reference which has made explicit this identification.

\subsection{Main results}\label{sec:1.3}
Because the variable $\alpha$ is discrete, from the functional limit result of \cite{LW} on can easily show that a subsequence of $\mathcal{B}^N(\alpha,t)$ converges in distribution to a limiting random field. Potentially there might be more than one limit depending on the subsequence, but we believe the limit is unique. From now on, we write $\mathcal{B}(\alpha,t)$ for a limiting random field and refer to it as a \textit{Bessel field}.

In this paper, we show that $\mathcal{B}(\alpha,t)$ enjoys two integrable structures: determinantal point processes and Gibbsian line ensembles. To state our result, let us first introduce two definitions, the space-like and time-like paths.

\begin{definition}[Space-like paths]\label{def:spacelike}
Let $(\alpha,t)$ and $(\beta,s)$ be two pairs in $\mathbb{N}_0\times \mathbb{R}$. We denote $(\alpha,t)\prec_{\textup{s}} (\beta,s)$ if it holds that
\begin{align}\label{spacelike}
\alpha\geq \beta,\ t\leq s\ \textup{and}\ (\alpha,t)\neq (\beta,s).
\end{align}
Also, we denote $(\alpha,t)\preceq_{\textup{s}} (\beta,s)$ if $(\alpha,t)\prec_{\textup{s}} (\beta,s)$ or $(\alpha,t)= (\beta,s)$ holds.

We say a path $\left\{\left(\alpha^{(k)},t^{(k)}\right)\right\}_{k=1}^m \subset \mathbb{N}_0\times \mathbb{R}$ is space-like if $$(\alpha^{(1)},t^{(1)})\prec_{\textup{s}} (\alpha^{(2)},t^{(2)})\prec_{\textup{s}} \dots\prec_{\textup{s}} (\alpha^{(m)},t^{(m)}).$$
\end{definition} 

The notion of space-like paths is introduced by Borodin and Ferrari \cite{BF08} for the PushASEP model. We further consider the following analogous paths which we call \textit{time-like}. 

\begin{definition}[Time-like paths]\label{def:timelike}
Let $(\alpha,t)$ and $(\beta,s)$ be two pairs in $\mathbb{N}_0\times \mathbb{R}$. We denote $(\alpha,t)\prec_{\textup{t}} (\beta,s)$ if it holds that
\begin{align}\label{timelike}
\alpha\leq \beta,\ t\leq s\ \textup{and}\  (\alpha,t)\neq (\beta,s).
\end{align}
Also, we denote $(\alpha,t)\preceq_{\textup{t}} (\beta,s)$ if $(\alpha,t)\prec_{\textup{t}} (\beta,s)$ or $(\alpha,t)= (\beta,s)$ holds.

We say a path $\left\{\left(\alpha^{(k)},t^{(k)}\right)\right\}_{k=1}^m \subset \mathbb{N}_0\times \mathbb{R}$ is time-like provided $$(\alpha^{(1)},t^{(1)})\prec_{\textup{t}} (\alpha^{(2)},t^{(2)})\prec_{\textup{t}} \dots \prec_{\textup{t}} (\alpha^{(m)},t^{(m)}).$$
\end{definition}

The following is our main theorem which describes various integrable structures of the Bessel field.
\begin{theorem}\label{thm:main}
Let $\mathcal{B}=\{\mathcal{B}(\alpha,t), \alpha\in\mathbb{N}_0, t\in\mathbb{R}\}$ be a Bessel field. Then $\mathcal{B}$ enjoys the following properties. 
\begin{enumerate}[label=(\roman*)]
\item Let $\left\{\left(\alpha^{(k)},t^{(k)}\right)\right\}_{k=1}^m$ be a time-like path. Then $\left\{\mathcal{B}\left(\alpha^{(k)},t^{(k)}\right)\right\}_{k=1}^m$ is a determinantal point process. The correlation kernel is given by $K^{\textup{Bes}}((\alpha^{(k)},t^{(k)},x);(\alpha^{(\ell)},t^{(\ell)},y))$  such that\\
$K^{\textup{Bes}}((\alpha,t,x);(\beta,s,y))$ is defined through 
\begin{equation}\label{def:KBes}
 \left\{ \begin{array}{cc}
-\displaystyle\int_{1/4}^{\infty} du\, u^{-(\beta-\alpha)/2}e^{-(s-t)u}J_\alpha(2\sqrt{xu})J_\beta(2\sqrt{yu}), & (\alpha,t)\prec_{\textup{t}} (\beta,s),\\[0.5cm]
\displaystyle \int_{0}^{1/4} du\, u^{-(\beta-\alpha)/2}e^{-(s-t)u}J_\alpha(2\sqrt{xu})J_\beta(2\sqrt{yu}), & (\beta,s)\preceq_{\textup{t}}   (\alpha,t).
\end{array}\right.
\end{equation}  
Here $J_\alpha(z)$ is the Bessel function of the first kind;
\item Let $\left\{\left(\alpha^{(k)},t^{(k)}\right)\right\}_{k=1}^m$ be a space-like path. Then $\left\{\mathcal{B}\left(\alpha^{(k)},t^{(k)}\right)\right\}_{k=1}^m$ is a determinantal point process. The correlation kernel is given by $\overline{K}^{\textup{Bes}}((\alpha^{(k)},t^{(k)},x);(\alpha^{(\ell)},t^{(\ell)},y))$ such that\\
$\overline{K}^{\textup{Bes}}((\alpha,t,x);(\beta,s,y))$ is defined through
\begin{equation}\label{def:KBes-s}
 \left\{ \begin{array}{cc}
-\displaystyle\int_{1/4}^{\infty} du\, u^{-(\alpha-\beta)/2}e^{-(s-t)u}J_\alpha(2\sqrt{xu})J_\beta(2\sqrt{yu}), & (\alpha,t)\prec_{\textup{s}} (\beta,s),\\[0.5cm]
\displaystyle \int_{0}^{1/4} du\, u^{-(\alpha-\beta)/2}e^{-(s-t)u}J_\alpha(2\sqrt{xu})J_\beta(2\sqrt{yu}), & (\beta,s)\preceq_{\textup{s}}   (\alpha,t).
\end{array}\right.
\end{equation}
Here $J_\alpha(z)$ is the Bessel function of the first kind; 
\item For fixed $t$, $\{\mathcal{B}(\alpha , t),\ \alpha\in\mathbb{N}_0\}$ satisfies the exponential Gibbs property.
\end{enumerate} 
\end{theorem}
\begin{remark}
By putting $m=1$ in Theorem~\ref{thm:main} (i) or (ii), we see that $\mathcal{B}(\alpha,t)$ is the \textit{Bessel point process} with index $\alpha$ \cite{For93}. Furthermore, if index $\alpha$ is unchanged along a time-like or space-like path, Theorem~\ref{thm:main} (i) or (ii) recovers the \textit{extended Bessel point process} with index $\alpha$ \cite{KT11}.
\end{remark}
\begin{remark}
For fixed $\alpha\in\mathbb{N}_0$, $\{\mathcal{B}(\alpha,t),\ t\in\mathbb{R}\}$ is the Bessel line ensemble studied in \cite{LW} and satisfies the \textit{squared Bessel Gibbs property}. See \cite{LW} for more details.
\end{remark}
We refer readers to Section~\ref{sec:pointprocess} and Section~\ref{sec:Gibbs} for the definitions of determinantal point processes and exponential Gibbs property. 

\subsection{The proof strategies}
A key property we develop is that along a time-like or space-like path, the Laguerre field $\vec{X}^N$ forms a Markov process. The Markov property permits explicit expressions of the joint density functions. In particular, the transition probabilities have an integral representation which involves the Bessel function $J_\alpha(z)$. Having the joint densities, we apply the Eynard--Mehta theorem \cite{EM,BR} to show that $\vec{X}$ along a time-like for space-like path is determinantal and compute the correlation kernels. We then obtain the determinantal structures of $\mathcal{B}(\alpha,t)$  by evaluating the limit of correlation kernels under the hard edge scaling \eqref{eqn:scaledBfield}.

To show that for fixed $t$, $\mathcal{B}(\alpha,t)$ satisfies the exponential Gibbs property, we first show that the Laguerre field $\vec{X}^N(\alpha,t)$ enjoys the same structure. The origin of the exponential Gibbs property of $\vec{X}^N(\alpha,t)$ comes from the feature that $\vec{X}^N(\alpha,t)$ can be identified as exponential last passage percolation (ELPP). However, we avoid introducing the ELPP model and give a self-contained proof based on Sasamoto's trick. We then show that the exponential Gibbs property survives under the limit through a coupling argument.

In \cite{FF}, Ferrari and Frings consider the range $\alpha\in \llbracket -N+1,0 \rrbracket$. In this regime, the number of eigenvalues depends on $\alpha$ and decreases along space-like paths. They show that in this range, eigenvalues along the space-like paths are both Markov and determinantal. In this paper, we focus on the range $\alpha\in \mathbb{N}_0$. 

A determinantal structure has also been unearthed on a specific time-like path in \cite{ipsen2019laguerre} where the \emph{Laguerre unitary process} has been introduced.  It consists in taking a path of the form $\{(N(\alpha+k),b)\}_{k=0}^m$ for a fixed $b$. Note that the Bessel scaling we are considering in this paper is not interesting in this context since the smallest eigenvalue is in the \emph{soft edge} scaling.

\subsection*{Outline} Section~\ref{sec:2} gives the definition and properties of determinantal processes as well as defining the exponential Gibbs property. In Section~\ref{sec:3}, we study the Bessel field along time-like paths and perform a similar analysis in Section~\ref{sec:4} along space-like paths under the assumption that the field along those paths are Markovian. We prove that the Bessel fields enjoys the exponential Gibbs property in Section~\ref{sec:5} and prove the Markov property of our model in Section~\ref{sec:A}. Finally, we give some properties of Bessel functions and the Hankel transform in Appendix~\ref{sec:B}.

\subsection*{Notation} The natural numbers are defined to be $\mathbb{N}=\{1,2,\dots\}$ and we write $\mathbb{N}_0=\mathbb{N}\cup\{0\}$. We use the notation $\mathbb{R}_+\coloneqq [0,\infty)$. For $N\in\mathbb{N}$, the $N$-dimensional Weyl chamber is defined by
\begin{equation}
\mathbb{W}^N \coloneqq \left\{ \vec{x}=(x_1,x_2,\dots,x_N)\in \mathbb{R}^N\,|\,   x_1< x_2< \dots< x_N\right\}.
\end{equation} 
We are mainly concerned with the Weyl chamber restricted on the positive reals
\begin{align}\label{def:Weylchamber}
\mathbb{W}^N_+\coloneqq \left\{ \vec{x}=(x_1,x_2,\dots,x_N)\in \mathbb{R}^N\,|\,0< x_1< x_2< \dots< x_N\right\}.
\end{align}
For $\vec{x}=(x_1,x_2,\dots,x_N)\in\mathbb{R}^N$, we use $\Delta(\vec{x})$ to denote the Vandermonde determinant
\begin{equation}\label{def:Vandermonde}
\Delta(\vec{x})\coloneqq \prod_{1\leq i<j\leq N} (x_j-x_i).
\end{equation} 
For two vectors $\vec{x}, \vec{y}\in \mathbb{W}^N$, we denote by $\vec{x}\prec\vec{y}$ if
\begin{equation}\label{def:interlacing}
x_1<y_1<x_2<\cdots<x_N<y_N.
\end{equation}
Also, we define the \textit{weakly interlacing} relation, $\vec{x}\preceq\vec{y}$, by
\begin{equation}\label{def:interlacingweak}
 x_1\leq  y_1\leq  x_2\leq \cdots\leq  x_N\leq  y_N. 
\end{equation}
For a set $E$, we use $\mathbbm{1}_E(x)$ to denote the indicator function of $E$. We may also use the defining property of $E$ to express the indicator function. For instance, $\mathbbm{1}(x\leq y )$ stands for the indicator function of the set $\{(x,y)\in\mathbb{R}^2\, |\, x\leq y \}$.

\section{Preliminaries}\label{sec:2}
In this section we introduce two integrable structures we use to describe a Bessel field. In Section~\ref{sec:pointprocess}, we give the definition of a determinantal point process. The exponential Gibbs property is discussed in Section~\ref{sec:Gibbs}.

\subsection{Determinantal Point Process}\label{sec:pointprocess} In this section we give a brief introduction to \textit{determinantal point processes}. We closely follow \cite{Soshnikov} and make several specializations for our needs. We refer readers to \cite{Soshnikov} for a more general treatment. 

Let $\mathbb{R}_+=[0,\infty)$ endowed with the standard topology. The one-particle space we consider is of the form
\begin{equation}\label{def:statespace}
\mathfrak{X}=\sqcup_{k=1}^m \mathbb{R}_+^{(k)}. 
\end{equation}
Here $\sqcup$ stand for the disjoint union, $m\in\mathbb{N}$ and $\mathbb{R}_+^{(k)}$ are distinct copies of $\mathbb{R}_+$. We equip $\mathfrak{X}$ with the disjoint union topology. We remark that the difference between $m=1$ and $m\geq 2$ is inconsequential until we specify the regularity of the correlation kernel. We will simply use $x,y$ to stand for variables in $\mathfrak{X}$ and don't specify which $\mathbb{R}_+^{(k)}$ $x$ lies in if not necessary.

 A \textit{point configuration} is a locally finite subset of $\mathfrak{X}$. That is, $X\subset\mathfrak{X}$ is a point configuration if for any compact set $F\subset \mathfrak{X}$, we have $|X\cap F|<\infty$. We write $\textup{Conf}(\mathfrak{X})$ the collection of point configurations. For any bounded borel set $E\subset \mathfrak{X}$, define the function $M_E:\textup{Conf}(\mathfrak{X})\to\mathbb{N}_0$ by $M_E(X)\coloneqq |X\cap E|$. A set of the form $\{M_E=n\}\subset\textup{Conf}( \mathfrak{X})$ is called a \textit{cylindrical set}. Let $\mathcal{F}_{\textup{cyl}} $ be the sigma algebra generated by all cylindrical sets. 
\begin{definition}
A random point process on $\mathfrak{X}$ is a probability measure $\mathcal{P}$ on $(\textup{Conf}(\mathfrak{X}),\mathcal{F}_{\textup{cyl}})$. We write $\mathcal{E}$ for the corresponding expectation.
\end{definition} 

An ordered random vector can be identified with a point process. Let $\vec{X}=(\vec{X}^{(1)},\dots, \vec{X}^{(m)})$ be random vectors such that $\vec{X}^{(k)}$ takes values in $\in \mathbb{W}_+^{N_k}$ for all $k\in\llbracket 1,m \rrbracket$ and for some $N_k\in\mathbb{N}\cup\{\infty\}$. We can view a realization of $\vec{X}$ as a subset of $\mathfrak{X}$. Assume that $\vec{X}\in\textup{Conf}(\mathfrak{X})$ almost surely. Then we define
\begin{align}\label{def:ppid}
\mathcal{P} (M_{E_1}=n_1, \dots,M_{E_\ell}=n_\ell)\coloneqq \mathbb{P}(M_{E_1}(\vec{X})=n_1,\dots,M_{E_\ell}(\vec{X})=n_\ell).
\end{align}
The right hand side of \eqref{def:ppid} is evaluated through the law of $\vec{X}$. Through the Carath\'eodory's extension theorem, the left hand side of \eqref{def:ppid} can be extended to a probability measure on $(\textup{Conf}(\mathfrak{X}),\mathcal{F}_{\textup{cyl}})$. This defines a point process corresponding to $\vec{X}$.

Now we define the \textit{correlation function} of a point process. 
\begin{definition}\label{def:correlation}
Let $\mathcal{P}$ be a point process on $(\textup{Conf}(\mathfrak{X}),\mathcal{F}_{\textup{cyl}})$. For $k\in \mathbb{N}$, a symmetric, locally integrable function $\rho_k:\mathfrak{X}^k\to [0,\infty)$ is called the $k$-point correlation function if the following holds. For any $\ell\in\mathbb{N}$, bounded Borel sets $E_1,E_2,\dots, E_\ell$ and $k_i\in\mathbb{N}$ with $\sum_{i=1}^\ell k_i=k$, it holds that
\begin{align}
\mathcal{E}\left[ \prod_{i=1}^\ell \frac{M_{E_i}!}{(M_{E_i}-k_i)!} \right]=\int_{E_1^{k_1}\times E_2^{k_2}\times\dots\times E_\ell^{k_\ell}} d x_1 \dots d x_k\, \rho_k(x_1,x_2,\dots,x_k)  .
\end{align}
Here we use $dx$ to denote the Lebesgue measure on $\mathfrak{X}=\sqcup_{k=1}^m\mathbb{R}^{(k)}_+$. 
\end{definition}

\begin{definition}\label{def:detPP}
Let $\mathcal{P}$ be a point process on $(\textup{Conf}(\mathfrak{X}),\mathcal{F}_{\textup{cyl}})$. $\mathcal{P}$ is called determinantal if the following statements hold. For all $k\in\mathbb{N}$, the $k$-point correlation function $\rho_k:\mathfrak{X}^k\to[0,\infty) $ exists. Moreover, there exists a function $K:\mathfrak{X}\times\mathfrak{X}\to\mathbb{R}$ such that for all $k\in\mathbb{N}$, we have
\begin{equation}
\rho_k(x_1,x_2,\dots,x_k)=\det[ K(x_i,x_j) ]_{1\leq i,j\leq k}.
\end{equation}
A random vector $\vec{X}=(\vec{X}^{(1)},\vec{X}^{(2)},\dots, \vec{X}^{(m)})$ is called determinantal if the corresponding point process given in \eqref{def:ppid} is determinantal. The function $K(x,y)$ is called the correlation kernel. 
\end{definition}

A lot of information of a determinantal point process is encoded in the correlation kernel $K(x,y)$ and can be extracted through Fredholm determinants. Usually, for Fredholm determinants to be well-defined, one requires $K(x,y)$ to be in the (locally) trace class. Moreover, to relate the Fredholm determinants with the correlation functions, one often assumes $K(x,y)$ is a continuous function. For example, see Theorem 2 in \cite{Soshnikov}. However, \textbf{neither} of these holds for the Laguerre field, see Remark~\ref{rmk:nontrace}. Therefore, we need a more general setup which requires less regularity for ``off-diagonal" terms.

Let $E\subset \mathfrak{X}$ be a bounded Borel set. $E$ can be divided into $E^{(k)}=E\cap\mathbb{R}_+^{(k)}$. We express an operator $A:L^2(E)\to L^2(E)$ in the matrix form
\begin{align}\label{equ:Asep}
A=\left[ \begin{array}{cccc}
A_{(1)(1)} & A_{(1)(2)} & \cdots & A_{(1)(m)}\\
A_{(2)(1)} & A_{(2)(2)} & \cdots & A_{(2)(m)}\\
\vdots & \vdots & \ddots & \vdots\\
A_{(m)(1)} & A_{(m)(2)} & \cdots & A_{(m)(m)}
\end{array} \right].
\end{align}
Here $A_{(i)(j)}$ is an operator from $:L^2(E^{(j)})$ to $L^2(E^{(i)})$. We denote by $\mathcal{L}_{1|2}(E)$ the collection of $A$ such that $A_{(i)(i)}$ are of trace class and $A_{(i)(j)}$ are Hilbert-Schmidt operators. We equip $\mathcal{L}_{1|2}(E)$ with the trace norm on the diagonals and Hilbert-Schmidt norm on the off-diagonals. That is
\begin{align*}
\|A\|_{1|2}\coloneqq \sum_{1\leq i\leq m } \|A_{(i)(i)}\|_1+\sum_{1\leq i,j\leq m, i\neq j} \|A_{(i)(j)}\|_{2}.
\end{align*} 
\cite[(A.1)]{BOO} introduced a generalization of the Fredholm determinant, still denoted as $\det$, such that  for $A\in\mathcal{L}_{1|2}(E)$ the following properties hold. Firstly, if $A$ is in the trace class, then $\det(1+A)$ agrees with the standard Fredholm determinant. Secondly,
\begin{equation}\label{equ:detconti}
 \det(1+A)\ \textup{  is continuous with respect to the}\  \mathcal{L}_{1|2}(E)\ \textup{  norm}.
\end{equation}
Thirdly, there exists a constant $C$ such that
 \begin{equation}\label{equ:detbound}
  |\det(1+A)|\leq \exp(C\|A\|^3_{1|2}). 
 \end{equation}

The first and the second properties are proved in \cite[Proposition A.1]{BOO}. The third one follows by combining \cite[Theorem 6.4]{Simon} and \cite[(A.1)]{BOO}. 

Let $K:\mathfrak{X}\times\mathfrak{X}\to\mathbb{R}$ be a function and let $E\subset \mathfrak{X}$ be a bounded Borel set. We write $\mathbbm{1}_E(x)$ for the indicator function of $E$. We identify $\mathbbm{1}_E(x)   K(x,y)\mathbbm{1}_E(y)$ with an operator on $L^2(E)$ defined by
\begin{align*}
f(y)\mapsto \int_E dy\, \mathbbm{1}_E(x) K(x,y) f(y).
\end{align*}
Moreover, similar to \eqref{equ:Asep}, we denote by $K_{(i)(j)}(x,y)$ for the restriction of $K(x,y)$ on $\mathbb{R}_+^{(i)}\times \mathbb{R}_+^{(j)}$. 
\begin{definition}\label{def:L12}
For a function $K:\mathfrak{X}\times\mathfrak{X}\to\mathbb{R}$, we say $K$ is of locally $\mathcal{L}_{1|2}$ if  $\mathbbm{1}_E(x)   K(x,y)\mathbbm{1}_E(y)$ belongs to $ \mathcal{L}_{1|2}(E)$ for all bounded Borel set $E\subset\mathfrak{X}$. 
\end{definition}
\begin{definition}\label{def:Kconti}
For a function $K:\mathfrak{X}\times\mathfrak{X}\to\mathbb{R}$, we say $K$ is continuous on the diagonal blocks if $K_{(i)(i)}(x,y)$ is continuous for all $i\in\llbracket 1,m \rrbracket$. 
\end{definition}

Definitions~\ref{def:L12} and \ref{def:Kconti} provide the regularity requirements of $K(x,y)$ in this paper. For this class of $K(x,y)$, the next proposition relates moment generating functions with the Fredholm determinant.

\begin{proposition}\label{pro:Kregularity}
Let $\mathcal{P}$ be a determinantal point process on $(\textup{Conf}(\mathfrak{X}),\mathcal{F}_{\textup{cyl}})$ with correlation kernel $K:\mathfrak{X}\times\mathfrak{X}\to\mathbb{R}$. Suppose that $K(x,y)$ is of locally $\mathcal{L}_{1|2}$ and is continuous on the diagonal blocks. Let $E_1,E_2,\dots, E_\ell$ be disjoint bounded Borel subsets of $\mathfrak{X}$ and $E=\sqcup_{j=1}^\ell E_j$. Then
\begin{align}\label{equ:moment}
\mathcal{E}\bigg[ \prod_{j=1}^\ell z_j^{M_{E_j}} \bigg]=\det\bigg(1+\mathbbm{1}_E\cdot \sum_{j=1}^\ell (z_j-1)\cdot K\cdot\mathbbm{1}_{E_j}\bigg).
\end{align}
\end{proposition}
\begin{proof}
Following the proof of \cite[Theorem 2]{Soshnikov}, the Taylor expansion of the left hand side of \eqref{equ:moment} at $z_1=z_2=\dots=z_\ell=1$ equals
\begin{align}\label{equ:momnet-mid}
1+\sum_{k=1}^\infty \sum_{k_1+k_2+\dots+k_\ell=k} \prod_{j=1}^\ell \frac{1}{k_j!}  \left( \int_{E_1^{k_1}\times E_2^{k_2}\times\dots\times E_\ell^{k_\ell}} d x_1 \dots d x_k\, \det[K(x_i,x_j)]_{1\leq i,j\leq k} \right) \prod_{j=1}^\ell (z_j-1)^{k_j}.
\end{align}
To simplify the notation, we write $D(\mathbf{z})$ for the right hand side of \eqref{equ:moment} and  $D'(\mathbf{z})$ for \eqref{equ:momnet-mid}. Here $\mathbf{z}=(z_1,\dots, z_\ell)$. We remark that for a moment $D'(\mathbf{z})$ is a formal power series.
\vspace{\baselineskip}

If $K(x,y)$ is of locally trace class and is continuous, then $D'(\mathbf{z})=D(\mathbf{z})$ \cite[Theorem 2]{Soshnikov}. In particular, $D'(\mathbf{z})$ and $D(\mathbf{z})$ are both entire functions. We proceed through approximation. Let $K_{(i)(j)}(x,y)$, $i\neq j$ be the off-diagonal terms of $K(x,y)$. For $\epsilon>0$ , let $K^{\epsilon}_{(i)(j)}(x,y)$ be a mollification of $K_{(i)(j)}(x,y)$ such that
\begin{enumerate}
\item $K^\epsilon_{(i)(j)} (x,y)$ is a smooth function on $\mathbb{R}_+\times\mathbb{R}_+$.
\item $K^\epsilon_{(i)(j)} (x,y)$ converges to $K_{(i)(j)} (x,y)$ locally in $L^2$.
\end{enumerate}
Let $K^\epsilon$ be obtained from $K$ by replacing $K_{(i)(j)}$ with $K^\epsilon_{(i)(j)}$ for all $i\neq j$. For consistency, we also let $K^{ \epsilon }_{(i)(i)}=K_{(i)(i)}$. We use $D_\epsilon(\mathbf{z})$ and $D_\epsilon'(\mathbf{z})$ to denote $D(\mathbf{z})$ and $D'(\mathbf{z})$ with $K$ replaced by $K^\epsilon$ respectively. For $i\neq j$, since $K^\epsilon_{(i)(j)}(x,y)$ is a smooth function, it is of locally trace class \cite[(10.18)]{GK}. This implies $K^{\epsilon}$ is of locally trace class and is continuous. From \cite[Theorem 2]{Soshnikov}, we have $D'_\epsilon (\mathbf{z})=D_\epsilon(\mathbf{z}).$
\vspace{\baselineskip}

For $i\in\llbracket 1,m \rrbracket $, $K^{\epsilon}_{(i)(i)}=K_{(i)(i)}$. For $i\neq j$ in $\llbracket 1,m \rrbracket$, $K^{\epsilon}_{(i)(j)}$ converges to $K_{(i)(j)}$ locally in the Hilbert-Schmidt norm. These imply $K^\epsilon (x,y)$ converges to $K(x,y)$ locally in $\mathcal{L}_{1|2}$. By the continuity property \eqref{equ:detconti}, $D_\epsilon(\mathbf{z})$ converges to $D(\mathbf{z})$ pointwisely. By \eqref{equ:detbound}, $D_\epsilon(\mathbf{z})$ are locally uniformly bounded in $\epsilon$. This implies $D(\mathbf{z})$ is also an entire function and $D_\epsilon(\mathbf{z})$ converges to $D(\mathbf{z})$ locally in $C^k$ for all $k\in\mathbb{N}$. To show that $D'(\mathbf{z})$ equals $D(\mathbf{z})$, it remains check that the coefficients in the Taylor expansions agree. That is, we need to show 
\begin{align*}
 \int_{E_1^{k_1}\times E_2^{k_2}\times\dots\times E_\ell^{k_\ell}} d x_1 \dots d x_k\, \det[K^\epsilon (x_i,x_j)]_{1\leq i,j\leq k}
\end{align*}
converges to
\begin{align*}
\int_{E_1^{k_1}\times E_2^{k_2}\times\dots\times E_\ell^{k_\ell}} d x_1 \dots d x_k\, \det[K(x_i,x_j)]_{1\leq i,j\leq k}.
\end{align*}

By expanding the determinant, 
\begin{align*}
 \det[K^\epsilon (x_i,x_j)]_{1\leq i,j\leq k} =\sum_{\sigma\in S_k} \textup{sgn}(\sigma) \sum_{i=1}^k K^\epsilon (x_i,x_{\sigma(i)}). 
\end{align*}
Here $S_k$ is the permutation group. For fixed $\sigma\in S_k$, let $I^\sigma_1\coloneqq\{i\in\llbracket 1,k \rrbracket\,|\, \sigma(i)=i\}$. Then
 \begin{align*}
 \prod_{i=1}^k K^\epsilon(x_i,x_{\sigma(i)})=\prod_{i\in I^\sigma_1} K^\epsilon(x_i,x_i) \prod_{i\notin I^\sigma_1} K^\epsilon(x_i,x_{\sigma(i)})
 =\prod_{i\in I^\sigma_1} K (x_i,x_i) \prod_{i\notin I^\sigma_1} K^\epsilon(x_i,x_{\sigma(i)}).
 \end{align*}
In the second equality, we used that $K=K^\epsilon$ on the diagonals. Let $(i_1i_2\dots i_q)$ be an arbitrary cycle of $\sigma$ with $q\geq 2$ and let $F_1,F_2,\dots, F_q$ be bounded Borel subsets of $\mathfrak{X}$. It suffices to show that
\begin{align*}
\int_{F_1\times F_2\dots \times F_q} dx_{i_1} dx_{i_2}\dots d x_{i_q}  K^\epsilon (x_{i_1},x_{i_2})K^\epsilon (x_{i_2},x_{i_3}) \dots K^\epsilon({x_{i_q},x_{i_1}})
\end{align*}
converges to
\begin{align*}
\int_{F_1\times F_2\dots \times F_q} dx_{i_1} dx_{i_2}\dots d x_{i_q}  K  (x_{i_1},x_{i_2})K  (x_{i_2},x_{i_3}) \dots K ({x_{i_q},x_{i_1}}).
\end{align*}
Because $K^\epsilon(x_{i_1},x_{i_2})$ converges to $K^\epsilon(x_{i_1},x_{i_2})$ locally in $L^2$, it remains to show that 
\begin{align*}
\int_{F_3\times F_4\dots \times F_q} dx_{i_3}\dots dx_{i_q}  K^\epsilon (x_{i_2},x_{i_3})K^\epsilon (x_{i_3},x_{i_4})\dots K^\epsilon({x_{i_q},x_{i_1}})
\end{align*}
converges to
\begin{align*}
\int_{F_3\times F_4\dots \times F_q} dx_{i_3}\dots dx_{i_q}  K (x_{i_2},x_{i_3})K (x_{i_3},x_{i_4})\dots K({x_{i_q},x_{i_1}})
\end{align*}
in $L^2(F_2\times F_1)$. Using Minkowski's inequality for integrals inductively on $q$, we have
\begin{multline*}
\left\| \int_{F_3\times F_4\dots \times F_q} dx_{i_3}\dots d_{i_q}  K (x_{i_2},x_{i_3})K (x_{i_3},x_{i_4})\dots K({x_{i_q},x_{i_1}}) \right\|_{L^2(F_2\times F_1)}\\
\leq \left\| K (x_{i_2},x_{i_3}) \right\|_{L^2(F_2\times F_3)}\left\| K (x_{i_3},x_{i_4}) \right\|_{L^2(F_3\times F_4)}\dots \left\| K (x_{i_q},x_{i_1}) \right\|_{L^2(F_q\times F_1)}. 
\end{multline*}
Therefore the assertion follows the local $L^2$ convergence of $K^\epsilon$.
\end{proof}

\begin{corollary}\label{cor:Fredet}
Under the same setting as in Proposition~\ref{pro:Kregularity}, we have
\begin{multline}\label{equ:Freddet}
   {n_1!n_2!\dots n_\ell !} \mathcal{P}(M_{E_1}=n_1,M_{E_2}=n_2,\dots,M_{E_\ell}=n_\ell )\\
= \frac{\partial^{n_1+n_2+\dots +n_\ell}}{\partial z_1^{n_1}\dots \partial z_\ell^{n_\ell}}\det\bigg(1+\mathbbm{1}_E\cdot \sum_{j=1}^\ell (z_j-1) K\cdot\mathbbm{1}_{E_j}\bigg)\bigg|_{z_1=z_2=\dots=z_\ell=0}.
\end{multline}
Moreover, the right hand side of \eqref{equ:Freddet} is continuous with respect to the locally $\mathcal{L}_{1|2}$ convergence. 
\end{corollary}
\begin{proof}
The equality \eqref{equ:Freddet} can be derived by taking derivatives in \eqref{equ:moment}. Suppose $K^N$, $N\in\mathbb{N}$ converges to $K$ locally in $\mathcal{L}_{1|2}$.  By the continuity property \eqref{equ:detconti}, 
\begin{equation*}
\det\bigg(1+\mathbbm{1}_E\cdot \sum_{j=1}^\ell (z_j-1) K^N\cdot\mathbbm{1}_{E_j}\bigg)
\xrightarrow[N\to\infty]{}
\det\bigg(1+\mathbbm{1}_E\cdot \sum_{j=1}^\ell (z_j-1) K\cdot\mathbbm{1}_{E_j}\bigg)
\end{equation*}
By \eqref{equ:detbound} , the convergence is locally in $C^k$. Then the assertion follows.
\end{proof}



We end this section with a corollary which will be used to prove Theorem~\ref{thm:main} (ii) and (iii). It is a direct consequence of Corollary~\ref{cor:Fredet}. 

\begin{corollary}\label{cor:det}
Let $\vec{X}^{N}=(\vec{X}^{(1),N},\dots,\vec{X}^{(m),N})$ be a sequence of random vectors with $\vec{X}^{(k),N}\in\mathbb{W}^N_+$ for all $k\in\llbracket 1,m \rrbracket$. Assume that $\vec{X}^N$ is determinantal with a correlation kernel $K^N(x,y)$ and that for all $\ell\in\mathbb{N}$, $\{X^{(k),N}_j \},\ (j,k)\in \llbracket 1,\ell \rrbracket\times \llbracket 1,m \rrbracket$ converges in distribution when $N$ goes to infinity. We write $\vec{X}=\{X^{(k)}_j \},\ (j,k)\in \mathbb{N} \times \llbracket 1,m \rrbracket$ for the limit of $\vec{X}^N$. Further suppose that there exists a function $K(x,y)$ such that $K^N(x,y)$ converges to $K(x,y)$ locally in $\mathcal{L}_{1|2}$ and that $K^N(x,y)$ and $K(x,y)$ are continuous on the diagonal blocks. Then $\vec{X}$ is determinantal with correlation kernel $K(x,y)$.
\end{corollary}

\subsection{Exponential Gibbs Property}\label{sec:Gibbs} 
In this section we introduce the exponential Gibbs property. We start by recalling the {interlacing} relation between two vectors. Let $k\in\mathbb{N}$. For two vectors $\vec{x}, \vec{y}\in \mathbb{W}^k$, we denote by $\vec{x}\prec\vec{y}$ if
\begin{equation*} 
x_1<y_1<x_2<\cdots<x_k<y_k.
\end{equation*}
Also, we use $\vec{x}\preceq\vec{y}$ to denote
\begin{equation*} 
 x_1\leq  y_1\leq  x_2\leq \cdots\leq  x_k\leq  y_k. 
\end{equation*}

Next, we define an interlacing set $I_{a,b}(\vec{x},\vec{y})$ and $I_{a,b}(\vec{x},\vec{y},\vec{z})$ with prescribed boundary $(\vec{x},\vec{y},\vec{z})$.

\begin{definition}\label{def:interlacingset}
Fix $k\in\mathbb{N}$ and $a<b$ with $a,b\in \mathbb{N}_0$. For any $\vec{x},\vec{y}\in\mathbb{W}^{k}$ we denote by $I_{a,b}(\vec{x},\vec{y})$ the collection of $\vec{w}^{ \ell }\in\mathbb{W}^k$, $\ell\in \llbracket a+1,b-1  \rrbracket $ which satisfies 
\begin{align*}
    \vec{x}  \preceq  \vec{w}^{ a+1 } \preceq  \vec{w}^{ a+2 }  \preceq\dots  \vec{w}^{ b-1 } \preceq   \vec{y}.  
\end{align*}
We denote by $Z_{a,b}(\vec{x},\vec{y})$ the Lebesgue measure of $I_{a,b}(\vec{x},\vec{y})$. For $\vec{z}=(z^{a},z^{a+1},\dots,z^{ b })\in \mathbb{R}^{b-a+1}$, we denote by $I_{a,b}(\vec{x},\vec{y},\vec{z})$ the collection of $\vec{w}^{\ell}\in\mathbb{W}^k$, $\ell\in \llbracket a+1,b-1  \rrbracket $ which satisfies 
\begin{align*}
   (\vec{x},z^{ a })  \preceq (\vec{w}^{ a+ 1 },z^{ a+1 })\preceq (\vec{w}^{ a+2 },z^{ a+2 })  \preceq\cdots \preceq (\vec{w}^{ b-1 },z^{ b-1 }) \preceq ( \vec{y}, z^{ b } ).  
\end{align*}
We denote by $Z_{a,b}(\vec{x},\vec{y},\vec{z})$ the Lebesgue measure of $I_{a,b}(\vec{x},\vec{y},\vec{z})$. Similarly, 
\end{definition}

In other words, $I_{a,b}(\vec{x},\vec{y})$ consists of interlacing sequences from $ \vec{x} $ to $\vec{y}$.  $I_{a,b}(\vec{x},\vec{y},\vec{z})$ is defined similarly by adding the components in $\vec{z}$ on the top of $\vec{x}$, $\vec{y}$ and $\left\{  \vec{w}^{\ell}\right\}_{\ell=a+1}^{b-1}$.  

\begin{lemma}\label{lem:boundaryZ}
Fix $k\in\mathbb{N}$ and $a<b$ with $a,b\in \mathbb{N}_0$. We follow the notation as in Definition~\ref{def:interlacingset}. Then the sets
\begin{align*}
\{ (\vec{x},\vec{y} )\, |\, I_{a,b}(\vec{x},\vec{y} )\neq \emptyset\ \textup{and}\ Z_{a,b}(\vec{x},\vec{y} )= 0\}, 
\end{align*}
and
\begin{align*}
\{ (\vec{x},\vec{y},\vec{z})\, |\, I_{a,b}(\vec{x},\vec{y},\vec{z})\neq \emptyset\ \textup{and}\ Z_{a,b}(\vec{x},\vec{y},\vec{z})= 0\},
\end{align*}
as subsets in an Euclidean space, have Lebesgue measure zero. 
\end{lemma}
\begin{proof}
Consider the set
\begin{align*}
S\coloneqq \{ (\vec{x},\vec{y},\vec{z},\vec{w})\, |\, \vec{w}\in I_{a,b}(\vec{x},\vec{y},\vec{z})\}.
\end{align*}
By the definition of interlacing \eqref{def:interlacingweak}, $S$ is a closed convex body with non-empty interior $\textup{int}(S)$. Let $p(\vec{x},\vec{y},\vec{z},\vec{w})=(\vec{x},\vec{y},\vec{z})$ be the projection. It is easy to check that $(\vec{x},\vec{y},\vec{z},\vec{w})\in \textup{int}(S)$ is equivalent to strict interlacing \eqref{def:interlacing}. This implies that $Z_{a,b}(\vec{x},\vec{y},\vec{z})>0$ provided $(\vec{x},\vec{y},\vec{z})\in p(\textup{int}(S))$. Also $I_{a,b}(\vec{x},\vec{y},\vec{z})\neq \emptyset$ is equivalent to $(\vec{x},\vec{y},\vec{z})\in p(S)$. As a result,
\begin{align*}
\{ (\vec{x},\vec{y},\vec{z})\, |\, I_{a,b}(\vec{x},\vec{y},\vec{z})\neq \emptyset\ \textup{and}\ Z_{a,b}(\vec{x},\vec{y},\vec{z})= 0\}\subset p(S)\setminus p(\textup{int}(S)).
\end{align*}
Because $S$ and $p(S)$ are both closed convex bodies with non-empty interior, $ p(S)\setminus p(\textup{int}(S))= p(S)\setminus \textup{int}(p(S))$ has Lebesgue measure zero. The proof for $\{ (\vec{x},\vec{y} )\, |\, I_{a,b}(\vec{x},\vec{y} )\neq \emptyset\ \textup{and}\ Z_{a,b}(\vec{x},\vec{y} )= 0\}$ is similar. 
\end{proof}

Now we define an interlacing exponential bridge ensemble with prescribed boundary data.
\begin{definition}\label{def:PPP}
Fix $k\in\mathbb{N}$ and $a<b$ with in $a,b\in \mathbb{N}_0$. Let $(\vec{x},\vec{y},\vec{z})$ be as in Definition~\ref{def:interlacingset}. Assume that $Z_{a,b}(\vec{x},\vec{y} )>0$. Then we denote by $\mathbb{P}^{k,\llbracket a+1,b-1 \rrbracket,\vec{x},\vec{y}}$ the probability measure on $I_{a,b}(\vec{x},\vec{y} )$ which is proportional to the Lebesgue measure. That is, for any Borel set $E\subset I_{a,b}(\vec{x},\vec{y} )$,
\begin{equation*}
\mathbb{P}^{k,\llbracket a+1,b-1 \rrbracket,\vec{x},\vec{y} }(E)=\frac{\textup{Lebesgue measure of}\ E}{Z_{a,b}(\vec{x},\vec{y} )}.  
\end{equation*}
We call $\mathbb{P}^{k,\llbracket a+1,b-1 \rrbracket,\vec{x},\vec{y} }$ the law of the interlacing exponential ensemble with entrance data $\vec{x}$ and exit data $\vec{y}$.

Similarly, assume that $Z_{a,b}(\vec{x},\vec{y},\vec{z})>0$. Then we denote by $\mathbb{P}^{k,\llbracket a+1,b-1 \rrbracket,\vec{x},\vec{y},\vec{z}}$ the probability measure on $I_{a,b}(\vec{x},\vec{y},\vec{z})$ which is proportional to the Lebesgue measure. That is, for any Borel set $E\subset I_{a,b}(\vec{x},\vec{y},\vec{z})$,
\begin{equation*}
\mathbb{P}^{k,\llbracket a+1,b-1 \rrbracket,\vec{x},\vec{y},\vec{z}}(E)=\frac{\textup{Lebesgue measure of}\ E}{Z_{a,b}(\vec{x},\vec{y},\vec{z})}.  
\end{equation*}
We call $\mathbb{P}^{k,\llbracket a+1,b-1 \rrbracket,\vec{x},\vec{y},\vec{z}}$ the law of the interlacing exponential ensemble with entrance data $\vec{x}$, exit data $\vec{y}$ and top boundary data $\vec{z}$.
\end{definition}
\begin{remark}
The terminology \textbf{exponential} bridge ensemble comes from the following observation. If we sample $k$ independent exponential bridges with entrance data $\vec{x}$ and exit data $\vec{y}$ and condition on the event that bridges lie in $I_{a,b}(\vec{x},\vec{y},\vec{z})$, then the resulting law is the uniform distribution over all possible weakly interlacing sequences connecting $\vec{x}$ and $\vec{y}$. This is exactly the law of $\mathbb{P}^{k,\llbracket a+1,b-1 \rrbracket,\vec{x},\vec{y},\vec{z}}$.
\end{remark}

We define the exponential Gibbs property below. The exponential Gibbs property could be viewed as a spatial Markov property. More specifically, it provides a description of the conditional law inside a compact set.
\begin{definition}[Exponential Gibbs property]\label{def:ExpoGP}
Fix $N\in\mathbb{N}\cup\{\infty\}$. Let $\vec{X}=\big( \vec{X} (a)\in \mathbb{W}^N_+,\ a\in\mathbb{N}_0\big)$ be a sequence of random vectors. We view $\vec{X}$ as a random function defined on $ \llbracket 1,N \rrbracket\times \mathbb{N}_0$. $\vec{X}$ is said to satisfy the exponential Gibbs property if the following condition holds. For all $k<N$ and $a<b$ with $k\in\mathbb{N}$ and $a,b\in\mathbb{N}_0$, let $\vec{x}=( {X}_1 {(a)}, {X}_2 {(a)} ,\dots,{X}_k {(a)} )$, $\vec{y}=( {X}_1 {(b)} , {X}_2 {(b)} ,\dots,{X}_k {(b)} )$ and $\vec{z}=( {X}_{k+1} {(a)} ,{X}_{k+1} {(a+1)} ,\dots,{X}_{k+1} {(b)})$. Then with probability one,
\begin{equation}\label{equ:Z>0}
Z_{a,b}(\vec{x},\vec{y},\vec{z})>0.
\end{equation}
Moreover, the conditional law of $\vec{X}$ inside $\llbracket 1,k \rrbracket \times \llbracket a+1,b-1\rrbracket $ takes the following form,
\begin{equation}\label{Gibbs-condition}
\textrm{Law}\left(\vec{X}\left\vert_{\llbracket 1,k \rrbracket \times \llbracket a+1,b-1\rrbracket}\  \textrm{conditional on } \vec{X}\right\vert_{\llbracket 1,N \rrbracket\times \mathbb{N}_0  \setminus \llbracket 1,k \rrbracket \times \llbracket a+1,b-1\rrbracket } \right) =\bP^{k,\llbracket a+1,b-1\rrbracket,\vec{x},\vec{y},\vec{z}}
\end{equation}

If $N\in\mathbb{N}$, we further make the following requirement. Let $\vec{x}=( {X}_1 {(a)}, {X}_2 {(a)} ,\dots,{X}_N {(a)} )$, $\vec{y}=( {X}_1 {(b)} , {X}_2 {(b)} ,\dots,{X}_N {(b)} )$. Then with probability one, 
\begin{equation}\label{equ:ZN>0}
Z_{a,b}(\vec{x},\vec{y})>0.
\end{equation}
Moreover, the conditional law of $\vec{X}$ inside $\llbracket 1,N \rrbracket \times \llbracket a+1,b-1\rrbracket $ takes the following form,
\begin{equation}\label{Gibbs-condition-N}
\textrm{Law}\left(\vec{X}\left\vert_{\llbracket 1,N \rrbracket \times \llbracket a+1,b-1\rrbracket}\  \textrm{conditional on } \mathcal{L}\right\vert_{\llbracket 1,N \rrbracket\times (\mathbb{N}_0  \setminus  \llbracket a+1,b-1\rrbracket) } \right) =\bP^{k,\llbracket a+1,b-1\rrbracket,\vec{x},\vec{y}}
\end{equation}
\end{definition}
\begin{remark}\label{rmk:ExpoGP}
If $N\in\mathbb{N}$, then it is straightforward to check that \eqref{equ:ZN>0} implies \eqref{equ:Z>0} and that \eqref{Gibbs-condition-N} implies \eqref{Gibbs-condition}.  
\end{remark}

\section{Time-Like Sequences}\label{sec:3}

In this section, we study the the Laguerre ensemble along a time-like path. Recall that for $ N\in\mathbb{N}$, $\alpha\in\mathbb{N}_0$ and $t>0$, the Laguerre ensemble $\vec{X}^N(\alpha,t)=\left(X^N_1(\alpha,t),X^N_2(\alpha,t),\dots ,X^N_N(\alpha,t)\right)$, defined in \eqref{def:X}, are the ordered eigenvalues of the random Hermitian matrix. The density of $\vec{X}^N(\alpha,t)$ is given by \cite{For93} 
\begin{align}\label{onept_density}
C(N,\alpha) t^{-N^2} \prod_{j=1}^N (x_j/t)^{\alpha}e^{-x_j /t} \Delta(\vec{x})^2  \mathbbm{1}(\vec{x}\in \mathbb{W}_+^N)  \prod_{j=1}^N dx_j.
\end{align}
The constant $C(N,\alpha)$ is  
\begin{align}\label{equ:CNalpha}
C(N,\alpha)\coloneqq  \prod_{j=1}^N \frac{1}{\Gamma(\alpha+j)\Gamma(j)}.
\end{align}
We remark that the density in \cite{For93} is given for $t=1$. $\eqref{onept_density}$ can then be derived by the scaling
\begin{align*}
\left(X^N_1(\alpha,t),X^N_2(\alpha,t),\dots ,X^N_N(\alpha,t)\right)\overset{(d)}{=\joinrel =}t\cdot \left(X^N_1(\alpha,1),X^N_2(\alpha,1),\dots ,X^N_N(\alpha,1)\right).
\end{align*}

The structure of this section is as follows. We compute the transition probability and the joint density of $\vec{X}^N(\alpha,t)$ along a time-like path in Section~\ref{sec:2.1}. In Section~\ref{sec:2.2}, we show that the process is determinantal and calculate the correlation kernel using the Eynard--Mehta Theorem. In Section~\ref{sec:2.3}, we compute the limit of the correlation kernel under the hard edge scaling. In particular, we give the proof of  Theorem~\ref{thm:main} (i).

\subsection{Transition Density}\label{sec:2.1}

The goal of this section is to calculate the joint density of Wishart ensemble along a time-like path. We start with the Markov property of $\vec{X}^N(\alpha,t)$ along a time-like path. The proof is postponed to Section~\ref{sec:A}.

\begin{lemma}\label{lem:Markov}
Fix $N\in\mathbb{N}$, $m\in\mathbb{N}$ and let $\left\{\left(\alpha^{(k)},t^{(k)}\right)\right\}_{k=1}^m\subset \mathbb{N}_0\times(0,\infty)$ be a time-like path as in Definition~\ref{def:timelike}. Let 
\begin{align}\label{def:Xk-timelike}
\vec{X}^{(k)}\coloneqq \left\{X^{N}_j\left(\alpha^{(k)},t^{(k)}\right)\right\}_{j=1}^N.
\end{align}  
Then $\vec{X}^{(k)}$ is a Markov process in $k$.
\end{lemma}

 There are two types of extreme time-like paths. In one type $\alpha$ is fixed and only $t$ increases. In another type $t$ is fixed and only $\alpha$ increases. We compute the transition probability for these two special types below.
 \vspace{\baselineskip}

For $\alpha<\beta$ with $\alpha,\beta \in\mathbb{N}_0$, $ t<s$ with $t,s\in (0,\infty) $ and $x,y\in (0,\infty)$, define the functions
\begin{align}
T_{\alpha}((t,x);(s,y))&\coloneqq (s-t)^{-1}(y/x)^{\alpha/2}e^{-(x+y)/(s-t)} I_\alpha \left( \frac{2\sqrt{xy} }{s-t} \right),\label{def:T}\\
W_t((\alpha,x);(\beta,y) )&\coloneqq \frac{1}{\Gamma(\beta- \alpha)}\frac{(y-x)^{\beta- \alpha-1}}{t^{ \beta- \alpha }}e^{-(y-x)/t} \mathbbm{1}(x\leq y).\label{def:W}
\end{align}
Here $I_\alpha$ is the modified Bessel function of the first kind.
 
\begin{lemma}\label{lem:timechange}
Fix $N\in\mathbb{N}$,  $t<s$ with $t,s\in ( 0,\infty)$ and $\alpha\in  \mathbb{N}_0$. For any $\vec{x}\in \mathbb{W}^N_+$, the transition probability from $\vec{X}^{N}(\alpha, t)= \vec{x} $ to $\vec{X}^{N }(\alpha,s)= \vec{y}$ is given by  
\begin{align}\label{timechange}
\det \left[T_{\alpha}( (t,x_i);(s,y_j))\right]_{1\leq i,j\leq N}\frac{\Delta(\vec{y}) } {\Delta(\vec{x})} \mathbbm{1}(\vec{y}\in\mathbb{W}^N_+) \prod_{j=1}^N dy_j .
\end{align} 
Here $\Delta(\vec{y})$ and $\Delta(\vec{x})$ are Vandermonde determinants defined in \eqref{def:Vandermonde}.
\end{lemma}
\begin{proof}
For fixed $N\in\mathbb{N}$ and $\alpha\in\mathbb{N}_0$, $\{\vec{X}^{N }(\alpha,t),t\in\mathbb{R} \}$ can be identified as non-intersecting squared Bessel processes \cite{KO}. The transition density can then be computed through the Karlin-McGregor formula \cite{KM}. See \cite[(1.6)]{KT} for more details.
\end{proof}

\begin{lemma}\label{lem:alphachange}
Fix $N\in\mathbb{N}$,  $\alpha<\beta$ with $\alpha,\beta \in \mathbb{N}_0 $ and $t\in  (0,\infty)$. For any $\vec{x}\in \mathbb{W}^N_+$, the transition probability from $\vec{X}^{N }(\alpha,t)= \vec{x} $ to $\vec{X}^{ N  }(\beta,t)= \vec{y}$ is given by  
\begin{align}\label{alphachange}
\det \left[W_{t}((\alpha,x_i);(\beta,y_j) )\right]_{1\leq i,j\leq N}\frac{\Delta(\vec{y})}{\Delta(\vec{x})} \mathbbm{1}(\vec{y}\in\mathbb{W}^N_+) \prod_{j=1}^N dy_j .
\end{align}
Here $\Delta(\vec{y})$ and $\Delta(\vec{x})$ are Vandermonde determinants defined in \eqref{def:Vandermonde}.
\end{lemma}
\begin{proof}
We first develop the case $\beta = \alpha + 1$. By the Cauchy interlacing theorem and the ordering of eigenvalues, the transition probability is supported on the set $\{ \vec{x} \preceq \vec{y},\  \vec{y} \in\mathbb{W}^N_+\}$. By unitary invariance of the model, we can assume that $A^N(\alpha,t)$ and $A^N(\alpha+1,t)$ are of the form
\[
A^N(\alpha,t)=\begin{pmatrix}
	\mathrm{diag}\left((\sqrt{x_i})_{1\leqslant i\leqslant N}\right)\\
	0_{\alpha \times N}
\end{pmatrix},
\quad A^N(\alpha+1,t)=\begin{pmatrix}
	\mathrm{diag}\left((\sqrt{x_i})_{1\leqslant i\leqslant N}\right)\\
	0_{\alpha \times N}\\
	\vec{a}(t)^*
\end{pmatrix}
\]
where $\vec{a}(t)$ is the column vector whose entries $a_i(t)$ are distributed as i.i.d. complex Brownian motions at time $t$. See Section~\ref{sec:A} for more details about the unitary invariance. In particular, we obtain that the entries of the matrix $\left(A^N(\alpha+1,t)\right)^*A^N(\alpha+1,t)$ are given by, for $i,j\in\{1,\dots,N\}$,
\begin{equation}\label{eq:entries}
	\left[\left(A^N(\alpha+1,t)\right)^*A^N(\alpha+1,t)\right]_{ij}
	=
	\left\{
	\begin{array}{ll}
		x_i + \vert a_i(t)\vert^2 &\text{if }i=j,\\
		\overline{a_i}(t)a_j(t) & \text{otherwise}.
	\end{array}
	\right.
\end{equation}
Since we suppose that the eigenvalues of $\left(A^N(\alpha+1,t)\right)^*A^N(\alpha+1,t)$ are given by $\vec{y}$, note that we have
\begin{equation}\label{eq:coeffpoly}
	P^N_{\alpha,t}(\lambda )\coloneqq\det\left(\lambda\mathrm{Id}-\left(A^N(\alpha+1,t)\right)^*A^N(\alpha+1,t)\right)
	=
	\sum_{k=0}^N (-1)^k e_k(\vec{y}) \lambda^{n-k},
\end{equation}
where $e_k$ is the $k$-th elementary symmetric polynomial. We first perform a change of variables from $\vec{y}$ to $(e_1(\vec{y}),\dots,e_N(\vec{y}))$ which gives a Jacobian factor of 
\begin{equation}\label{eq:jacobiansym}
	\left\vert \frac{\partial(e_1(\vec{y}),\dots, e_N(\vec{y}))}{\partial(y_1,\dots,y_k)}\right\vert = \Delta(\vec{y}). 
\end{equation}
Next, we perform a change of variables from $(e_1(\vec{y}),\dots e_N(\vec{y}))$ to $(\vert a_1(t)\vert^2,\dots, \vert a_N(t)\vert^2)$. To compute the Jacobian factor, first see that 
\begin{align*}
	P^N_{\alpha, t}(\lambda)=\det\left(\mathrm{diag}((\lambda-x_i)_{1\leq i\leq N})-\vec{a}(t)\vec{a}(t)^*\right)=
	\left(1- \sum_{i=1}^N\frac{|a_i(t)|^2}{\lambda-x_i}\right)\prod_{i=1}^N(\lambda-x_i),
\end{align*}
where we used the matrix determinant lemma in the second equality. Thus we obtain that 
\[
P_{\alpha, t}^N(\lambda)=\prod_{i=1}^N(\lambda-x_i)-\sum_{j=1}^N\vert a_j(t)\vert^2\prod_{i\neq j}(\lambda-x_i).
\]
In particular, by identifying the coefficients from \eqref{eq:coeffpoly} we can see that
\[
\frac{\partial e_i(\vec{y})}{\partial \vert a_j\vert^2}=(-1)^ie_{i-1}(\hat{x}_j)=(-1)^i\frac{\partial e_i(\vec{x})}{\partial x_j}
\]
where $\hat{x}_j$ is the $N-1$ dimensional vector $(x_1,\dots,x_{j-1},x_{j+1},\dots,x_N)$ so that, with the same identity \eqref{eq:jacobiansym}, we obtain
\[
\left \vert\frac{\partial(\vert a_1\vert^2,\dots,\vert a_N\vert^2)}{\partial (e_1(\vec{y}),\dots, e_N(\vec{y}))}\right\vert
=
\left\vert \frac{\partial (e_1(\vec{y}),\dots, e_N(\vec{y}))}{\partial (\vert a_1\vert^2,\dots,\vert a_N\vert^2)}\right\vert^{-1}=\frac{1}{\Delta(\vec{x})}.
\]
Now, since $a_i(t)$'s are complex Gaussian random variables, their modulus is exponentially distributed and we have that the joint density of $(\vert a_1(t)\vert^2,\dots,\vert a_N(t)\vert^2)$ is given by
\[
\prod_{i=1}^N \frac{1}{t}e^{-z_i/t}=\frac{1}{t^N}e^{-\sum_{i=1}^N\vert a_i(t)\vert^2/t}=\frac{1}{t^N}e^{-\sum_{i=1}^N (y_i-x_i)/t}
\]
where we used the fact that $\mathrm{Tr}((A^N(\alpha+1,t))^*A^N(\alpha+1,t))=\sum_{i=1}^N(x_i+\vert a_i(t)\vert^2)=\sum_{i=1}^N y_i$ from \eqref{eq:entries}. 

Finally we obtain that the transition probability from $\vec{X}^{N}(\alpha,t)=\vec{x}$ to $\vec{X}^{N}( \alpha+1 ,t)=\vec{y}$ is given by
\begin{multline*}
	\frac{1}{t^N}e^{-\sum_{i=1}^N(y_i-x_i)/t}\mathbbm{1}(\vec{x} \preceq \vec{y})\frac{\Delta(\vec{y})}{\Delta(\vec{x})}\mathbbm{1}(y\in \mathbb{W}^N_+)\prod_{j=1}^N dy_j\\
	=
	\det[W_t((\alpha,x_i);(\alpha+1,y_j))]_{1\leq i,j\leq N}\frac{\Delta(\vec{y})}{\Delta(\vec{x})}\mathbbm{1}(y\in \mathbb{W}^N_+)\prod_{j=1}^N dy_j
\end{multline*}
where we used Sasamoto's trick given by Lemma \ref{lem:sasamoto} to write the interlacing property as a determinant. Thus the result is proved for $\beta=\alpha+1$. 

The general case can be obtained by using the generalized Cauchy--Binet formula. Consider $\beta > \alpha$, if we denote this transition probability as $p(X^N(\beta,t)=\vec{y}\,\vert X^N(\alpha,t)=\vec{x})$ and if we set $\vec{z}^{\,(0)}=\vec{x}$ and $\vec{z}^{\,(\beta-\alpha)}=\vec{y}$, we can write
\begin{align*}
	&p(X^N(\beta,t)=\vec{y}\,\vert X^N(\alpha,t)=\vec{x})\\
	&=
	\int_{\mathbb{W}^N_+}\dots\int_{\mathbb{W}^N_+}
	\prod_{k=0}^{\beta-\alpha-1}
	p(X^N(\alpha+k+1,t)=\vec{z}^{\,(k+1)}\vert X^N(\alpha+k,t)=\vec{z}^{\,(k)})
	\prod_{k=1}^{\beta-\alpha-1}
	 {d^N\vec{z}^{\,(k)}} \\
	&=\frac{\Delta(\vec{y})}{\Delta(\vec{x})}\mathbbm{1}(\vec{y}\in W_+^N)
	\int_{\mathbb{R}^N}\dots\int_{\mathbb{R}^N}
	\prod_{k=0}^{\beta-\alpha-1}
	\det[W_t((\alpha+k,{z}^{(k)}_i);(\alpha+k+1,{z}^{(k+1)}_j))]\prod_{k=1}^{\beta-\alpha-1}
	\frac{d^N\vec{z}^{\,(k)}}{N!}
\end{align*}
Now, if we first consider the integration with respect to $\vec{z}^{\,(1)}$, we have that 
\begin{align*}
	\frac{1}{N!}\int_{\mathbb{R}^N}
	&\det[W_t((\alpha,x_i);(\alpha+1,{z}^{(1)}_j))]_{1\leq i,j\leq N}
	\det[W_t((\alpha+1,{z}^{(1)}_j);(\alpha+2,{z}^{(2)}_i))]_{1\leq i,j\leq N}
	d^N \vec{z}^{\,(1)}\\
	&=
	\det\left[
	\int_\mathbb{R} 
	W_t((\alpha,x_i);(\alpha+1,{z}))
	W_t((\alpha+1,z);(\alpha+2,{z}^{(2)}_j))
	dz
	\right]_{1\leq i,j\leq N}\\
	&=
	\det\left[
	\frac{1}{t^2}e^{-(z_j^{(2)}-x_i)}
	\int_{\mathbb{R}}
	\mathbbm{1}(x_i\leq z)
	\mathbbm{1}(z\leq z_j^{(2)})
	\right]_{1\leq i,j\leq N}\\
	&=
	\det\left[
	\frac{(z_j^{(2)}-x_i)}{t^2}e^{-(z_j^{(2)}-x_i)}\mathbbm{1}(x_i\leq z_j^{(2)})
	\right]_{1\leq i,j\leq N}.
\end{align*}
By iterating the application of the generalized Cauchy--Binet formula we obtain the final result.
\end{proof}

Next, we consider the transition probability for a general time-like path in which both $\alpha$ and $t$ can increase. From the Cauchy--Binet formula, it suffices to compute the transition density for $N=1$. By setting $N=1$, $T_\alpha(x,y;s-t)dy$ is the transition density from $X^{1 }(\alpha ,t)=x$ to $X^{1}( \alpha ,s)=y$. Similarly, $W_t((\alpha,x);(\beta,y) )dy$ is the transition density from $X^{1}(\alpha ,t)=x$ to $X^{1}(\beta,t)=y$. Given $(\alpha,t)\prec_t (\beta,s)$ as in Definition~\ref{def:timelike}, let 
$$Q((\alpha,t,x);(\beta,s,y))dy$$ 
be the transition density from $X^{1}( \alpha,t)=x$ to $X^{1}(\beta,s)=y$. Because of the Markov property along time-like paths, for all $(\alpha,t)\prec_{\textup{t}} (\beta,s)\prec_{\textup{t}} (\gamma,\tau)$ , we have
\begin{align}\label{equ:QQ}
\int_{0}^\infty dy\, Q  ((\alpha,t,x);(\beta,s,y))Q((\beta,s,y);(\gamma,\tau,z))=Q((\alpha,t,x);(\gamma,\tau,z)).
\end{align}
Moreover, $Q((\alpha,t,x);(\beta,s,y))$ takes the following form:
\begin{equation}\label{equ:Qform}
Q((\alpha,t,x);(\beta,s,y))=\left\{ \begin{array}{cc}
T_\alpha((t,x);(s,y)), &\textup{if}\ \alpha=\beta\ \textup{and}\ t<s,\\
W_t((\alpha,x);(\beta,y) ), &\textup{if}\ \alpha<\beta\ \textup{and}\ t=s,\\
\int_0^\infty dz\, T_\alpha((t,x);(s,z))W_s((\alpha,z);(\beta,y) ), &\textup{if}\ \alpha<\beta\ \textup{and}\ t<s.
\end{array} \right.
\end{equation}
We remark that \eqref{equ:QQ} and \eqref{equ:Qform} implies a commutation relation between $T_\alpha((t,x);(s,y))$ and $W_t((\alpha,x);(\beta,y) )$. For any $\alpha<\beta\in\mathbb{N}_0$ and $t<s\in (0,\infty)$, it holds that
\begin{align}\label{equ:commm}
\int_0^\infty dz\, T_\alpha((t,x);(s,z) )W_s((\alpha,z);(\beta,y) )= \int_0^\infty dz\, W_t((\alpha,x);(\beta,z))T_\beta((t,z);(s,y)). 
\end{align}
An integral representation of $Q((\alpha,t,x);(\beta,s,y))$ is given at the end of this section. We now use $Q((\alpha,t,x);(\beta,s,y))$ to express the transition probability from $\vec{X}^{(k)}$ to $\vec{X}^{(k+1)}$. 
\begin{corollary}\label{cor:transition}
Fix $N\in\mathbb{N}$ and $m\in\mathbb{N}$. Let $\left\{\left(\alpha^{(k)},t^{(k)}\right)\right\}_{k=1}^m\subset\mathbb{N}_0\times (0,\infty)$  be a time-like path as in Definition~\ref{def:timelike}. Let $\vec{X}^{(k)}$ be defined as in \eqref{def:Xk-timelike}. Then for any $k\in \llbracket1,m-1\rrbracket $ and $\vec{x}^{(k)}\in \mathbb{W}^N_+$, the transition density from $\vec{X}^{(k)}=\vec{x}^{(k)}$ to $\vec{X}^{(k+1)}=\vec{x}^{(k+1)}$ is given by
\begin{align}\label{equ:transition}
\det\bigg[ Q((\alpha^{(k)},t^{(k)},x^{(k)}_i );(\alpha^{(k+1)},t^{(k+1)},x^{(k+1)}_j )) \bigg]_{1\leq i,j\leq N} \frac{\Delta(\vec{x}^{(k+1)})}{\Delta(\vec{x}^{(k)})}\mathbbm{1}(\vec{x}^{(k+1)}\in\mathbb{W}^N_+) \prod_{j=1}^N dx^{(k+1)}_j.
\end{align}
\end{corollary}
\begin{proof}
If $\alpha^{(k)}=\alpha^{(k+1)}$ and $t^{(k)}<t^{(k+1)}$, from \eqref{equ:Qform},  \eqref{equ:transition} is equivalent to \eqref{timechange}. If $\alpha^{(k)}<\alpha^{(k+1)}$ and $t^{(k)}=t^{(k+1)}$, from \eqref{equ:Qform}, \eqref{equ:transition} is equivalent to \eqref{alphachange}. Using \eqref{equ:QQ} and the Cauchy–Binet formula, \eqref{equ:transition} for general $(\alpha^{(k)},t^{(k)})\prec_{\textup{t}} (\alpha^{(k+1)},t^{(k+1)})$ follows.
\end{proof}

Combining \eqref{onept_density} and Corollary~\ref{cor:transition}, we obtain the joint density of $\vec{X}^{(k)}$.  
\begin{corollary}\label{cor:density}
Fix $N\in\mathbb{N}$ and $m\in\mathbb{N}$. Let $\left\{\left(\alpha^{(k)},t^{(k)}\right)\right\}_{k=1}^m\subset\mathbb{N}_0\times (0,\infty)$  be a time-like path as in Definition~\ref{def:timelike}. Then the joint density of $\vec{X}^{(k)}$ defined in \eqref{def:Xk-timelike} is given by
\begin{equation}\label{density}
\begin{split}
C(N,\alpha^{(1)}) &\left(t^{(1)}\right)^{-N^2}\prod_{j=1}^N  (x^{(1)}_j/t^{(1)})^{\alpha^{(1)}}e^{-x^{(1)}_j /t^{(1)}}\times  \Delta(\vec{x}^{(1)})\\
&\times \prod_{k=1}^{m-1} \det\bigg[ Q((\alpha^{(k)},t^{(k)},x^{(k)}_i );(\alpha^{(k+1)},t^{(k+1)},x^{(k+1)}_j )) \bigg]_{1\leq i,j\leq N}\\
&\times \Delta(\vec{x}^{(m)}) \prod_{k=1}^m \mathbbm{1}(\vec{x}^{(k)}\in\mathbb{W}^N_+)  \prod_{k=1}^m\prod_{j=1}^N dx^{(k)}_j.
\end{split}
\end{equation}

\end{corollary}

We end this section by giving an integral representation of $Q((\alpha,t,x);(\beta,s,y))$. The proof can be found in Appendix \ref{sec:B}.  
\begin{lemma}\label{lem:Qintegral}
Let $(\alpha,t)\prec_{\textup{t}} (\beta,s)$ be time-like pairs in $ \mathbb{N}_0\times (0,\infty)$ as in Definition \ref{def:timelike} and $x,y\in (0,\infty)$. It holds that
\begin{equation}\label{equ:Qintegral}
\begin{split}
Q((\alpha,t,x);(\beta,s,y))= &(st)^{-(\beta-\alpha)/2} e^{-y/s+x/t} x^{-\alpha/2}y^{ \beta/2}\\
&\times\int_0^\infty du\,  e^{-(s-t)u}u^{-(\beta-\alpha)/2}J_{\alpha}(2\sqrt{t^{-1}s xu})J_{\beta}(2\sqrt{s^{-1}t yu}).
\end{split}
\end{equation} 
Here $J_\alpha $ is the Bessel function of the first kind.
\end{lemma} 

\subsection{Correlation Kernel}\label{sec:2.2}
In this section, we show that $\vec{X}^{(k)}$ defined in \eqref{def:Xk-timelike} is determinantal and calculate its correlation kernel.
\vspace{\baselineskip}
 
We begin with introducing the generalized Laguerre polynomials which are used to rewrite \eqref{density}. For $j\in\mathbb{N}_0$ and $\alpha\in\mathbb{N}_0$, the generalized Laguerre polynomial with degree $j$ and parameter $\alpha$ is defined by
\begin{align}\label{def:Laguerre}
L^\alpha_{j}(x)\coloneqq \frac{x^{-\alpha}e^x}{\Gamma(j+1)}\frac{d^{j }}{dx^{j }}(x^{\alpha+j}e^{-x}).
\end{align}
Note that the following orthogonal relation holds.
\begin{align}\label{Lagorthogonal}
\int_0^\infty dx\, x^\alpha e^{-x} L^\alpha_i(x)L^\alpha_j(x)=\frac{\Gamma(\alpha+j+1)}{\Gamma(j+1)}\delta_{ij}.
\end{align}

For $j\in\mathbb{N}$, $\alpha\in\mathbb{N}_0$ and $t>0$, define the functions
\begin{align}
\phi_j(\alpha,t,x)&\coloneqq \frac{\Gamma(j)}{\Gamma(\alpha+j)}t^{-j}(x/t)^{\alpha }e^{-x/t}L^\alpha_{j-1}(t^{-1}x),\label{def:phi} \\
\psi_j(\alpha,t,x)&\coloneqq t^{j-1} L^\alpha_{j-1}(t^{-1}x).\label{def:psi}
\end{align}
Because the leading coefficient in $L^{\alpha}_{j-1}(x)$ is $\frac{(-1)^{j-1}}{\Gamma(j)}$, we have
\begin{align*}
&\det\left[\phi_i(\alpha ,t , x_j)\right]_{1\leq i,j\leq N}=(-1)^{N(N-1)/2}t^{-N^2}\prod_{j=1}^N \frac{1}{\Gamma(\alpha+j)}\times\prod_{j=1}^N (x_j/t)^{\alpha}e^{-x_j /t} \times\Delta(\vec{x}),\\
&\det\left[\psi_i(\alpha ,t , x_j)\right]_{1\leq i,j\leq N}=(-1)^{N(N-1)/2} \prod_{j=1}^N \frac{1}{\Gamma(j)}   \times\Delta(\vec{x}).
\end{align*} 
Therefore, the joint density \eqref{density} can be expressed as
\begin{equation}\label{density2}
\begin{split}
&\det \left[\phi_i(\alpha^{(1)},t^{(1)}, x^{(1)}_j)\right]_{1\leq i,j\leq N}\\
\times &\prod_{k=1}^{m-1} \det\bigg[ Q((\alpha^{(k)},t^{(k)},x^{(k)}_i );(\alpha^{(k+1)},t^{(k+1)},x^{(k+1)}_j )) \bigg]_{1\leq i,j\leq N}\\
\times &\det \left[\psi_j(\alpha^{(m)},t^{(m)}, x^{(m)}_i)\right]_{1\leq i,j\leq N}\times \prod_{k=1}^m \mathbbm{1}(\vec{x}^{(k)}\in\mathbb{W}^N_+)  \prod_{k=1}^m\prod_{j=1}^N dx^{(k)}_j.
\end{split}
\end{equation}

We now show that $\phi_j(\alpha,t,x)$ and $ \psi_j(\alpha,t,x)$ are compatible with ${Q}((\alpha,t,x);(\beta,s,y))$.

\begin{lemma}\label{lem:phiQpsi}
Let $(\alpha,t)\prec_{\textup{t}} (\beta,s)$ be time-like pairs in $ \mathbb{N}_0\times (0,\infty)$ as in Definition \ref{def:timelike}. Then the following statements hold. 
\begin{align}\label{delta}
\int_{0}^\infty dx\, \phi_i(\alpha,t,x) \psi_j(\alpha,t,x)  =\delta_{ij}.
\end{align}
\begin{align}\label{phiQ}
\int_{0}^\infty dx\, \phi_j(\alpha,t,x) {Q}((\alpha,t,x);(\beta,s,y)) =\phi_j(\beta,s,y).
\end{align}
\begin{align}\label{Qpsi}
\int_{0}^\infty dy\, {Q}((\alpha,t,x);(\beta,s,y))\psi_j(\beta,s,y) =\psi_j(\alpha,t,x).
\end{align}
\end{lemma}	

As a direct corollary, we show that $\vec{X}^{(k)}$, $k\in\llbracket 1,m \rrbracket$ is determinantal.

\begin{corollary}\label{cor:kernel}
Fix $N\in\mathbb{N}$ and $m\in\mathbb{N}$. Let $\left\{\left(\alpha^{(k)},t^{(k)}\right)\right\}_{k=1}^m\subset \mathbb{N}_0\times (0,\infty)$  be a time-like path as in Definition~\ref{def:timelike}. Let $\vec{X}^{(k)}$ be defined as in \eqref{def:Xk-timelike}. Then  $ \{\vec{X}^{(k)} \}_{k=1}^m$ is a determinantal point process. Furthermore, the correlation kernel is given by $K^N((\alpha^{(k)},t^{(k)},x);(\alpha^{(\ell)},t^{(\ell)},y))$, where
\begin{equation}\label{equ:kernel}
K^N((\alpha,t,x);(\beta,s,y))=- {Q}((\alpha,t,x);(\beta,s,y))\mathbbm{1}((\alpha,t)\prec_{\textup{t}} (\beta,s) )+\sum_{j=1}^N \psi_j(\alpha,t,x)\phi_j(\beta,s,y).
\end{equation} 

\end{corollary}
\begin{proof}
The assertion follows by combining \eqref{density2}, Lemma \ref{lem:phiQpsi} and the Eynard--Mehta theorem \cite{EM,BR}. 
\end{proof}
We end this section by proving Lemma~\ref{lem:phiQpsi}.
\begin{proof}[Proof of Lemma~\ref{lem:phiQpsi}]
\eqref{delta} is a reformulation of the orthogonality of generalized Laguerre polynomials \eqref{Lagorthogonal}.
\vspace{\baselineskip}

 We turn to \eqref{phiQ}. In view of \eqref{equ:QQ} and \eqref{equ:Qform}, it suffices to show 
\begin{align}
\int_{0}^\infty \phi_j(\alpha,t,x) T_{\alpha}((t,x);(s,y))\, dx=\phi_j(\alpha,s,y),\label{phiT}
\end{align}
and
\begin{align}
\int_{0}^\infty \phi_j(\alpha,t,x) W_{t}((\alpha,x);(\beta,y))\, dx=\phi_j(\beta,t,y).\label{phiW}
\end{align}
Actually, the equality \eqref{phiT} is proved in \cite[Lemma 3.4(ii)]{FF} and we focus on showing \eqref{phiW}.

By a direct calculation,
\begin{align*}
\int_{0}^\infty dx\, \phi_j(\alpha,t,x) W_{t}((\alpha,x)&;(\beta,y))
\\
&=\frac{\Gamma(j)}{\Gamma(\alpha+j)\Gamma(\beta-\alpha)}t^{-j-\beta}e^{-y/t}\int_0^y dx\,  x^{\alpha}(y-x)^{\beta-\alpha-1}L^{\alpha}_{j-1}(x/t)  \\
&=\frac{\Gamma(j)}{\Gamma(\alpha+j)\Gamma(\beta-\alpha)}t^{-j-\beta} e^{-y/t} y^{\beta} \int_0^1 dz\, z^{\alpha}(1-z)^{\beta-\alpha-1}L^{\alpha}_{j-1}(yz/t) .
\end{align*}
From \cite[(7.412-1),page 809]{GR}
\begin{align*}
\int_0^1  dz\, z^{\alpha}(1-z)^{\beta-\alpha-1}L^{\alpha}_{j-1}(yz/t) =\frac{\Gamma(\alpha+j)\Gamma(\beta-\alpha)}{\Gamma(\beta+j)} L^{\beta}_{j-1}(y/t).
\end{align*}
Therefore,
\begin{align*}
\int_{0}^\infty dx\, \phi_j(\alpha,t,x) W_{t}((\alpha,x);(\beta,y)) =&\frac{\Gamma(j)}{\Gamma(\beta+j)} t^{-j}(y/t)^\beta e^{-y/t}  L^{\beta}_{j-1}(y/t)=\phi_j(\beta,t,y).
\end{align*}
This finishes the proof of \eqref{phiQ}.
\vspace{\baselineskip}

Next, we prove \eqref{Qpsi}. For $j\in\mathbb{N}$,  define 
\begin{align*}
\tilde{\psi}_j(x)\coloneqq \int_{0}^\infty  dy\, {Q}((\alpha,t,x);(\beta,s,y))\psi_j(\beta,s,y) .
\end{align*}
We aim to show that $\tilde{\psi}_j(x)=\psi_j(\alpha,t,x).$ By the Cauchy--Binet formula and \eqref{density2}, we have
\begin{align*}
\det [\tilde{\psi}_j(x_i)]_{1\leq i,j\leq N}= (-1)^{N(N-1)/2} \prod_{j=1}^N \frac{1}{\Gamma(j)}   \times\Delta(\vec{x}).
\end{align*}
Therefore, $\tilde{\psi}_j(x)$ is a polynomial of degree $j-1$. Also, from \eqref{phiQ} and \eqref{delta} we have
\begin{align*}
\int_{0}^\infty  dx\, \phi_i(\alpha,t,x)\tilde{\psi}_j(x) =\int_{0}^\infty  dx\, \phi_i(\beta,s,x){\psi}_j(\beta,s,x)=\delta_{ij}. 
\end{align*}
As a result, $\tilde{\psi}_j(x)=\psi_j(\alpha,t,x)$.
\end{proof}

\subsection{Scaling Limit}\label{sec:2.3}
In this section, we perform the hard edge scaling and show that when $N$ goes to infinity, the scaled correlation kernel converges to $K^{\textup{Bes}}$ defined in \eqref{def:KBes}. To be precise, we fix a time-like path $\{ (\alpha^{(k)},t^{(k)}) \}$ in $\mathbb{N}_0\times \mathbb{R}$. For $N>-4^{-1}t^{(1)}$, Define
\begin{equation}
\mathcal{B}_j^{N,(k)}\coloneqq 4N\cdot X_j^{N}\left(\alpha^{(k)}, 1+ {t^{(k)}}/{4N} \right),\ j\in\llbracket 1,N \rrbracket. 
\end{equation}  
Because of Corollary \ref{cor:kernel}, $\mathcal{B}_j^{N,(k)}$ is determinantal with the correlation kernel $$K^N_{\textup{scaled}}((\alpha^{(k)},t^{(k)},x);(\alpha^{(\ell)},t^{(\ell)},y)).$$ 
Here $K^N_{\textup{scaled}}((\alpha ,t ,x);(\beta ,s,y))$ is given by
\begin{align*}
K^N_{\textup{scaled}}((\alpha ,t ,x);(\beta ,s,y))\coloneqq (4N)^{-1}K^N\left(  ( \alpha ,1+ {t }/{4N}, {x}{/4N} );( \beta,1+ {s}/{4N}, {y}{/4N} )\right).
\end{align*}    
Define the gauged kernel as
\begin{equation}\label{def:Kgauge}
\begin{split}
K^N_{\textup{gauge}}((\alpha,t,x)&;(\beta,s,y))
\\&\coloneqq (4N)^{(\beta-\alpha)/2-1}\left( \frac{x}{4N+t} \right)^{\alpha/2}\left( \frac{y}{4N+s} \right)^{-\beta/2}\exp\left( -\frac{x}{8N+2t}+\frac{y}{8N+2s} \right) \\
&\times  K^N\left(  ( \alpha ,1+ t/{4N}, {x}{/4N} );( \beta,1+ t/{4N}, {s}{/4N} )\right).
\end{split}
\end{equation}
We remark that $K^N_{\textup{gauge}}$ and $K^N_{\textup{scaled}}$ are related through a gauge transform $$K^N_{\textup{gauge}}((\alpha,t,x);(\beta,s,y))=  \frac{f^N(\alpha,t,x)}{f^N(\beta,s,y)}\cdot  K^N_{\textup{scaled}}((\alpha,t,x);(\beta,s,y)). $$
This implies $K^N_{\textup{gauge}}((\alpha^{(k)},t^{(k)},x);(\alpha^{(\ell)},t^{(\ell)},y))$ is also a correlation kernel for $\mathcal{B}_j^{N,(k)}$.
\vspace{\baselineskip}

We now show the convergence of the kernel $K^N_{\textup{gauge}}((\alpha^{(k)},t^{(k)},x);(\alpha^{(\ell)},t^{(\ell)},y))$ to the kernel $K^{\textup{Bes}}((\alpha^{(k)},t^{(k)},x);(\alpha^{(\ell)},t^{(\ell)},y))$. To simplify the notation, we define the following quantities.
\begin{align*}
K^{N,1}((\alpha,t,x)&;(\beta,s,y))
\\&\coloneqq -(4N)^{  (\beta-\alpha) /{2}-1}\left( \frac{x}{4N+t} \right)^{\alpha/2}\left( \frac{y}{4N+s} \right)^{-\beta/2} \exp\left( -\frac{x}{8N+2t}+\frac{y}{8N+2s} \right) \\
 &\times \mathbbm{1}((\alpha,t)\prec_{\textup{t}} (\beta,s) )  \times {Q}\left(\left(\alpha,1+  {t}/{4N},  {x}/{4N}\right);\left(\beta,1+  {s}/{4N}, {y}/{4N}\right)\right),
\end{align*}
\begin{align*}
K^{N,2}((\alpha,t,x)&;(\beta,s,y))
\\&\coloneqq (4N)^{ {(\beta-\alpha)}/{2}-1}\left( \frac{x}{4N+t} \right)^{\alpha/2} \left( \frac{y}{4N+s} \right)^{-\beta/2} \exp\left( -\frac{x}{8N+2t}+\frac{y}{8N+2s} \right)\\
&\times \sum_{j=1}^N \psi_j\left(\alpha,1+ {t}/{4N}, {x}/{4N}\right)\cdot\phi_j\left(\beta,1+ {s}/{4N}, {y}/{4N}\right),
\end{align*}
and
\begin{align*}
K^{N,1}_{k\ell}(x,y)&\coloneqq  K^{N,1}((\alpha^{(k)},t^{(k)},x);(\alpha^{(\ell)},t^{(\ell)},y)),\\ K^{N,2}_{k\ell}(x,y)&\coloneqq  K^{N,2}((\alpha^{(k)},t^{(k)},x);(\alpha^{(\ell)},t^{(\ell)},y)). 
\end{align*}
Similarly, let
\begin{align*}
K^{1}((\alpha,t,x);(\beta,s,y))&\coloneqq -\int_0^\infty du\,  e^{-(s-t)u}u^{-(\beta-\alpha)/2}J_{\alpha}\left(2\sqrt{xu}\right)J_{\beta}\left(2\sqrt{yu}\right)  \mathbbm{1}((\alpha,t)\prec_{\textup{t}} (\beta,s) ) ,\\
K^{2}((\alpha,t,x);(\beta,s,y))&\coloneqq  \int_0^{1/4} du\,  e^{-(s-t)u}u^{-(\beta-\alpha)/2}J_{\alpha}\left(2\sqrt{xu}\right)J_{\beta}\left(2\sqrt{yu}\right)   
\end{align*}
and
\begin{align*}
K^{1}_{k\ell}(x,y)&\coloneqq  K^{1}((\alpha^{(k)},t^{(k)},x);(\alpha^{(\ell)},t^{(\ell)},y)),\\ K^{2}_{k\ell}(x,y)&\coloneqq  K^{2}((\alpha^{(k)},t^{(k)},x);(\alpha^{(\ell)},t^{(\ell)},y)). 
\end{align*}
Clearly
\begin{align*}
K^N_{\textup{gauge}}((\alpha^{(k)},t^{(k)},x);(\alpha^{(\ell)},t^{(\ell)},y))&=K^{N,1}_{k\ell}(x,y)+K^{N,2}_{k\ell}(x,y),\\
K^{\textup{Bes}}((\alpha^{(k)},t^{(k)},x);(\alpha^{(\ell)},t^{(\ell)},y))&=K^{1}_{k\ell}(x,y)+E^{ 2}_{k\ell}(x,y).
\end{align*}
We view $\{ K^{N,1}_{k\ell} \}_{1\leq k,\ell\leq m}$ as a function defined on $\mathfrak{X}\times \mathfrak{X}$ with $\mathfrak{X}=\llbracket 1,m \rrbracket\times\mathbb{R}_+$.  
\begin{lemma}\label{lem:E1}
$\{ K^{N,1}_{k\ell} \}_{1\leq k,\ell\leq m}$ converges to $\{ K^{1}_{k\ell} \}_{1\leq k,\ell\leq m}$ locally in $\mathcal{L}_{1|2}$ defined in Definition \ref{def:L12}.  
\end{lemma}
\begin{lemma}\label{lem:E2}
$\{ K^{N,2}_{k\ell} \}_{1\leq k,\ell\leq m}$ converges to $\{ K^{2}_{k\ell} \}_{1\leq k,\ell\leq m}$ locally in $\mathcal{L}_{1|2}$ defined in Definition \ref{def:L12}. Moreover, $\{ K^{N,2}_{k\ell} \}_{1\leq k,\ell\leq m}$ and $\{ K^{2}_{k\ell} \}_{1\leq k,\ell\leq m}$ are continuous on the diagonal blocks. See Definition~\ref{def:Kconti}.
\end{lemma}
With Lemma~\ref{lem:E1} and Lemma~\ref{lem:E2}, we are ready to prove Theorem~\ref{thm:main} (ii).
\begin{proof}[Proof of Theorem~\ref{thm:main} (i)]
From the definition of $\mathcal{B}(\alpha,t)$, a subsequence of $\mathcal{B}^N(\alpha^{(k)},t^{(k)})$ converges in distribution to $\mathcal{B}(\alpha^{(k)},t^{(k)})$. We abuse the notation and still denote such subsequence by $N$. From Lemma~\ref{lem:E1} and Lemma~\ref{lem:E2}, correlation kernels of $\mathcal{B}^N(\alpha^{(k)},t^{(k)})$ converges locally in $\mathcal{L}_{1|2}$ to a kernel given by $K^{\textup{Bes}}$. Moreover, all the kernels are continuous on the diagonal blocks. Then the assertion follows Corollary~\ref{cor:det}.
\end{proof}
The rest of the section is devoted to prove Lemma~\ref{lem:E1} and Lemma~\ref{lem:E2}.
\begin{proof}[Proof of Lemma~\ref{lem:E1}]
If $k\geq \ell$, then both $K^{N,1}_{k\ell}(x,y)$ and $K^{1}_{k\ell}(x,y)$ vanish. Therefore, it suffices to prove that $K^{N,1}_{k\ell}(x,y)$ converges to $K^{1}_{k\ell}(x,y)$ locally in $L^2$ for all $k<\ell$. See Definition~\ref{def:L12}. To simplify the notation, we set $(\alpha^{(k)},t^{(k)})=(\alpha,t)$ and $(\alpha^{(\ell)},t^{(\ell)})=(\beta,s)$.   Using the integral representation \eqref{equ:Qintegral} and a change of variable, $${Q}\left(\left(\alpha,1+  {t}/{4N},  {x}/{4N}\right);\left(\beta,1+  {s}/{4N}, {y}/{4N}\right)\right)$$ equals
\begin{multline*}
(4N)^{-(\beta-\alpha)/2+1}[(1+t/4N)(1+s/4N)]^{-(\beta-\alpha)/2}\exp\left( -\frac{y}{4N+s}+\frac{x}{4N+t} \right)\\
\times \left( \frac{x}{4N} \right)^{-\alpha/2}\left( \frac{y}{4N} \right)^{ \beta/2}\int_0^\infty du\, e^{-(s-t)u}u^{-(\beta-\alpha)/2}J_{\alpha}\left(2\sqrt{\frac{4N+s}{4N+t}xu }\right)J_{\beta}\left(2\sqrt{\frac{4N+t}{4N+s}yu }\right).
\end{multline*}
Therefore, $E_{k\ell}^{N,1}(x,y)$ equals
\begin{multline*}
-(1+t/4N)^{-\beta/2}(1+s/4N)^{\alpha/2} \exp\left( -\frac{y}{8N+2s}+\frac{x}{8N+2t} \right)\\
 \times \int_0^\infty du\, e^{-(s-t)u}u^{-(\beta-\alpha)/2}J_{\alpha}\left(2\sqrt{\frac{4N+s}{4N+t}xu }\right)J_{\beta}\left(2\sqrt{\frac{4N+t}{4N+s}yu }\right).
\end{multline*}
Denote $x_N=\frac{4N+s}{4N+t}\cdot x $ and $y_N=\frac{4N+s}{4N+t}\cdot y $ . Then
\begin{align*}
K^{N,1}_{k\ell}(x,y)= -&(1+t/4N)^{-\beta/2}(1+s/4N)^{\alpha/2} \exp\left( -\frac{y}{8N+2s}+\frac{x}{8N+2t} \right) K_{k\ell}^{1}(x_N,y_N). 
\end{align*}
We now consider two cases. In the first case, we assume $t=s$. In particular, we have $x_N=x$, $y_N=y$ and $\beta\geq \alpha+1$. Moreover, from Lemma~\ref{lem:W},
\begin{align*}
K^{1}_{k\ell}(x,y)=x^{\alpha/2} y^{-\beta/2} \frac{(y-x)^{\beta-\alpha-1}}{\Gamma(\beta-\alpha)}\times \mathbbm{1}(x\leq y).
\end{align*}
It is straightforward to check that $K^{1}_{k\ell}(x,y)$ is locally a $L^2$ function and that $K_{k\ell}^{N,1}(x,y)$ converges to $K_{k\ell}^{1}(x,y)$ locally in $L^2$. In the second case, we assume $t<s$. From the asymptotics of the Bessel function $J_\alpha(z)$ \eqref{equ:asymptoticzero} and \eqref{equ:asymptoticinfinity}, it is straightforward to check that $K_{k\ell}^{1}(x,y)$ is a continuous function on $[0,\infty)^2$. Therefore, $K_{k\ell}^{N,1}(x,y)$ converges to $K_{k\ell}^{1}(x,y)$ locally uniformly. This suffices to imply $K_{k\ell}^{N,1}(x,y)$ converges to $K_{k\ell}^{1}(x,y)$ locally in $L^2$. 
\end{proof}

\begin{remark}\label{rmk:nontrace}
It can be checked that the kernel $K_0(x,y)=x^{\alpha/2}y^{-(\alpha+1)/2}\mathbbm{1}(x\leq y)$ is not of local trace class. Using the isometry on $L^2([0,1])$ given by $f(x)\mapsto 2^{1/2}x^{1/2}f(x^2)$, $K_0(x,y)$ is similar to the integral kernel $K_0'(x,y)=2 x^{\alpha+1/2}y^{-\alpha-1/2}\mathbbm{1}(x\leq y)$. From \cite[Theorem 1.3]{BD}, $\mu>0$ is a singular values of $\mathbbm{1}(0\leq  x\leq 1)\cdot K_0'(x,y)\cdot \mathbbm{1}(0\leq  y\leq 1)$ if and only if $J_{\alpha}(2/\mu)=0$. From \eqref{equ:asymptoticinfinity}, the $n$-th singular value decays like $1/n$, which is not summable.  
\end{remark}

\begin{proof}[Proof of Lemma~\ref{lem:E2}]
First, we show it is sufficient to prove that $K_{k\ell}^{N,2}(x,y)$ converges to $K_{k\ell}^{2}(x,y)$ locally uniformly for all $k,\ell\in\llbracket 1,m \rrbracket$. From \eqref{equ:asymptoticzero}, $K_{k\ell}^{2}(x,y)$ is a continuous function on $[0,\infty)^2$. Therefore, locally uniform convergence implies the locally convergence in $L^2$. In view of Definition~\ref{def:L12}, the off-diagonals are settled. For the diagonal terms, we need to show $K_{\ell\ell}^{2,N}(x,y)$ converges to $K_{\ell\ell}^{2}(x,y)$ locally in trace norm. From \cite[Theorem 3.9]{Simonbook}, the trace of $K_{\ell\ell}^{N,2}(x,y)$ on $L^2([0,M])$ is given by
$\int_0^M dx\,  K_{\ell\ell}^{N,2}(x,x)$.

Because $K_{\ell\ell}^{N,2}(x,y)$ is symmetric and non-negative (see \eqref{equ:E2N}), the above equals the trace norm of $K_{\ell\ell}^{N,2}(x,y)$ on $L^2([0,M])$. The same holds for $K_{\ell\ell}^{2}(x,y)$. Therefore, the trace norm of $K_{\ell\ell}^{N,2}(x,y)$ on $L^2([0,M])$ convergence to the one of $K_{\ell\ell}^{2}(x,y)$. Convergence in Hilbert-Schmidt norm implies the convergence in operator norm. From \cite[Theorem 2.20]{Simonbook}, these imply the convergence in trace norm.
\vspace{\baselineskip}

It remain to show locally uniform convergence. Fix $k,\ell\in\llbracket 1,m \rrbracket$ and let $(\alpha^{(k)},t^{(k)})=(\alpha,t)$ and $(\alpha^{(\ell)},t^{(\ell)})=(\beta,s)$. Using the definitions of $\psi_j$ and $\phi_j$ in \eqref{def:phi} and \eqref{def:psi}, 
\begin{equation}\label{equ:E2N}
\begin{split}
K^{N,2}_{k\ell}(x,y) =&(4N)^{  \frac{\beta-\alpha}{2}-1}\left( \frac{x}{4N+t} \right)^{\alpha/2}\left( \frac{y}{4N+s} \right)^{\beta/2}\left(1+\frac{t}{4N}\right)^{-1} \exp\left( -\frac{y}{8N+2s}-\frac{x}{8N+2t} \right) \\
&\times \sum_{j=1}^N \frac{\Gamma(j)}{\Gamma(\beta+j)} \left(\frac{1+t/4N}{1+s/4N}\right)^{j}  \cdot L^\alpha_{j-1}\bigg(\frac{x}{4N+t}\bigg)\cdot L^\beta_{j-1}\left(\frac{y}{4N+s}\right).
\end{split}
\end{equation}
From now on we always assume $j\in \llbracket 1,N \rrbracket$ in the proof. Moreover, we assume $x,y\in [0,M]$ for some $M>0$ and denote by $O(1)$ a quantity which can be bounded by a constant depending on $\alpha,\beta,s,t$ and $M$. We proceed by applying the asymptotic limit of Laguerre polynomial for $j$ large. For $z>0$, define
\begin{align}\label{def:galpha}
g_{\alpha}(z)=z^{-\alpha/2}J_\alpha(2\sqrt{z}).
\end{align}
$g_\alpha(z)$ can be extended to an entire function. See \eqref{equ:gexpansion}. Using $g_\alpha(z)$ and $g_\beta(z)$, we have the following asymptotics for Laguerre polynomials.
\begin{equation}\label{equ:Lasymptotic}
\begin{split}
L^\alpha_{j-1}\bigg(\frac{x}{4N+t}\bigg)=j^\alpha g_\alpha\left( \frac{jx}{4N} \right)+O(j^{\alpha-1}),\\
L^\beta_{j-1}\bigg(\frac{y}{4N+s}\bigg)=j^\beta g_\beta\left( \frac{jy}{4N} \right)+O(j^{\beta-1}). 
\end{split}
\end{equation}
We postpone the proof of \eqref{equ:Lasymptotic} and continue the argument. Combining \eqref{equ:Lasymptotic} with 
\begin{equation*}
	\begin{gathered}
4N+t=4N(1+O(N^{-1})),\quad 4N+s=4N(1+O(N^{-1})),\\
 \exp\left( -\frac{y}{8N+2s}-\frac{x}{8N+2t} \right)=1+O(N^{-1}),\\ \frac{\Gamma(j)}{\Gamma(\beta+j)}=j^{-\beta}(1+O(j^{-1})),
 \quad\text{and}\quad \left(\frac{1+t/4N}{1+s/4N}\right)^{j} =\exp\left( -\frac{(s-t)j}{4N} \right)(1+O(N^{-1})),  
	\end{gathered}
\end{equation*}
we derive
\begin{multline*}
K^{N,2}_{k\ell}(x,y)=\frac{1}{4N} x^{\alpha/2}y^{\beta/2}\sum_{j=1}^N\exp \left( -\frac{(s-t)j}{4N} \right)\left( \frac{j}{4N} \right)^\alpha \times \left[ g_\alpha\left( \frac{jx}{4N} \right)+O(j^{-1}) \right]\\
\times \left[ g_\beta\left( \frac{jy}{4N} \right)+O(j^{-1}) \right]\times (1+O(j^{-1}))   
\end{multline*}
Because $g_\alpha (jx/4N)$ and $g_\beta(jy/4N)$ are bounded, it holds that
 \begin{multline*}
K^{N,2}_{k\ell}(x,y)-\frac{1}{4N} x^{\alpha/2}y^{\beta/2}\sum_{j=1}^N\exp \left( -\frac{(s-t)j}{4N} \right)\left( \frac{j}{4N} \right)^\alpha  g_\alpha\left( \frac{jx}{4N} \right) g_\beta\left( \frac{jy}{4N} \right)
\\= x^{\alpha/2}y^{\beta/2}O(N^{-1}\log N). 
\end{multline*}
Since the function $u\mapsto e^{-(s-t)u}u^\alpha g_\alpha(xu)g_\beta(yu) $ has bounded derivatives on $u\in [0,4^{-1}]$, we have
\begin{multline*}
\frac{1}{4N} \sum_{j=1}^N\exp \left( -\frac{(s-t)j}{4N} \right)\left( \frac{j}{4N} \right)^\alpha  g_\alpha\left( \frac{jx}{4N} \right) g_\beta\left( \frac{jy}{4N} \right)\\
=O(N^{-1})+\int_0^{1/4} du\,e^{-(s-t)u}u^\alpha g_\alpha(xu)g_\beta(yu).
\end{multline*}
As a result,
\begin{align*}
K^{N,2}_{k\ell}(x,y)=&x^{\alpha/2}y^{\beta/2} O(N^{-1}\log N)+x^{\alpha/2}y^{\beta/2} \int_0^{1/4} du\,e^{-(s-t)u}u^\alpha g_\alpha(xu)g_\beta(yu)\\
=&x^{\alpha/2}y^{\beta/2} O(N^{-1}\log N)+\int_0^{1/4} du\,e^{-(s-t)u}u^{-(\beta-\alpha)/2} J_\alpha(2\sqrt{xu}) J_\beta(2\sqrt{yu}).
\end{align*}
We used \eqref{def:galpha} in the second equality. This proves locally uniform convergence in $x$ and $y$.
\vspace{\baselineskip}

Lastly, we prove \eqref{equ:Lasymptotic}.  Let $ {j}' = j+2^{-1}(\alpha-1) $. \cite[(8.22.4),page 199]{S} reads
\begin{align*} 
 L^\alpha_{j-1}(z)\sim \frac{\Gamma(\alpha+j)}{\Gamma(j)} \cdot  e^{z/2} \cdot (z  {j}')^{-\alpha/2} \cdot J_\alpha\left(2\sqrt{ j'   z }\right).
\end{align*}
More precisely, from \cite[(8.22.4)]{S},
\begin{align*}
\left| L^\alpha_{j-1}(z)- \frac{\Gamma(\alpha+j)}{\Gamma(j)} \cdot  e^{z/2} \cdot (z  {j}')^{-\alpha/2} \cdot J_\alpha\left(2\sqrt{ j'   z }\right)\right|\leq Ce^{z/2}z^{2} j^{\alpha} 
\end{align*}
provided $z\leq cj^{-1}$. Here $C$ is a constant depending on $c>0$. Using the function $g_\alpha(z)$ defined in \eqref{def:galpha}, we have
\begin{align*}
\left| L^\alpha_{j-1}(z)- \frac{\Gamma(\alpha+j)}{\Gamma(j)}  \cdot e^{z/2} \cdot   g_\alpha\left(    j' z) \right)\right|\leq Ce^{z/2}z^{2} j^{\alpha}. 
\end{align*}
Then \eqref{equ:Lasymptotic} follows by plugging $z=x/(4N+t)$ and using the smoothness of $g_\alpha(z)$. The proof is finished.  
\end{proof}

\section{Space-like Curves}\label{sec:4}

We consider now the eigenvalues of the Laguerre ensemble along a space-like path. The structure is similar to the one in Section~\ref{sec:3}. We start by computing in Section~\ref{sec:3.1} the transition probability and the joint density of $\vec{X}^N(\alpha,t)$. We prove that the process is determinantal and give its correlation kernel in In Section~\ref{sec:3.2}. Finally, in section~\ref{sec:3.3}, we compute the limit, under the hard edge scaling, of the correlation kernel which gives a proof of  Theorem~\ref{thm:main} (ii).

\subsection{Transition Density}\label{sec:3.1}

The goal of this section is to calculate the joint density of Laguerre ensemble along a space-like path. The following lemma gives the Markov property of $\vec{X}^N(\alpha,t)$ now considered along a time-like path. The proof is postponed to Section~\ref{sec:A}.

\begin{lemma}\label{lem:Markov-spacelike}
Fix $N\in\mathbb{N}$, $m\in\mathbb{N}$ and let $\left\{\left(\alpha^{(k)},t^{(k)}\right)\right\}_{k=1}^m$  be a space-like path as in Definition~\ref{def:spacelike}. Let \begin{align}\label{def:Xk-spacelike}
\vec{X}^{(k)}\coloneqq \left\{X^{N }_j\left(\alpha^{(k)},t^{(k)}\right)\right\}_{j=1}^N.
\end{align}   
Then $\vec{X}^{(k)}$ is a Markov process in $k$.
\end{lemma}

There are again two types of extreme space-like paths. Either $t$ is fixed and only $\alpha$ decreases or $\alpha$ is fixed and only $t$ increases. We first compute the transition probability in these two cases.
\vspace{\baselineskip}

For $\alpha>\beta$ with $\alpha,\beta \in\mathbb{N}_0$ and $t>0$, define
\begin{align}\label{def:W-s}
\overline{W}_t((\alpha,x);(\beta,y) ):=\frac{x^{-\alpha}y^{\beta}(x-y)^{\alpha-\beta-1}}{\Gamma(\alpha-\beta)}\mathbbm{1}(y\leq x).
\end{align}
Note that $\overline{W}_t((\alpha,x);(\beta,y) )$ does not depend on $t$ but we still use this notation to be consistent with $ {W}_t((\alpha,x);(\beta,y) )$ defined in \eqref{def:W}.
 
\begin{lemma}\label{lem:alphachange-s}
Fix $N\in\mathbb{N}$,  $\alpha>\beta$ with $\alpha,\beta \in \mathbb{N}_0 $ and $t\in  (0,\infty)$. For any $\vec{x}\in \mathbb{W}^N_+ $, the transition probability from $\vec{X}^{N}(\alpha,t)= \vec{x} $ to $\vec{X}^{ N }(\beta,t)= \vec{y}$ is given by  
\begin{align}\label{alphachange-s}
\frac{C(N,\beta)}{C(N,\alpha)} \det \left[\overline{W}_t((\alpha,x_i);(\beta,y_j) )\right]_{1\leq i,j\leq N}\frac{\Delta(\vec{y})}{\Delta(\vec{x})}\mathbbm{1}(\vec{y}\in\mathbb{W}^N_+) \prod_{j=1}^N dy_j  .
\end{align}
 $C(N,\alpha)$ and $C(N,\beta)$ are constants defined in \eqref{equ:CNalpha}.
\end{lemma}
\begin{proof}
From \eqref{onept_density} and Lemma~\ref{lem:alphachange}, the joint density of $\vec{X}^{N,\alpha}(t)= \vec{x} $ and $\vec{X}^{(N,\beta)}(t)= \vec{y}$ is
\begin{align*}
C(N,\beta) t^{-N^2-\alpha N}\prod_{j=1}^N y_j^\beta e^{-x_j/t} \det\left[ \frac{(x_i-y_j)^{\alpha-\beta-1}}{\Gamma(\alpha-\beta)} \right]_{1\leq i,j\leq N} \hspace{-.6em}\Delta (\vec{x})\Delta(\vec{y})  \mathbbm{1}(\vec{x}\in \mathbb{W}_+^N,\vec{y}\in \mathbb{W}_+^N) \prod_{j=1}^N dx_j dy_j.
\end{align*}
Also, the density of $\vec{X}^{N,\alpha}(t)= \vec{x} $ is given by
\begin{align*}
C(N,\alpha) t^{-N^2-\alpha N}\prod_{j=1}^N x_j^{\alpha} e^{-x_j/t} \Delta(\vec{x})^2 \mathbbm{1}(\vec{x}\in \mathbb{W}_+^N) \prod_{j=1}^N dx_j.
\end{align*}
We then derive \eqref{alphachange-s} by comparing the two.
\end{proof}

Note that for $\alpha\in\mathbb{N}_0$ and $t<s$, $(\alpha,t)\prec_{\textup{s}} (\alpha,s)$ and $(\alpha,t)\prec_{\textup{t}} (\alpha,s)$ both hold. In particular, the transition probability from $\vec{X}^{N}(\alpha,t)= \vec{x} $ to $\vec{X}^{N }(\alpha,s)= \vec{y}$ is computed in Lemma~\ref{lem:timechange}. For consistency, we still write
\begin{align}\label{def:T-s}
\overline{T}_{\alpha}((t,x);(s,y) )\coloneqq {T}_{\alpha}((t,x);(s,y) ).
\end{align} 

We now compute the transition probability along a general space-like path in which both $\alpha$ and $t$ can change. We can again consider first consider the case $N=1$ and obtain the full result using the Cauchy--Binet formula. In this case, $\overline{T}_\alpha((t,x);(s,y))dy$ is the transition density from $X^{1 }(\alpha,t)={x}$ to $X^{1 }(\alpha,s)=y$ and $\overline{W}_t((\alpha,x);(\beta,y) )dy$ is the transition density from $X^{1 }(\alpha,t)=x$ to $X^{1 }(\beta,t)=y$. 

Considering $(\alpha,t)\prec_{\textup{t}}(\beta,s)$, define $\overline{Q}((\alpha,t,x);(\beta,s,y))dy$ as the transition density from $X^{1 }(\alpha,t)=x$ to $X^{1 }(\beta,s)=y$. From the Markov property along space-like paths, for all $(\alpha,t)\prec_{\textup{s}} (\beta,s)\prec_{\textup{s}} (\gamma,\tau)$, we have
\begin{align}\label{equ:QQ-s}
\int_{0}^\infty dy\, \overline{Q}  ((\alpha,t,x);(\beta,s,y))\overline{Q}((\beta,s,y);(\gamma,\tau,z))=\overline{Q}((\alpha,t,x);(\gamma,\tau,z)).
\end{align}
Besides, we can write $\overline{Q}((\alpha,t,x);(\beta,s,y))$ as
\begin{equation}\label{equ:Qform-s}
\overline{Q}((\alpha,t,x);(\beta,s,y))=\left\{ \begin{array}{cc}
\overline{T}_\alpha((t,x);(s,y)), &\textup{if}\ \alpha=\beta\ \textup{and}\ t<s,\\
\overline{W}_t((\alpha,x);(\beta,y) ), &\textup{if}\ \alpha>\beta\ \textup{and}\ t=s,\\
\int_0^\infty dz\, \overline{W}_t((\alpha,x);(\beta,z) )\overline{T}_\beta((t,z);(s,y)), &\textup{if}\ \alpha>\beta\ \textup{and}\ t<s.
\end{array} \right.
\end{equation}
Note that we also recover a commutation relation between $\overline{T}_\alpha$ and $\overline{W}_t$ with \eqref{equ:QQ-s} and \eqref{equ:Qform-s}. For any $\alpha>\beta\in\mathbb{N}_0$ and $t<s\in (0,\infty)$, it holds that
\begin{align}\label{equ:commm-s}
\int_0^\infty dz\, \overline{T}_\alpha((t,x);(s,z))\overline{W}_t((\alpha,z);(\beta,y) )= \int_0^\infty dz\, \overline{W}_t((\alpha,x);(\beta,z) )\overline{T}_\beta((t,z);(s,y)). 
\end{align}
An integral representation of $\overline{Q}((\alpha,t,x);(\beta,s,y))$ is given at the end of this section. We are now able to give the transition probability from $\vec{X}^{(k)}$ to $\vec{X}^{(k+1)}$ using $\overline{Q}((\alpha,t,x);(\beta,s,y))$.
\begin{corollary}\label{cor:transition-s}
Fix $N\in\mathbb{N}$ and $m\in\mathbb{N}$. Let $\left\{\left(\alpha^{(k)},t^{(k)}\right)\right\}_{k=1}^m\subset\mathbb{N}_0\times (0,\infty)$  be a space-like sequence as in Definition~\ref{def:spacelike}. Let $\vec{X}^{(k)}$ be defined as in \eqref{def:Xk-spacelike}. Then for all $k\in \llbracket  1,m-1 \rrbracket $ and $\vec{x}^{(k)}\in\mathbb{W}^N_+$,
The transition density from $\vec{X}^{(k)}=\vec{x}^{(k)}$ to $\vec{X}^{(k+1)}=\vec{x}^{(k+1)}$ is given by
\begin{multline}\label{equ:transition-s}
\frac{C(N,\alpha^{(k+1)})}{C(N,\alpha^{(k)})}  \det\left[ \overline{Q}((\alpha^{(k)},t^{(k)},x^{(k)}_i );(\alpha^{(k+1)},t^{(k+1)},x^{(k+1)}_j )) \right]_{1\leq i,j\leq N}
\times \\ \times \frac{\Delta(\vec{x}^{(k+1)})}{\Delta(\vec{x}^{(k)})} \mathbbm{1}(\vec{x}^{(k+1)}\in \mathbb{W}_+^N)\prod_{j=1}^N dx_j^{(k+1)} .
\end{multline}
\end{corollary}
\begin{proof}
If $\alpha^{(k)}=\alpha^{(k+1)}$ and $t^{(k)}<t^{(k+1)}$, from \eqref{equ:Qform-s},  \eqref{equ:transition-s} is equivalent to \eqref{timechange}. If $\alpha^{(k)}>\alpha^{(k+1)}$ and $t^{(k)}=t^{(k+1)}$, from \eqref{equ:Qform-s},  \eqref{equ:transition-s} is equivalent to \eqref{alphachange-s}. Using \eqref{equ:QQ-s} and the Cauchy–Binet formula, \eqref{equ:transition-s} for general $(\alpha^{(k)},t^{(k)})\prec_{\textup{s}} (\alpha^{(k)},t^{(k)})$ follows.
\end{proof}
Combining \eqref{onept_density} and Corollary~\ref{cor:transition-s}, we obtain the joint density of $\vec{X}^{(k)}$.
\begin{corollary}\label{cor:density-s}
Fix $N\in\mathbb{N}$ and $m\in\mathbb{N}$. Let $\left\{\left(\alpha^{(k)},t^{(k)}\right)\right\}_{k=1}^m\subset\mathbb{N}_0\times (0,\infty)$  be a space-like path as in Definition~\ref{def:spacelike}. Then the joint density of $\vec{X}^{(k)}$ defined in \eqref{def:Xk-spacelike} is given by
\begin{equation}\label{density-s}
\begin{split}
&\left(t^{(1)}\right)^{-N^2}\prod_{j=1}^N  (x^{(1)}_j/t^{(1)})^{\alpha^{(1)}}e^{-x^{(1)}_j /t^{(1)}}\times  \Delta(\vec{x}^{(1)})\\
&\times \prod_{k=1}^{m-1} \det\bigg[ \overline{Q}((\alpha^{(k)},t^{(k)},x^{(k)}_i );(\alpha^{(k+1)},t^{(k+1)},x^{(k+1)}_j )) \bigg]_{1\leq i,j\leq N}\\
&\times C(N,\alpha^{(m)})  \Delta(\vec{x}^{(m)}) \prod_{k=1}^m \mathbbm{1}(\vec{x}^{(k)}\in\mathbb{W}^N_+)  \prod_{k=1}^m\prod_{j=1}^N dx^{(k)}_j.
\end{split}
\end{equation}
\end{corollary}

Finally, we give an integral representation of $\overline{Q}$ which we prove in Appendix \ref{sec:B}.  
\begin{lemma}\label{lem:Qintegral-s}
Let $(\alpha,t)\prec_{\textup{s}} (\beta,s)$ be space-like pairs in $ \mathbb{N}_0\times (0,\infty)$ as in Definition \ref{def:spacelike} and $x,y\in (0,\infty)$. It holds that
\begin{equation}\label{equ:Qintegral-s}
\begin{split}
\overline{Q}((\alpha,t,x);(\beta,s,y))= &x^{-\alpha/2}y^{ \beta/2} \times\int_0^\infty du\,  e^{-(s-t)u}u^{-(\alpha-\beta)/2}J_{\alpha}(2\sqrt{ xu})J_{\beta}(2\sqrt{  yu}).
\end{split}
\end{equation} 
Here $J_\alpha $ is the Bessel function of the first kind.
\end{lemma}

\subsection{Correlation Kernel}\label{sec:3.2}
We prove in this subsection that $\vec{X}^{(k)}$ defined in \eqref{def:Xk-spacelike} is a determinantal point process and give its correlation kernel.
\vspace{\baselineskip}

For $j\in\mathbb{N}$, $\alpha\in\mathbb{N}_0$ and $t>0$, define the functions
\begin{align}
\overline{\phi}_j(\alpha,t,x)& \coloneqq \Gamma(\alpha+j)\phi_j(\alpha,t,x)=  {\Gamma(j)}t^{-j}(x/t)^{\alpha }e^{-x/t}L^\alpha_{j-1}(t^{-1}x),\label{def:phi-s}\\
\overline{\psi}_j(\alpha,t,x)&\coloneqq \frac{1}{\Gamma(\alpha+j)}\psi_j(\alpha,t,x)=\frac{1}{ \Gamma(\alpha+j) } t^{j-1} L^\alpha_{j-1}(t^{-1}x)\label{def:psi-s}
\end{align}
with $L^\alpha_{j}(x)$ the generalized Laguerre polynomials defined in Section~\ref{sec:2.2}. 

Similarly as in Section \ref{sec:2}, the joint density \eqref{density-s} can be written
\begin{equation}\label{density2-s}
\begin{split}
&\det \left[\overline{\phi}_i(\alpha^{(1)},t^{(1)}, x^{(1)}_j)\right]_{1\leq i,j\leq N}\\
\times &\prod_{k=1}^{m-1} \det\bigg[ \overline{Q}((\alpha^{(k)},t^{(k)},x^{(k)}_i );(\alpha^{(k+1)},t^{(k+1)},x^{(k+1)}_j )) \bigg]_{1\leqq i,j\leq N}\\
\times &\det \left[\overline{\psi}_j(\alpha^{(m)},t^{(m)}, x^{(m)}_i)\right]_{1\leq i,j\leq N}\times \prod_{k=1}^m \mathbbm{1}(\vec{x}^{(k)}\in\mathbb{W}^N_+)  \prod_{k=1}^m\prod_{j=1}^N dx^{(k)}_j.
\end{split}
\end{equation}

We also have the compatibility property of $\overline{\phi}_j(\alpha,t,x)$ and $ \overline{\psi}_j(\alpha,t,x)$ with $\overline{Q}((\alpha,t,x);(\beta,s,y))$.

\begin{lemma}\label{lem:phiQpsi-s}
Let $(\alpha,t)\prec_{\textup{s}} (\beta,s)$ be space-like pairs as in Definition \ref{def:spacelike}. Then the following hold. 
\begin{align}\label{delta-s}
\int_{0}^\infty dx \,   \overline{\phi}_i(\alpha,t,x) \overline{\psi}_j(\alpha,t,x)=\delta_{ij}.
\end{align}
\begin{align}\label{phiQ-s}
\int_{0}^\infty dx \,  \overline{\phi}_j(\alpha,t,x) \overline{Q}((\alpha,t,x);(\beta,s,y))=\overline{\phi}_j(\beta,s,y).
\end{align}
\begin{align}\label{Qpsi-s}
\int_{0}^\infty dy\, \overline{Q}((\alpha,t,x);(\beta,s,y))\overline{\psi}_j(\beta,s,y) =\overline{\psi}_j(\alpha,t,x).
\end{align}
\end{lemma}

\begin{corollary}\label{cor:kernel-s}
Fix $N\in\mathbb{N}$ and $m\in\mathbb{N}$. Let $\left\{\left(\alpha^{(k)},t^{(k)}\right)\right\}_{k=1}^m$  be a space-like sequence as in Definition~\ref{def:spacelike}. Let $\vec{X}^{(k)}$ be defined as in \eqref{def:Xk-spacelike}. Then  $ \{\vec{X}^{(k)} \}_{k=1}^m$ is a determinantal point process. Furthermore, the correlation kernel is given by $\overline{K}^N((\alpha^{(k)},t^{(k)},x);(\alpha^{(\ell)},t^{(\ell)},y))$, where
\begin{multline}\label{equ:kernel-s}
\overline{K}^N((\alpha,t,x);(\beta,s,y))=- \overline{Q}((\alpha,t,x);(\beta,s,y))\mathbbm{1}((\alpha,t)\prec_{\textup{s}} (\beta,s) )\\
+\sum_{j=1}^N \overline{\psi}_j(\alpha,t,x)\overline{\phi}_j(\beta,s,y).
\end{multline}

\end{corollary}
\begin{proof}
The assertion follows by combining \eqref{density2-s}, Lemma \ref{lem:phiQpsi-s} and the Eynard--Mehta Theorem \cite{EM,BR}. 
\end{proof}
We now prove Lemma~\ref{lem:phiQpsi-s}.
\begin{proof}[Proof of Lemma~\ref{lem:phiQpsi-s}]
Because $\overline{\phi}_j(\alpha,t,x)=\Gamma(\alpha+j)\phi_j(\alpha,t,x)$ and $\overline{\psi}_j(\alpha,t,x)=\frac{1}{\Gamma(\alpha+j)}\psi_j(\alpha,t,x)$, \eqref{delta-s} is a direct consequence of \eqref{delta}.
\vspace{\baselineskip}

 We turn to \eqref{Qpsi-s}. Note that it is enough to show by \eqref{equ:QQ-s} and \eqref{equ:Qform-s}
\begin{align}
\int_{0}^\infty  dy \, \overline{T}_{\alpha}((t,x);(s,y) )\overline{\psi}_j(\alpha,s,y) =\overline{\psi}_j(\alpha,t,x),\label{Tpsi-s}\\
\int_{0}^\infty   dy\, \overline{W}_t((\alpha,x);(\beta,y) )\overline{\psi}_j(\beta,t,y)=\overline{\psi}_j(\alpha,t,x).\label{Wpsi-s}
\end{align}
Because $\overline{T}_{\alpha}((t,x);(s,y))={T}_{\alpha}((t,x);(s,y) )$, \eqref{Tpsi-s} is a special case of \eqref{Qpsi}. We then focus on \eqref{Wpsi-s}.
\begin{align*}
\int_{0}^\infty  dy\,  \overline{W}_t((\alpha,x);(\beta,y) )\overline{\psi}_j(\beta,t,y) =&\frac{t^{j-1}x^{-\alpha}}{\Gamma(\alpha-\beta)\Gamma(\beta+j)}\int_0^x dy\,  y^\beta(x-y)^{\alpha-\beta-1}L^\alpha_{j-1}(y/t)\\
=&\frac{t^{j-1} }{\Gamma(\alpha-\beta)\Gamma(\beta+j)}\int_0^1 dz\, z^\beta(1-z)^{\alpha-\beta-1}L^\alpha_{j-1}(xz/t). 
\end{align*}
From \cite[(7.412-1),page 809]{GR}
\begin{align*}
\int_0^1 dz\, z^\beta(1-z)^{\alpha-\beta-1}L^\alpha_{j-1}(xz/t) =\frac{\Gamma(\beta+j)\Gamma(\alpha-\beta)}{\Gamma(\alpha+j)} L^{\alpha}_{j-1}(x/t).
\end{align*}
Therefore,
\begin{align*}
\int_{0}^\infty dy\,  \overline{W}_t((\alpha,x);(\beta,y) )\overline{\psi}_j(\beta,t,y) =&\frac{t^{j-1} }{\Gamma(\alpha+j)}    L^{\alpha}_{j-1}(x/t)=\overline{\psi}_j(\alpha,t,x).
\end{align*}
This finishes the proof of \eqref{Qpsi-s}.
\vspace{\baselineskip}

Next, we prove \eqref{phiQ-s}. For $j\in\mathbb{N}$,  denote 
\begin{align*}
\tilde{\phi}_j(y)\coloneqq \int_{0}^\infty \overline{\phi}_j(\alpha,t,x) \overline{Q}((\alpha,t,x);(\beta,s,y))\, dx
\end{align*}
The goals is to show that $\tilde{\phi}_j(y)=\overline{\phi}_j(\beta,s,y).$ Using \eqref{density2-s} and the Cauchy--Binet formula, we have 
\begin{align*}
\det [\tilde{\phi}_j(y_i)]_{1\leq i,j\leq N} =(-1)^{N(N-1)/2}s^{-N^2}  \prod_{j=1}^N (y_j/s)^{\alpha}e^{-y_j /s} \times\Delta(\vec{y})
\end{align*}
and $(y/s)^{-\alpha} e^{y/s} \tilde{\phi}_j(y)$ is a polynomial of degree $j-1$. Since  we have
\begin{align*}
\int_{0}^\infty \tilde{\phi}_i(y) \overline{\psi}_j(\beta,s,y)\, dy=\int_{0}^\infty \overline{\phi}_i(\alpha,t,x)\overline{\psi}_j(\alpha,t,x)\, dx=\delta_{ij} 
\end{align*}
from \eqref{Qpsi-s} and \eqref{delta-s}, we obtain that $\tilde{\phi}_j(y)=\overline{\phi}_j(\beta,s,y).$
\end{proof}

\subsection{Scaling Limit}\label{sec:3.3}
We now perform the hard edge scaling and show that the scaled kernel converges locally uniformly. Consider a fixedd space-like path $\{ (\alpha^{(k)},t^{(k)}) \}$ in $\mathbb{N}_0\times \mathbb{R}$. Note again that $t^{(k)}$ is not required to be positive. Define
\begin{equation}
\overline{ \mathcal{B}}_j^{N,(k)}\coloneqq 4N\cdot X_j^{N,\alpha^{(k)}}\left( 1+ {t^{(k)}}/{4N} \right),\ j\in\llbracket 1,N \rrbracket. 
\end{equation} 
$\overline{\mathcal{B}}_j^{N,(k)}$ is a determinantal point process by Corollary \ref{cor:kernel-s} with correlation kernel \[\overline{K}^N_{\textup{scaled}}((\alpha^{(k)},t^{(k)},x);(\alpha^{(\ell)},t^{(\ell)},y)).
\] 
where
\begin{align*}
\overline{K}^N_{\textup{scaled}}((\alpha ,t ,x);(\beta ,s,y))\coloneqq (4N)^{-1}\overline{K}^N\left(  ( \alpha ,1+ {t }/{4N}, {x}{/4N} );( \beta,1+ {s}/{4N}, {y}{/4N} )\right).
\end{align*}    
We define the gauged kernel as
\begin{multline}\label{def:Kgauge-s}
\overline{K}^N_{\textup{gauge}}((\alpha,t,x);(\beta,s,y))\coloneqq (4N)^{-1}x^{\alpha/2}y^{-\beta/2}\exp\left( -\frac{x}{8N+2t}+\frac{y}{8N+2s}\right) \\
\times  \overline{K}^N\left(  ( \alpha ,1+ t/{4N}, {x}{/4N} );( \beta,1+ t/{4N}, {s}{/4N} )\right)
\end{multline}
We remark that we have the following relation between $\overline{K}^N_{\textup{gauge}}$ and $\overline{K}^N_{\textup{scaled}}$
 $$\overline{K}^N_{\textup{gauge}}((\alpha,t,x);(\beta,s,y))=  \frac{\overline{f}^N(\alpha,t,x)}{\overline{f}^N(\beta,s,y)}\cdot  \overline{K}^N_{\textup{scaled}}((\alpha,t,x);(\beta,s,y)). $$
In particular, $\overline{K}^N_{\textup{gauge}}((\alpha^{(k)},t^{(k)},x);(\alpha^{(\ell)},t^{(\ell)},y))$ is also a correlation kernel for $\overline{\mathcal{B}}_j^{N,(k)}$.
\vspace{\baselineskip}

We want to show the convergence of the kernel $\overline{K}^N_{\textup{gauge}}((\alpha^{(k)},t^{(k)},x);(\alpha^{(\ell)},t^{(\ell)},y))$ to the kernel $\overline{K}^{\textup{Bes}}((\alpha^{(k)},t^{(k)},x);(\alpha^{(\ell)},t^{(\ell)},y))$. For ease of notation, we denote
\begin{multline*}
\overline{K}^{N,1}((\alpha,t,x);(\beta,s,y))\coloneqq -(4N)^{   -1}x^{\alpha/2}y^{-\beta/2} \exp\left( -\frac{x}{8N+2t}+\frac{y}{8N+2s}\right) \\
\times \mathbbm{1}((\alpha,t)\prec_{\textup{s}} (\beta,s) )\times \overline{Q}\left(\left(\alpha,1+  {t}/{4N},  {x}/{4N}\right);\left(\beta,1+  {s}/{4N}, {y}/{4N}\right)\right),
\end{multline*}
\begin{multline*}
\overline{K}^{N,2}((\alpha,t,x);(\beta,s,y))\coloneqq (4N)^{   -1}x^{\alpha/2}y^{-\beta/2} \exp\left( -\frac{x}{8N+2t}+\frac{y}{8N+2s}\right)\\
\times \sum_{j=1}^N \overline{\psi}_j\left(\alpha,1+ {t}/{4N}, {x}/{4N}\right)\cdot \overline{\phi}_j\left(\beta,1+ {s}/{4N}, {y}/{4N}\right),
\end{multline*}
and
\begin{align*}
\overline{K}^{N,1}_{k\ell}(x,y)& \coloneqq  \overline{K}^{N,1}((\alpha^{(k)},t^{(k)},x);(\alpha^{(\ell)},t^{(\ell)},y)),\\ 
\overline{K}^{N,2}_{k\ell}(x,y)& \coloneqq  \overline{K}^{N,2}((\alpha^{(k)},t^{(k)},x);(\alpha^{(\ell)},t^{(\ell)},y)). 
\end{align*}
Similarly, let
\begin{align*}
\overline{K}^{1}((\alpha,t,x);(\beta,s,y))&\coloneqq -\displaystyle\int_{0}^{\infty} du\, u^{-(\alpha-\beta)/2}e^{-(s-t)u}J_\alpha(2\sqrt{xu})J_\beta(2\sqrt{yu})\times   \mathbbm{1}((\alpha,t)\prec_{\textup{s}} (\beta,s) ) ,\\
\overline{K}^{2}((\alpha,t,x);(\beta,s,y))&\coloneqq  \displaystyle\int_{0}^{\infty} du\, u^{-(\alpha-\beta)/2}e^{-(s-t)u}J_\alpha(2\sqrt{xu})J_\beta(2\sqrt{yu}),
\end{align*}
and
\begin{align*}
\overline{K}^{1}_{k\ell}(x,y)&\coloneqq  \overline{K}^{1}((\alpha^{(k)},t^{(k)},x);(\alpha^{(\ell)},t^{(\ell)},y)),\\ 
\overline{K}^{2}_{k\ell}(x,y)&\coloneqq  \overline{K}^{2}((\alpha^{(k)},t^{(k)},x);(\alpha^{(\ell)},t^{(\ell)},y)). 
\end{align*}
We have
\begin{align*}
\overline{K}^N_{\textup{gauge}}((\alpha^{(k)},t^{(k)},x);(\alpha^{(\ell)},t^{(\ell)},y))&=\overline{K}^{N,1}_{k\ell}(x,y)+K^{N,2}_{k\ell}(x,y),\\
\overline{K}^{\textup{Bes}}((\alpha^{(k)},t^{(k)},x);(\alpha^{(\ell)},t^{(\ell)},y))&=\overline{K}^{1}_{k\ell}(x,y)+\overline{K}^{ 2}_{k\ell}(x,y).
\end{align*}


$\{ \overline{K}^{N,1}_{k\ell} \}_{1\leq k,\ell\leq m}$ can be seen as a function from $\mathfrak{X}\times \mathfrak{X}$ with $\mathfrak{X}=\llbracket 1,m \rrbracket\times\mathbb{R}_+$.  
\begin{lemma}\label{lem:E1-s}
$\{ \overline{K}^{N,1}_{k\ell} \}_{1\leq k,\ell\leq m}$ converges to $\{ \overline{K}^{1}_{k\ell} \}_{1\leq k,\ell\leq m}$ locally in $\mathcal{L}_{1|2}$ as defined in Definition~\ref{def:L12}.  
\end{lemma}
\begin{lemma}\label{lem:E2-s}
$\{ \overline{K}^{N,2}_{k\ell} \}_{1\leq k,\ell\leq m}$ converges to $\{ \overline{K}^{2}_{k\ell} \}_{1\leq k,\ell\leq m}$ locally in $\mathcal{L}_{1|2}$ as defined in Definition~\ref{def:L12}. Moreover, $\{ \overline{K}^{N,2}_{k\ell} \}_{1\leq k,\ell\leq m}$ and $\{ \overline{K}^{2}_{k\ell} \}_{1\leq k,\ell\leq m}$ are continuous on the diagonal blocks. See Definition~\ref{def:Kconti}.
\end{lemma}

We are now ready to prove Theorem~\ref{thm:main} (ii).
\begin{proof}[Proof of Theorem~\ref{thm:main} (ii)]
By definition of $\mathcal{B}(\alpha,t)$, there exists a subsequence of $\overline{\mathcal{B}}^N(\alpha^{(k)},t^{(k)})$ which converges in distribution to $\mathcal{B}(\alpha^{(k)},t^{(k)})$. Note that we abuse the notation and still denote such subsequence by $N$ for simplicity. From Lemma~\ref{lem:E1} and Lemma~\ref{lem:E2}, the correlation kernel of $\overline{\mathcal{B}}^N(\alpha^{(k)},t^{(k)})$ converges locally in $\mathcal{L}_{1|2}$ to $\overline{K}^{\textup{Bes}}$. Finally, all kernels considered are continuous on the diagonal and the assertion follows by Corollary~\ref{cor:det}.
\end{proof}
The rest of the section is devoted to the proof of Lemmas~\ref{lem:E1-s} and \ref{lem:E2-s}.

\begin{proof}[Proof of Lemma~\ref{lem:E1-s}]
Note that it suffices to prove that $\overline{K}^{N,1}_{k\ell}(x,y)$ converges to $\overline{K}^{1}_{k\ell}(x,y)$ locally in $L^2$ for all $k<\ell$ since both $\overline{K}^{N,1}_{k\ell}(x,y)$ and $\overline{K}^{1}_{k\ell}(x,y)$ vanish for $k\geq \ell$. We denote $(\alpha^{(k)},t^{(k)})=(\alpha,t)$ and $(\alpha^{(\ell)},t^{(\ell)})=(\beta,s)$.   By a change of variables combined with the integral representation \eqref{equ:Qintegral-s}, we have
  \begin{multline*}\overline{Q}\left(\left(\alpha,1+  {t}/{4N},  {x}/{4N}\right);\left(\beta,1+  {s}/{4N}, {y}/{4N}\right)\right)
  \\=
 (4N) x^{-\alpha/2}y^{\beta/2}  \int_0^\infty du\,u^{-(\alpha-\beta)/2}  e^{-(s-t)u} J_{\alpha}\left(2\sqrt{ xu }\right)J_{\beta}\left(2\sqrt{ yu }\right)
 =-(4N) x^{-\alpha/2}y^{\beta/2}\overline{K}_{k\ell}^{1}(x,y).
\end{multline*}
Thus, we have
\begin{align*}
\overline{K}_{k\ell}^{N,1}(x,y)=  \exp\left(- \frac{x}{8N+2t}+\frac{y}{8N+2s}\right)\overline{K}_{k\ell}^{1}(x,y).
\end{align*}
Similar to the proof of Lemma~\ref{lem:E1}, $\overline{K}_{k\ell}^{1}(x,y)$ is a locally $L^2$ function. This suffices to imply $\overline{K}_{k\ell}^{N,1}(x,y)$ converges to $\overline{K}_{k\ell}^{1}(x,y)$ locally in $L^2$. 
\end{proof}

\begin{proof}[Proof of Lemma~\ref{lem:E2-s}]
By the same argument as in the proof of Lemma~\ref{lem:E2}, it  suffices to prove that $\overline{K}_{k\ell}^{N,2}(x,y)$ converges to $\overline{K}_{k\ell}^{2}(x,y)$ locally uniformly. Fix $k,\ell\in\llbracket 1,m \rrbracket$ and let $(\alpha^{(k)},t^{(k)})=(\alpha,t)$ and $(\alpha^{(\ell)},t^{(\ell)})=(\beta,s)$. Using the definitions of $\overline{\psi}_j$ and $\overline{\phi}_j$ in \eqref{def:phi-s} and \eqref{def:psi-s}, we have 
\begin{multline}\label{equ:E2N-s}
\overline{K}^{N,2}_{k\ell}(x,y)=(4N)^{-1}x^{\alpha/2}y^{\beta/2} (4N+s)^{-\beta}(1+t/4N)^{-1} \exp\left(-\frac{x}{8N+2t}-\frac{y}{8N+2s}\right)
\\\times 
\sum_{j=1}^N \frac{\Gamma(j)}{\Gamma(\alpha+j)} \left(\frac{1+t/4N}{1+s/4N}\right)^{j}  \cdot L^\alpha_{j-1}\bigg(\frac{x}{4N+t}\bigg)\cdot L^\beta_{j-1}\left(\frac{y}{4N+s}\right)
\end{multline}
We assume that $x,y\in [0,M]$ for some $M>0$, $j\in \llbracket 1,N \rrbracket$ and denote by $O(1)$ a quantity which can be bounded by a constant depending on $\alpha,\beta,s,t$ and $M$. If we combine \eqref{equ:Lasymptotic} with 
\begin{equation*}
	\begin{gathered}
4N+t=4N(1+O(N^{-1})),\quad 4N+s=4N(1+O(N^{-1})),
\\ \exp\left( -\frac{y}{8N+2s}-\frac{x}{8N+2t} \right)=1+O(N^{-1}),\\ \frac{\Gamma(j)}{\Gamma(\alpha+j)}=j^{-\alpha}(1+O(j^{-1})),\quad\text{and}\quad \left(\frac{1+t/4N}{1+s/4N}\right)^{j} =\exp\left( -\frac{(s-t)j}{4N} \right)(1+O(N^{-1})),
	\end{gathered}  
\end{equation*}
we obtain
\begin{multline*}
\overline{K}^{N,2}_{k\ell}(x,y)=\frac{1}{4N} x^{\alpha/2}y^{\beta/2}\sum_{j=1}^N\exp \left( -\frac{(s-t)j}{4N} \right)\left( \frac{j}{4N} \right)^\beta \left[ g_\alpha\left( \frac{jx}{4N} \right)+O(j^{-1}) \right]
\\ \times 
\left[ g_\beta\left( \frac{jy}{4N} \right)+O(j^{-1}) \right] (1+O(j^{-1}))   
\end{multline*}
Since $g_\alpha (jx/4N)$ and $g_\beta(jy/4N)$ are bounded, we have
 \begin{multline*}
\overline{K}^{N,2}_{k\ell}(x,y)-\frac{1}{4N} x^{\alpha/2}y^{\beta/2}\sum_{j=1}^N\exp \left( -\frac{(s-t)j}{4N} \right)\left( \frac{j}{4N} \right)^\beta  g_\alpha\left( \frac{jx}{4N} \right) g_\beta\left( \frac{jy}{4N} \right)
\\= x^{\alpha/2}y^{\beta/2}O(N^{-1}\log N). 
\end{multline*}
Since the function $u\mapsto e^{-(s-t)u}u^\beta g_\alpha(xu)g_\beta(yu) $ has bounded derivatives on $u\in [0,4^{-1}]$, we have
\begin{multline*}
\frac{1}{4N} \sum_{j=1}^N\exp \left( -\frac{(s-t)j}{4N} \right)\left( \frac{j}{4N} \right)^\beta  g_\alpha\left( \frac{jx}{4N} \right) g_\beta\left( \frac{jy}{4N} \right)
\\=O(N^{-1})+\int_0^{1/4} du\,e^{-(s-t)u}u^\beta g_\alpha(xu)g_\beta(yu).
\end{multline*}
Finally,
\begin{align*}
\overline{K}^{N,2}_{k\ell}(x,y)=&x^{\alpha/2}y^{\beta/2} O(N^{-1}\log N)+x^{\alpha/2}y^{\beta/2} \int_0^{1/4} du\,e^{-(s-t)u}u^\beta g_\alpha(xu)g_\beta(yu)\\
=&x^{\alpha/2}y^{\beta/2} O(N^{-1}\log N)+\int_0^{1/4} du\,e^{-(s-t)u}u^{-(\alpha-\beta)/2} J_\alpha(2\sqrt{xu}) J_\beta(2\sqrt{yu})
\end{align*}
where we used \eqref{def:galpha}. This proves locally uniform convergence in $x$ and $y$.
\end{proof}

\section{Exponential Gibbs property}\label{sec:5}
Let $\mathcal{B}=\{\mathcal{B}(\alpha,t),\alpha\in\mathbb{N}_0,t\in\mathbb{R}\}$ be a Bessel field defined in Section~\ref{sec:1.3}. The goal of this section is to show that for any $t\in\mathbb{R}$, $\{\mathcal{B}(\alpha,t), \alpha\in\mathbb{N}_0\}$ satisfies the exponential Gibbs property. We first illustrate the exponential Gibbs property for the Laguerre field in Section~\ref{sec:6.1}. Then in Section~\ref{sec:6.2} we show that the Gibbs property is preserved under the hard edge limit which leads to the exponential Gibbs property for $\{\mathcal{B}(\alpha,t), \alpha\in\mathbb{N}_0\}$.  
\subsection{Exponential Gibbs property for the Laguerre field}\label{sec:6.1}
We need the following lemma, which is known as Sasamoto's trick. We provide a proof for reader's convenience.
\begin{lemma}\label{lem:sasamoto}
Fix $N\in\mathbb{N}$. Let $\vec{a}=(a_1,a_2, \cdots,a_N), \vec{b}=(b_1,b_2, \cdots,b_N)\in\mathbb{W}_+^N$ and assume $a_i\neq b_j$ for $i,j\in\llbracket 1,N \rrbracket $. We have the following equivalence,
\begin{align}\label{eqn:sasamoto}
\det\left[\mathbbm{1}(a_i<b_j)\right]_{1\leq i,j\leq N } =\mathbbm{1} \left(\vec{a}\prec\vec{b}\right).
\end{align}
\end{lemma}

\begin{proof}
We prove by induction on $N$. The case when $N=1$ is trivial. Assume \eqref{eqn:sasamoto} holds for $N=k$, we proceed to show it as well holds for $N=k+1$. Let $M=\{\mathbbm{1}(a_i<b_j)\}_{1\leq i,j\leq k+1}$. If $a_{k+1}>b_{k+1}$, the last row of $M$ has zero entries. This implies both sides of \eqref{eqn:sasamoto} are zero. If $a_{k+1}<b_k$, the last two columns of $M$ are identical. This again implies both sides of \eqref{eqn:sasamoto} are zero. It remains to consider the case $b_k<a_{k+1}<b_{k+1}$. The last row of $M$ is all zeros except for the $(k+1,k+1)$-th entry, the assertion follows by the induction hypothesis. The proof is finished.
\end{proof}

Recall that for $\alpha\in\mathbb{N}_0$ and $t>0$, $\vec{X}^N(\alpha,t)=( {X}_1^N(\alpha,t), {X}_2^N(\alpha,t),\dots, {X}_N^N(\alpha,t))$ are the ordered eigenvalues of $(A^{N}(\alpha,t))^*  A^{N}(\alpha,t)  $ defined in \eqref{def:A}. In the next lemma,  we show that for any $t_0>0 $, $\vec{X}^N(\alpha,t_0)$ satisfies the exponential Gibbs property.

\begin{lemma}\label{lem:bridgeinterlacing}
Fix $N\in\mathbb{N}$ and $t_0>0$. Then $X^N_j(\alpha,t_0)$, $(j,\alpha)\in\llbracket 1,N \rrbracket\times \mathbb{N}_0$ satisfies the exponential Gibbs property defined in Definition~\ref{def:ExpoGP}. 
\end{lemma}

\begin{proof}
From Remark \ref{rmk:ExpoGP}, it suffices to prove \eqref{equ:ZN>0} and \eqref{Gibbs-condition-N}. We start with \eqref{equ:ZN>0}. Fix $\alpha<\beta$ with $\alpha,\beta\in\mathbb{N}$. Let 
$$\vec{x}=(X^N_1(\alpha,t_0),X^N_2(\alpha,t_0),\dots,X^N_N(\alpha,t_0)),$$ $$\vec{y}=(X^N_1(\beta,t_0),X^N_2(\beta,t_0),\dots,X^N_N(\beta,t_0)).$$ 
We aim to show that $Z_{\alpha ,\beta}(\vec{x},\vec{y})>0$ almost surely. Let $\vec{w}=(\vec{w}^{ \alpha+1 },\dots,\vec{w}^{ \beta-1 })$ with $$\vec{w}^{ \gamma }=(X^N_1(\gamma,t_0),X^N_2(\gamma,t_0),\dots,X^N_N(\gamma,t_0)). $$
By the Cauchy interlace theorem, we have $\vec{w}\in I_{\alpha ,\beta}(\vec{x},\vec{y})$. In particular, $I_{\alpha ,\beta}(\vec{x},\vec{y})\neq \emptyset$. Note that from \eqref{density2}, the joint density of $(\vec{x},\vec{y} )$ is absolutely continuous with respect to the Lebesgue measure. In the view of Lemma~\ref{lem:boundaryZ}, we have $Z_{\alpha ,\beta}(\vec{x},\vec{y} )>0$ with probability one.
\vspace{\baselineskip}

Next, we turn to proving  \eqref{Gibbs-condition-N}. Fix $(\vec{x},\vec{y})$ with $Z_{\alpha,\beta}(\vec{x},\vec{y})>0$. 
In view of Lemma~\ref{lem:alphachange}, the conditional density of $\vec{w}$ is proportional to
\begin{align*}
\prod_{\gamma=\alpha+1}^{\beta }\det\left[e^{ {w}^{\gamma}_j- {w}^{\gamma-1}_i} \mathbbm{1}\left( {w}^{\gamma-1}_i\leq  w^{\gamma}_j\right)\right]_{1\leq i, j\leq N}  \frac{\Delta(\vec{w}^{\gamma})}{\Delta(\vec{w}^{\gamma-1})} \prod_{\gamma=\alpha+1}^{\beta-1} \mathbbm{1}\left( \vec{w}^{\gamma}\in\mathbb{W}_+^N\right).
\end{align*}
Here we adopt the convention $\vec{w}^{ \alpha }=\vec{x}$ and $\vec{w}^{ \beta }=\vec{y}$. After canceling the telescoping terms, the above term is proportional to
\begin{align*}
\prod_{\gamma=\alpha+1}^{\beta }\det\left[  \mathbbm{1}\left( {w}^{\gamma-1}_i\leq  w^{\gamma}_j\right)\right]_{1\leq i, j\leq N}    \prod_{\gamma=\alpha+1}^{\beta-1} \mathbbm{1}\left( \vec{w}^{\gamma}\in\mathbb{W}_+^N\right).
\end{align*}
Applying Lemma~\ref{lem:sasamoto}, the above term further reduces to (up to a measure zero set)
\begin{align*}
\prod_{\gamma=\alpha}^{\beta+1}\mathbbm{1}\left( \vec{w}^{\gamma-1}\prec \vec{w}^{\gamma}\right)=  \mathbbm{1}\left(\vec{x} \prec \vec{w}^{(\alpha+1)}\prec \vec{w}^{(\alpha+2)}\prec \cdots \prec\vec{w}^{(\beta-1)}\prec \vec{y} \right).
\end{align*}
This is the indicator function of $\textup{int}(I_{\alpha ,\beta}(\vec{x},\vec{y}))$. From Lemma~\ref{lem:boundaryZ}, $I_{\alpha ,\beta}(\vec{x},\vec{y})\setminus \textup{int}(I_{\alpha ,\beta}(\vec{x},\vec{y}))$ is of measure zero. Therefore \eqref{Gibbs-condition-N} follows.
\end{proof}

\subsection{Proof of Theorem~\ref{thm:main} (iii)}\label{sec:6.2} In this section we seek to prove Theorem \ref{thm:main} (iv), i.e. the Bessel field $\{\mathcal{B}(\alpha , t),\ \alpha\in\mathbb{N}_0\}$ enjoys the exponential Gibbs property. Recall that $\mathcal{B}^N_i(\alpha,t)= 4N\cdot X^N_i\left(\alpha,1+ {t}/{4N}\right)$ is defined in \eqref{eqn:scaledBfield}. Because of Lemma~\ref{lem:bridgeinterlacing}, $\{\mathcal{B}^N (\alpha,t),\ \alpha\in\mathbb{N}_0\}$ satisfies the exponential Gibbs property. Since $t$ is fixed throughout this section, for simplicity, we denote $\mathcal{B}^N(\alpha,t)$ by $\mathcal{B}^N(\alpha)$ and $\mathcal{B}(\alpha,t)$ by $\mathcal{B}^{\infty}(\alpha)$. Recall that a subsequence of $\mathcal{B}^N(\alpha)$ converges to $\mathcal{B}^\infty(\alpha)$ in distribution. We will abuse the notation and assume $\mathcal{B}^N(\alpha)$ converges to $\mathcal{B}^\infty(\alpha)$ in distribution. Besides notational changes, this inconsequential.
\vspace{\baselineskip}

Fix $k\in\mathbb{N}$ and $\alpha <\beta$ with $\alpha,\beta \in\mathbb{N}_0$. Let 
\begin{align*}
\vec{x}=( \mathcal{B}^\infty_1(\alpha),\mathcal{B}^\infty_2(\alpha),\dots ,\mathcal{B}^\infty_k(\alpha) ),\quad\vec{y}=( \mathcal{B}^\infty_1(\beta),\mathcal{B}^\infty_2(\beta),\dots ,\mathcal{B}^\infty_k(\beta) ),
\end{align*}
and
\begin{align*}
\vec{z}=( \mathcal{B}^\infty_{k+1}(\alpha+1),\mathcal{B}^\infty_{k+1}(\alpha+2),\dots ,\mathcal{B}^\infty_{k+1}(\beta-1) ).
\end{align*}
By the Cauchy interlacing theorem and the convergence of $\mathcal{B}^N,$ we have $I_{\alpha,\beta}(\vec{x},\vec{y},\vec{z})\neq \emptyset$ with probability one. Here  $I_{\alpha,\beta}(\vec{x},\vec{y},\vec{z})$ is the interlacing set defined in Definition~\ref{def:interlacingset}. Because of Lemma~\ref{lem:boundaryZ}, we have $Z_{\alpha ,\beta}(\vec{x},\vec{y},\vec{z})>0$ with probability one, where $Z_{\alpha ,\beta}(\vec{x},\vec{y},\vec{z})$ is the Lebesgue measure of $I_{\alpha ,\beta}(\vec{x},\vec{y},\vec{z})$. The goal is to show that the law of $\mathcal{B}^\infty $ is unchanged when one resamples $\mathcal{B}_j^{\infty}(\gamma)$ between $(j,\gamma)\in \llbracket 1,k \rrbracket\times \llbracket \alpha+1 ,\beta-1 \rrbracket $ according to $\mathbb{P}^{k,\llbracket \alpha+1,\beta-1 \rrbracket,\vec{x},\vec{y},\vec{z}}$ defined in Definition~\ref{def:PPP}.

\vspace{\baselineskip}

From the Skorohod representation theorem \cite[Theorem 6.7]{Bil}, there exists a probability space $(\Omega,\mathcal{F},\mathbb{P})$ on which $\mathcal{B}^N$ for $N\in\mathbb{N}\cup\{\infty\}$ are defined and almost surely $\mathcal{B}^N(\omega)\to\mathcal{B}^\infty(\omega)$ in each entry.\\[0.1cm]
\indent 
For $(j,\ell)\in \llbracket 1,k  \rrbracket\times \mathbb{N}$, let $E_j^\ell $ be a collection of independent random variables with distribution $\mathbb{P}^{1,\llbracket \alpha+1,\beta-1 \rrbracket,0,1}$. In other words, $E^\ell_j$ are distributed uniformly on $I_{\alpha,\beta}(0,1)$ and are independent among $(j,\ell)\in \llbracket 1,k  \rrbracket\times \mathbb{N}$. See Section~\ref{sec:Gibbs} for the definition of $\mathbb{P}^{1,\llbracket \alpha+1,\beta-1 \rrbracket,0,1}$ and $I_{\alpha,\beta}(0,1)$.

We define the $\ell$-th candidate of the resampling trajectory. For $N\in\mathbb{N}\cup \{\infty\}$, define
\begin{align*}
\mathcal{B}^{N,\ell}_j(\gamma)\coloneqq  \left\{ \begin{array}{cc}
\mathcal{B}^N_j(\alpha)+(\mathcal{B}^N_j(\beta)-\mathcal{B}^N_j(\alpha)) E^\ell_j(\gamma) , & \gamma\in \llbracket \alpha+1,\beta-1\rrbracket ,\\[0.1cm]
\mathcal{B}^{N}_j(\gamma),\ & \gamma\notin \llbracket \alpha+1,\beta-1\rrbracket.
\end{array}  \right.	
\end{align*}
For $N\in\mathbb{N}\cup\{\infty\}$, we {accept} the candidate resampling $\mathcal{B}^{N,\ell}$ if it satisfies the interlacing relation
\begin{align*}
\mathcal{B}^{N,\ell} (\alpha)\prec   \mathcal{B}^{N,\ell} (\alpha+1)\prec  \dots\prec   \mathcal{B}^{N,\ell} (\beta-1)\prec   \mathcal{B}^{N,\ell} (\beta).
\end{align*}
   For $N\in \mathbb{N}\cup\{\infty\}$, define $\ell(N)$ to be the minimal value of $\ell$ of which we accept $\mathcal{B}^{N,\ell}$. That is,
\begin{align*}
\ell(N)\coloneqq  \inf\{\ell\in\mathbb{N}\,|\, \mathcal{B}^{N,\ell} (\alpha)\prec   \mathcal{B}^{N,\ell} (\alpha+1)\prec  \dots\prec   \mathcal{B}^{N,\ell} (\beta-1)\prec   \mathcal{B}^{N,\ell} (\beta)\}.
\end{align*} 
Write $\mathcal{B}^{N,\textup{re}}$ for the line ensemble $ \mathcal{B}^{N,\ell(N)}$. 
\begin{figure}
\begin{tikzcd}  
\mathcal{B}^{N,\textup{re}} \arrow[rr,equal ,"(d)","(1)"'] \arrow[dd,"\textup{a.s.}","(2)'"']& & \mathcal{B}^{N } \arrow[dd,"\textup{a.s.}","(2)"']  \\
 & &\\
\mathcal{B}^{\infty,\textup{re}} \arrow[rr,equal ,"(d)","(1)'"'] & & \mathcal{B}^{\infty }
\end{tikzcd}
\caption{(1) is equivalent to the exponential Gibbs property of $\mathcal{B}^N$. The {\bf goal} is to prove (1)', which implies the exponential Gibbs property for $\mathcal{B}^\infty$. (1)' follows from the convergence in (2) and (2)'. (2) follows from the Skorohod representation theorem and (2)' is proved in Lemma \ref{lem:ellconvergence}.}\label{figure:BGP}
\end{figure}

Since $Z_{\alpha,\beta}(\vec{x},\vec{y},\vec{z})>0$ holds almost surely, we have almost surely $\ell(\infty)$ is finite. Suppose that $\ell(N)$ converges to $\ell(\infty)$  almost surely. Then each entry of $\mathcal{B}^{N,\textup{re}}$ converges to the corresponding entry of $\mathcal{B}^{\infty,\textup{re}}$ almost surely. This implies $\mathcal{B}^{N,\textup{re}}$ converges weakly to $\mathcal{B}^{\infty,\textup{re}}$. As a consequence, $\mathcal{B}^{\infty,\textup{re}}$ has the same distribution as $\mathcal{B}^{\infty}$, which is equivalent to the exponential Gibbs property. See Figure \ref{figure:BGP} for an illustration.

\begin{lemma}\label{lem:ellconvergence}
Almost surely $\ell(N)$ converges to $\ell(\infty)$.
\end{lemma}
\begin{proof}
Let $\mathsf{E}$ be the event such that the following conditions hold

\begin{enumerate}
\item $\ell(\infty)<\infty$
\item  For all $(\ell,j,\gamma)\in\mathbb{N}\times\llbracket 1,k-1 \rrbracket\times
\llbracket \alpha+1,\beta-1 \rrbracket$ 
$$\mathcal{B}^{\infty}_j(\alpha)+ (\mathcal{B}^{\infty}_j(\beta)-\mathcal{B}^{\infty}_j(\alpha)) E^\ell_j(\gamma)\neq  \mathcal{B}^{\infty}_{j+1}(\alpha)+ (\mathcal{B}^{\infty}_{j+1}(\beta)-\mathcal{B}^{\infty}_{j+1}(\alpha)) E^\ell_{j+1}(\gamma).$$
\item  For all $(\ell,\gamma)\in\mathbb{N}\times
\llbracket \alpha+1,\beta-1 \rrbracket$ 
$$\mathcal{B}^{\infty}_k(\alpha)+ (\mathcal{B}^{\infty}_k(\beta)-\mathcal{B}^{\infty}_k(\alpha)) E^\ell_k(\gamma)\neq  \mathcal{B}^{\infty}_k(\gamma).$$
\item Each entry of $\mathcal{B}^N$ converges the corresponding entry of $\mathcal{B}^\infty$.
\end{enumerate}
From the above discussion, conditions (1) and (4) hold with probability one. Conditions (2) and (3) also hold with probability one because each $E^\ell_j$ is independent of $\mathcal{B}^\infty$. In short, $\mathsf{E}$ has probability one.

We will show that when $\mathsf{E}$ occurs, $\ell(N) \to \ell(\infty)$. 
From now on we fix a realization $\omega\in \mathsf{E}$ and the constants below may depend on $\omega$. By the definition of $\ell(\infty)$,  
\begin{align*}
\mathcal{B}^{\infty,\ell(\infty)} (\alpha)\prec   \mathcal{B}^{\infty,\ell(\infty)} (\alpha+1)\prec  \dots\prec   \mathcal{B}^{\infty,\ell(\infty)} (\beta-1)\prec   \mathcal{B}^{\infty,\ell(\infty)} (\beta).
\end{align*} 
Because of condition (4), for $N$ large enough we have
\begin{align*}
\mathcal{B}^{N,\ell(\infty)} (\alpha)\prec   \mathcal{B}^{N,\ell(\infty)} (\alpha+1)\prec  \dots\prec   \mathcal{B}^{N,\ell(\infty)} (\beta-1)\prec   \mathcal{B}^{N,\ell(\infty)} (\beta).
\end{align*}
Therefore,
\begin{align*}
\limsup_{N\to\infty}\ell(N)\leq \ell(\infty). 
\end{align*}

On the other hand, because of conditions (2) and (3), for all $\ell\in\llbracket 1,\ell(\infty)-1\rrbracket$,   
\begin{align*}
\mathcal{B}^{\infty,\ell } (\alpha)\preceq   \mathcal{B}^{\infty,\ell } (\alpha+1)\preceq  \dots\preceq   \mathcal{B}^{\infty,\ell } (\beta-1)\preceq   \mathcal{B}^{\infty,\ell } (\beta) 
\end{align*} 
fails. Because this is an open condition, for $N$ large enough, we have
\begin{align*}
\mathcal{B}^{N,\ell } (\alpha)\preceq   \mathcal{B}^{N,\ell } (\alpha+1)\preceq  \dots\preceq   \mathcal{B}^{N,\ell } (\beta-1)\preceq   \mathcal{B}^{N,\ell } (\beta) 
\end{align*} 
fails as well. As a consequence,
\begin{align*}
\liminf_{N\to\infty}\ell(N)\geq \ell(\infty).
\end{align*}
Hence $\ell(N)$ converges to $\ell(\infty)$  and the proof is finished.
\end{proof} 

\section{Markov Property}\label{sec:A}
The goal of this section is to show the Markov property of the Laguerre ensemble along a time-like or a space-like path, i.e. Lemma~\ref{lem:Markov} and Lemma~\ref{lem:Markov-spacelike}. The proof is based on singular value decomposition and the unitary invariance of complex Brownian motions.  In section~\ref{sec:A.1}, we collect basic properties of singular value decomposition which we need later. Lemma~\ref{lem:Markov} is proved in Section~\ref{sec:A.2} and Lemma~\ref{lem:Markov-spacelike} is proved in Section~\ref{sec:A.3}.

\subsection{Singular Value Decomposition}\label{sec:A.1}
We consider the singular value decomposition (SVD) in this section. Using SVD, we can assign a matrix with its singular values and a pair of unitary matrices. Because unitary matrices in SVD are not unique, we consider the corresponding quotient space to obtain a one-to-one identification.
\vspace{\baselineskip}

For $N\in\mathbb{N}$ and $\alpha\in\mathbb{N}_0$, let $\mathbf{Mat}_+^{N,\alpha}$ be the space of $(N+\alpha)\times N$ complex matrices with positive and distinct singular values. Let $A\in \mathbf{Mat}_+^{N,\alpha}$ and  $ A=U_1\Sigma U_2^{*} $ be its singular value decomposition. In other words, $U_1\in \mathcal{U}(N+\alpha )$,\ $U_2\in \mathcal{U}(N )$ are unitary matrices and $\Sigma$ is a $(N+\alpha)\times N$ matrix with zero entries except on the diagonal. By requiring the entries in $\Sigma$ to be increasing from top-left to bottom-right, we can identify $\Sigma$ with an element in the Weyl chamber $\mathbb{W}^N_+$ defined in \eqref{def:Weylchamber}. Moreover, the pair $(U_1,U_2)$ is unique up to the action
\begin{align}\label{actionappendix}
(U_1',U_2')= ( U_1 R',U_2 R),  
\end{align}
with
\begin{align*}
R=\textup{diag}(e^{\mathrm{i}\theta_1},e^{\mathrm{i}\theta_2},\dots,e^{\mathrm{i}\theta_N}),\ \textup{and}\  {R}'=\left[ \begin{array}{cc}
R & 0\\
0 & Q
\end{array} \right], Q\in\mathcal{U}\left(\alpha \right).
\end{align*}
We denote $u=(U_1,U_2)$ and write $v=[u]$ for equivalent the class of $u$ under the action \eqref{actionappendix}. Let $\mathcal{U}_{N,\alpha}\coloneqq \mathcal{U}(N+\alpha)\times \mathcal{U}(N)$ and $\mathcal{V}_{N,\alpha}$ be the quotient space of $\mathcal{U}_{N,\alpha}$ by the action \eqref{actionappendix}. Let
\begin{align}\label{equ:proj}
\pi:\mathcal{U}_{N,\alpha}\to\mathcal{V}_{N,\alpha}
\end{align}
be the projection and equip $\mathcal{V}_{N,\alpha}$ with the quotient topology. From the above discussion, we have the following lemma.
\begin{lemma}\label{lem:id}
The exists a homeomorphism between $\mathbb{W}_+^N\times \mathcal{V}_{N,\alpha}$ and $\mathbf{Mat}_+^{N,\alpha}$. 
\end{lemma}
\begin{proof}
The natural map  $\mathcal{F}:\mathbb{W}^N_+\times \mathcal{U}_{N,\alpha}\to \mathbf{Mat}_+^{N,\alpha}$ defined by
\begin{align*}
\mathcal{F}(\vec{x},U_1,U_2)\coloneqq U_1\left[ \begin{array}{cccc}
x_1 & 0 & \cdots & 0\\
0 & x_2 & \cdots & 0\\
\vdots &\vdots &\ddots &\vdots \\
0 & 0 & \cdots & x_N
\end{array} \right]U_2^*.
\end{align*} 
is continuous. By the construction of $\mathcal{V}_{N,\alpha}$, $\mathcal{F}$ descends to a continuous, injective map from $\mathbb{W}^N_+\times \mathcal{V}_{N,\alpha}$ to $\mathbf{Mat}_+^{N,\alpha}$. We abuse the notation and still denote it by $\mathcal{F}$. The singular value decomposition implies $\mathcal{F}$ is surjective. Because both $\mathcal{V}_{N,\alpha}$ and $\mathbf{Mat}_+^{N,\alpha}$ are of real $2N(N+\alpha)$-dimensional, by Invariance of Domain \cite[Theorem 2B.3]{H}, $\mathcal{F}$ is a homeomorphism between $\mathbb{W}_+^N\times \mathcal{V}_{N,\alpha}$ and $\mathbf{Mat}_+^{N,\alpha}$.
\end{proof}

Next, we want to define the ``uniform" measure on $\mathcal{V}_{N,\alpha}$ which is induced from the Haar measure on $\mathcal{U}_{N,\alpha}$ and give it a characterization.

 Let $\mathfrak{M}(\mathcal{U}_{N,\alpha})$ and $\mathfrak{M}(\mathcal{V}_{N,\alpha})$ be the space of signed, finite Borel measures on $\mathcal{U}_{N,\alpha}$ and $\mathcal{V}_{N,\alpha}$ respectively. Let $\mathfrak{M}_{\textup{inv}}(\mathcal{U}_{N,\alpha})$ be the subsapce of $\mathfrak{M}(\mathcal{U}_{N,\alpha})$ consisting of measures which are invariant under the action \eqref{actionappendix}.
\begin{lemma}\label{lem:iso}
The projection \eqref{equ:proj} induces an isomorphism
\begin{equation}\label{equ:iso}
\pi_*:\mathfrak{M}_{\textup{inv}}(\mathcal{U}_{N,\alpha})\to \mathfrak{M}(\mathcal{V}_{N,\alpha}).
\end{equation}
\end{lemma}
\begin{proof}
We start by showing $\pi_*$ is injective on $\mathfrak{M}_{\textup{inv}}(\mathcal{U}_{N,\alpha})$. Suppose $\mu\in \ker(\pi_*)\cap \mathfrak{M}_{\textup{inv}}(\mathcal{U}_{N,\alpha})$. Let $\mu_+$ and $\mu_-$ be the Hahn decomposition of $\mu$. In other words, $\mu_\pm$ are non-negative measures and $\mu_{\pm}(E)=\pm\mu(\Omega_\pm\cap E)$ for some Borel sets $\Omega_\pm$ with $\Omega_+\cup\Omega_-=\mathcal{U}_{N,\alpha}$. Because $\mu\in \mathfrak{M}_{\textup{inv}}(\mathcal{U}_{N,\alpha})$, we can pick both $\Omega_{\pm}$ to be invariant under the action \eqref{actionappendix}. In particular, $\pi^{-1}(\pi(\Omega_{\pm}))=\Omega_{\pm}$. Because $\mu\in \ker(\pi_*)$, we have $\mu(\Omega_{\pm})=(\pi_*\mu) (\pi(\Omega_{\pm}))=0$. Therefore, $\mu=0$.
\vspace{\baselineskip}

Next, we show $\pi_*$ is surjective. Given $\nu\in \mathfrak{M}(\mathcal{V}_{N,\alpha})$, we can construct $\mu\in \mathfrak{M}_{\textup{inv}}(\mathcal{U}_{N,\alpha})$ such that $\pi_*\mu=\nu$ as follows. Note that $\pi:\mathcal{U}_{N,\alpha}\to\mathcal{V}_{N,\alpha}$ is a principal $\mathcal{U}(1)^N\times \mathcal{U}(\alpha)$-bundle. Let $\mu_0$ be the Haar measure on $\mathcal{U}(1)^N\times \mathcal{U}(\alpha)$ with total mass one. Let $\mathcal{U}_{N,\alpha}\supset E\cong \Omega\times \mathcal{U}(1)^N\times \mathcal{U}(\alpha)$ be a local trivialization with $\Omega\subset \mathcal{V}_{N,\alpha}$ and let $h_E:\Omega\times \mathcal{U}(1)^N\times \mathcal{U}(\alpha)\to E$ be the identification map. Define on $E$ a measure $\mu_E\coloneqq  (h_E)_*(\nu|_{\Omega}\times \mu_0)$. By the uniqueness of the Haar measure, $\mu_E$ does not depend on the trivialization. Putting $\mu_E$ together using a partition of unity gives the desired $\mu\in \mathfrak{M}_{\textup{inv}}(\mathcal{U}_{N,\alpha})$. Therefore, $\pi_*$ is surjective. This shows $\pi_*$ is an isomorphism. 
\end{proof}
We write 
\begin{equation}\label{equ:iso-inv}
\pi^{-1}_*: \mathfrak{M}(\mathcal{V}_{N,\alpha}) \to\mathfrak{M}_{\textup{inv}}(\mathcal{U}_{N,\alpha}) 
\end{equation}
for the inverse map of \eqref{equ:iso}.
\vspace{\baselineskip}

Let $\mu_{N,\alpha}\in \mathfrak{M}_{\textup{inv}}(\mathcal{U}_{N,\alpha})$ be the Haar measure with total mass one. We define the ``uniform" measure on $\mathcal{V}_{N,\alpha}$ to be
\begin{equation}\label{def:Haar}
 \nu_{N,\alpha}\coloneqq \pi_*\mu_{N,\alpha}. 
\end{equation}  

We now give a characterization of $\nu_{N,\alpha}$ which is similar to the one for a Haar measure. To do so, we consider the left multiplication as follows. For any ${u}=({U}_1,{U}_2)\in \mathcal{U}_{N,\alpha}$, let $L_u:\mathcal{U}_{N,\alpha}\to \mathcal{U}_{N,\alpha}$ be the left multiplication. Because \eqref{actionappendix} is a right action, $L_u$ descends to a map $L_u:\mathcal{V}_{N,\alpha}\to \mathcal{V}_{N,\alpha}$. We write $ {u}_*$ for the induced map on $\mathfrak{M}_{\textup{inv}}(\mathcal{U}_{N,\alpha})$ or on $\mathfrak{M}(\mathcal{V}_{N,\alpha})$. Because $\pi\circ L_u=L_u\circ\pi$, we have $\pi_* \circ u_*=u_*\circ \pi_*$ on $\mathfrak{M}(\mathcal{U}_{N,\alpha})$. Moreover, because of Lemma~\ref{lem:iso}, on $\mathfrak{M}(\mathcal{V}_{N,\alpha})$ we have 
\begin{align}\label{equ:commute}
 u_*\circ \pi^{-1}_*=\pi^{-1}_* \circ u_*. 
\end{align}

\begin{lemma}\label{lem:Haar}
Let $\nu\in \mathfrak{M}(\mathcal{V}_{N,\alpha})$ be a finite Borel measure on $\mathcal{V}_{N,\alpha}$. Suppose for all $u\in \mathcal{U}_{N,\alpha}$ it holds that $u_*\nu=\nu.$ Then $\nu=c\nu_{N,\alpha}$ with $c=\nu(\mathcal{V}_{N,\alpha})$.
\end{lemma}
\begin{proof}
Let $\mu=\pi_*^{-1}\nu$. From \eqref{equ:commute} we have $u_*\mu=u_*\pi_*^{-1}\nu=\pi_*^{-1}u_*\nu=\pi_*^{-1}\nu=\mu.$ By the uniqueness of the Haar measure, $\mu=c \mu_{N,\alpha}$ for some $c\in\mathbb{R}$. This implies $\nu=c \nu_{N,\alpha}$. Then $c=\nu(\mathcal{V}_{N,\alpha})$ follows easily.

\end{proof} 
\subsection{Proof of Lemma~\ref{lem:Markov}}\label{sec:A.2}
In this section, we aim to prove Lemma~\ref{lem:Markov}. Fix $m\in\mathbb{N}$ and a time-like path $\left\{\left(\alpha^{(k)},t^{(k)}\right)\right\}_{k=1}^m$ in $\mathbb{N}_0\times (0,\infty)$ as in Definition \ref{def:timelike}. Let
\begin{equation}\label{def:Ak-time}
 A_{(k)}=A^{ N}\left(\alpha^{(k)} ,t^{(k)}\right).
\end{equation}
With probability one, $A_{(k)}$ has positive and distinct singular values for all $k\in \llbracket 1,m\rrbracket $. From Lemma~\ref{lem:id}, $\{A_{(k)}\}_{k=1}^m$ can be identified as a sequence random variables $\{(\vec{Y}^{(k)},  V_{(k)})\}_{k=1}^m$ taking values in $\mathbb{W}^N_+\times\mathcal{V}_{N,\alpha^{(k)}}$. Note that $\vec{Y}^{(k)}$ is related to $\vec{X}^{(k)}$ defined in \eqref{def:Xk-timelike} by $X^{(k)}_j=\left( Y^{(k)}_j\right)^2$ for all $j\in\llbracket 1,N \rrbracket$.
\vspace{\baselineskip}

It is straightforward to see that $A_{(k)}$ is a Markov process in $k$. Hence $(\vec{Y}^{(k)},  V_{(k)})$ is also Markov. We write 
$$F_{k}\left(\left( \vec{y}^{(k)},  {v}_{(k)}\right);d\left(\vec{y}^{(k+1)},  {v}_{(k+1)}\right)\right) $$
for the Markov kernel. The next lemma is a consequence of the unitary invariance of complex Brownian motions. 
\begin{lemma}\label{clm:key}
\begin{align}\label{key}
\int_{\mathcal{V}_{N,\alpha^{(k+1)}} } F_{k}\left(\left( \vec{y}^{(k)},  {v}_{(k)}\right);d\left(\vec{y}^{(k+1)},  {v}_{(k+1)}\right)\right)=G_{k}\left(\vec{x}^{(k)};d\vec{x}^{(k+1)}\right).
\end{align}
In other words, \eqref{key} does not depend on $ {v}_{(k)}$.
\end{lemma}
 Now we assume Lemma~\ref{clm:key} holds and prove Lemma~\ref{lem:Markov}.
\begin{proof}[Proof of Lemma~\ref{lem:Markov}]
Denote by $\tilde{G}_{1}\left(d\vec{y}^{(1)}\right)$ the density of $\vec{Y}^{(1)}$. From Lemma~\ref{clm:key}, the joint density of $(\vec{Y}^{(1)}, \vec{Y}^{(2)},\dots,\vec{Y}^{(m)})$ is given by
\begin{align*}
\tilde{G}_{1}\left(d\vec{y}^{(1)}\right)\prod_{k=1}^{m-1}G_{k}\left(\vec{y}^{(k)};d\vec{y}^{(k+1)}\right).
\end{align*}
This implies $\vec{Y}^{(k)}$ is a Markov chain. As a result, $\vec{X}^{(k)}$ is a Markov as well.
\end{proof}   

In the rest of the section, we prove Lemma~\ref{clm:key} based on the unitary invariance of Brownian motions. In the following, we consider two special cases. In case 1, we assume $\alpha^{(k+1)}=\alpha^{(k )} $ and $t^{(k+1)}>t^{(k )}$. In case 2, we assume $\alpha^{(k+1)}>\alpha^{(k )} $ and $t^{(k+1)}=t^{(k )}$. The general case follows by combining the two. 
\begin{proof}[Proof of Lemma~\ref{clm:key}, case 1]
Suppose $\alpha^{(k)}=\alpha^{(k+1)}$ and $t^{(k)}<t^{(k+1)}$. To simplify the notation, we denote $\alpha^{(k)}=\alpha^{(k+1)}=\alpha$, $t^{(k )}=t$, and  $t^{(k+1)}=t'$. Moreover, we write $A$ for $A_{(k)}$, $A'$ for $A_{(k+1)}$ and $ F_{k}\left(\left( \vec{x},  v \right);d\left(\vec{x}',  {v}'\right)\right)$ for the Markov kernel. From \eqref{def:A} and \eqref{def:Ak-time},
\begin{align}\label{equ:A-2-case1-1}
A'=A+B_k,\ B_{k}=\left\{ B_{ij}\left( t'-t \right)\right\}_{1\leq i\leq N+\alpha, 1\leq j\leq N}.  
\end{align}
Here $B_{k}$ is independent of $A$. Let
\begin{equation}\label{equ:A-2-case1-2}
 A=U_1\Sigma U_2^{*}\ \textup{  and}\  A'=U'_1\Sigma' (U'_2)^{*} 
\end{equation}  
be the singular value decomposition of $A$ and $A'$ respectively. Fix an arbitrary $\tilde{u}=(\tilde{U}_1,\tilde{U}_2)\in \mathcal{U}_{N,\alpha}=\mathcal{U}(N+\alpha)\times\mathcal{U}(N)$. Multiplying \eqref{equ:A-2-case1-1} by $\tilde{U}_1$ from the left, by $\tilde{U}^{*}_2$ from the right, and using \eqref{equ:A-2-case1-2}, we get  
\begin{align}\label{equ:A-2-case1-3}
\left(\tilde{U}_1U'_1 \right)\Sigma' \left(\tilde{U}_2U'_
2 \right)^{*}= \left(\tilde{U}_1U_1 \right)\Sigma \left(\tilde{U}_2U_2 \right)^{*}+\tilde{U}_1B_k\tilde{U}^{*}_2.
\end{align}
Because $\tilde{U}_1B_k\tilde{U}^{*}_2\overset{(d)}{=}B_k,$ we deduce from \eqref{equ:A-2-case1-3} that 
\begin{align}\label{equ:Brownian_inv}
\tilde{u}_* F_k((\vec{x},v );d(\vec{x}',v'))=F_k((\vec{x},L_{\tilde{u}} v);d(\vec{x}',v')).
\end{align}

We are ready to prove \eqref{key}. Fix  $(\vec{x},v)\in\mathbb{W}^+_N\times\mathcal{V}_{N,\alpha}$. Let $u=(U_1,U_2)\in v$ be a representation of the class $v$ and let $e=[(I_{N+\alpha},I_N)]$ be the class of the pair of identity matrices. Clearly $L_u e=v.$ From \eqref{equ:Brownian_inv},
\begin{align*}
\int_{\mathcal{V}_{N,\alpha}} F_k((\vec{x}, v);d(\vec{x}',v'))=&\int_{\mathcal{V}_{N,\alpha}} F_k((\vec{x}, L_u e );d(\vec{x}',v'))= \int_{\mathcal{V}_{N,\alpha}} u_* F_k((\vec{x}, e);d(\vec{x}',v'))\\
=&\int_{\mathcal{V}_{N,\alpha}} F_k((\vec{x}, e);d(\vec{x}',v'))=:G_k(\vec{x};d\vec{x}').
\end{align*}
This proves \eqref{key} for case 1.
\end{proof}  

\begin{proof}[Proof of Lemma~\ref{clm:key}, case 2]
Suppose $\alpha^{(k)}<\alpha^{(k+1 )} $ and $t^{(k)}=t^{(k +1)}$. To simplify the notation, we denote $t^{(k)}=t^{(k+1 )}=t$, $\alpha^{(k)}=\alpha$ and $\alpha^{(k+1)}=\alpha'$. Moreover, we write $A$ for $A_{(k)}$, $A'$ for $A_{(k+1)}$ and $ F_{k}\left(\left( \vec{x},  v \right);d\left(\vec{x}',  {v}'\right)\right)$ for the Markov kernel. Recall that 
\begin{align}\label{equ:A-2-case2-1}
A'= \left[ \begin{array}{c}
A \\
\tilde{B}_{k}
\end{array}
\right] ,\ \tilde{B}_{k}=\left\{ B_{ij}\left(t\right)\right\}_{N+\alpha+1\leq i\leq N+\alpha' , 1\leq j\leq N}.  
\end{align}  
 Here $\tilde{B}_{k}$ is independent of $A$. Let
\begin{equation}\label{equ:A-2-case2-2}
 A=U_1\Sigma U_2^{*}\ \textup{  and}\  A'=U'_1\Sigma' (U'_2)^{*} 
\end{equation}  
be the singular value decomposition of $A$ and $A'$ respectively. Fix an arbitrary $\tilde{u}=(\tilde{U}_1,\tilde{U}_2)\in \mathcal{U}_{N,\alpha}=\mathcal{U}(N+\alpha)\times\mathcal{U}(N)$.  Let
\begin{align*}
\hat{U}_1=\left[ \begin{array}{cc}
\tilde{U}_1 & 0\\
0 & I_{\alpha'-\alpha}
\end{array} \right]\in\mathcal{U}(N+\alpha'),
\end{align*}
and $\hat{u}=(\hat{U}_1,\tilde{U}_2)\in \mathcal{U}_{N,\alpha'}.$ Multiplying \eqref{equ:A-2-case2-1} by $\hat{U}_1$ from the left, by $\tilde{U}^{*}_2$ from the right and using \eqref{equ:A-2-case2-2}, we get  
\begin{align*}
\left(\hat{U}_1U'_1 \right)\Sigma' \left(\tilde{U}_2U'_
2 \right)^{*}= \left[ \begin{array}{c}
\left(\tilde{U}_1U_1 \right)\Sigma \left(\tilde{U}_2U_2 \right)^{*} \\
 \tilde{B}_{k}\tilde{U}_2^{*}
\end{array}
\right].
\end{align*}
 Because $\tilde{B}_{k} \tilde{U}_2^{*}\overset{(d)}{=}\tilde{B}_{(k)},$ we deduce
 \begin{align}\label{equ:Brownian_inv_2}
\hat{u}_* F_k((\vec{x},v );d(\vec{x}',v'))=F_k((\vec{x},L_{\tilde{u}} v);d(\vec{x}',v')).
\end{align}

We are ready to prove \eqref{key}. Fix  $(\vec{x},v)\in\mathbb{W}^+_N\times\mathcal{V}_{N,\alpha}$. Let $u=(U_1,U_2)\in v$ be a representation of the class $v$. Recall that $e=[(I_{N+\alpha},I_N)]$ is the class of the pair of identity matrices. Clearly $L_u e=v.$ From \eqref{equ:Brownian_inv_2},
\begin{align*}
\int_{\mathcal{V}_{N,\alpha'}} F_k((\vec{x}, v);d(\vec{x}',v'))=&\int_{\mathcal{V}_{N,\alpha'}} F_k((\vec{x}, L_u e );d(\vec{x}',v'))= \int_{\mathcal{V}_{N,\alpha'}} \hat{u}_* F_k((\vec{x}, e);d(\vec{x}',v'))\\
=&\int_{\mathcal{V}_{N,\alpha'}} F_k((\vec{x}, e);d(\vec{x}',v'))=:G_k(\vec{x};d\vec{x}').
\end{align*}
This proves \eqref{key} for case 2.
\end{proof}  
\subsection{Proof of Lemma~\ref{lem:Markov-spacelike}}\label{sec:A.3}

The goal of this section is to prove Lemma~\ref{lem:Markov-spacelike}. Fix $m\in\mathbb{N}$ and a space-like sequence $\left\{\left(\alpha^{(k)},t^{(k)}\right)\right\}_{k=1}^m$ in $\mathbb{N}_0\times (0,\infty)$ as in Definition \ref{def:spacelike}. Let 
\begin{equation}\label{def:Ak-space}
 A_{(k)}=A^{N }\left( \alpha^{(k)} ,t^{(k)}\right). 
\end{equation}

 We can consider $\{A_{(k)}\}_{k=1}^m$ as a sequence random variables $\{(\vec{Y}^{(k)},  V_{(k)})\}_{k=1}^m$ in $\mathbb{W}^N_+\times\mathcal{V}_{N,\alpha^{(k)}}$ by Lemma~\ref{lem:id}.
\vspace{\baselineskip}

$A_{(k)}$ is clearly a Markov process in $k$ and we denote 
$$\overline{F}_{(k)}\left(\left( \vec{y}^{(k)},  {v}_{(k)}\right);d\left(\vec{y}^{(k+1)},  {v}_{(k+1)}\right)\right) $$
its Markov kernel. We set $\nu_k=\nu_{N,\alpha^{(k)}}$ as in \eqref{def:Haar}. From the unitary invariance of complex Brownian motions, we have the following lemma.
\begin{lemma}\label{clm:keyspace}
\begin{equation}\label{keyspace}
\begin{split}
&\int_{\mathcal{V}_{N,\alpha^{(k)}}} d\nu_{k}\left( v_{(k)} \right)\,  \overline{F}_{(k)}\left(\left( \vec{y}^{(k)},  {v}_{(k)}\right);d\left(\vec{y}^{(k+1)},  {v}_{(k+1)}\right)\right)=\overline{G}_{k}\left(\vec{y}^{(k)};d\vec{y}^{(k+1)}\right)d\nu_{k+1}\left( v_{(k+1)} \right).
\end{split}
\end{equation}

\end{lemma}
 Assuming Lemma~\ref{clm:keyspace} we can prove Lemma~\ref{lem:Markov-spacelike}.
 \begin{proof}[Proof of Lemma~\ref{lem:Markov-spacelike}] We can write the density of $(\vec{Y}^{(1)},V_{(1)}) $ as $$ {G}'_{(1)}\left(d\vec{y}^{(1)}\right)d\nu_{1}\left( v_{(1)} \right)$$ 
 by unitary invariance of complex Brownian motions.
Besides, from Lemma~\ref{clm:keyspace}, the joint density of $\vec{Y}^{(1)}, \vec{Y}^{(2)},\dots,\vec{Y}^{(m)} $ is given by
\begin{align*}
 {G}'_{(1)}\left(d\vec{y}^{(1)}\right)\prod_{k=1}^{m-1}\overline{G}_{(k)}\left(\vec{y}^{(k)};d\vec{y}^{(k+1)}\right)
\end{align*}
which implies that $\vec{Y}^{(k)}$ is a Markov chain and thus $\vec{X}^{(k)}$ is also Markovian.
 \end{proof}
Similarly as in the previous subsection, we now prove Lemma~\ref{clm:keyspace} and we split the proof in two cases. We first assume that $\alpha^{(k)}=\alpha^{(k+1)} $ and $t^{(k )}<t^{(k+1 )}$ and then assume that $\alpha^{(k )}>\alpha^{(k+1 )} $ and $t^{(k )}=t^{(k+1)}$. The case of general space-like paths follows by combining these two cases.
\begin{proof}[Proof of Lemma~\ref{clm:keyspace}, case 1]
 We suppose that $\alpha^{(k)}=\alpha^{(k+1)}$ and $t^{(k)}<t^{(k+1)}$. As in the proof of Lemma~\ref{clm:key}, we set $\alpha^{(k)}=\alpha^{(k+1 )} =\alpha$, $t^{(k )}=t$ and $t^{(k+1)}=t'$, and we write $A$ for $A_{(k)}$, $A'$ for $A_{(k+1)}$, $\nu$ for $\nu_k$ and $ \overline{F}_{k}\left(\left( \vec{x},  v \right);d\left(\vec{x}',  {v}'\right)\right)$ for the Markov kernel. We want to prove that
 \begin{equation}\label{keyspace2}
\begin{split}
&\int_{\mathcal{V}_{N,\alpha}} d\nu \left( v  \right)\,  \overline{F}_{ k }\left(\left( \vec{x} ,  {v} \right);d\left(\vec{x}',  {v}'\right)\right)=\overline{G}_{k}\left(\vec{x} ;d\vec{x}'\right)d\nu \left( v' \right).
\end{split}
\end{equation}
By Lemma~\ref{lem:Haar}, it suffices to show that for all $ {u}\in \mathcal{U}_{N,\alpha}$,
 \begin{equation}\label{keyspace3}
\begin{split}
&\int_{\mathcal{V}_{N,\alpha} } d\nu \left( v  \right)\,   {u}_*\overline{F}_{ k }\left(\left( \vec{x} ,  {v} \right);d\left(\vec{x}',  {v}'\right)\right)=\int_{\mathcal{V}_{N,\alpha}} d\nu \left( v  \right)\,  \overline{F}_{ k }\left(\left( \vec{x} ,  {v} \right);d\left(\vec{x}',  {v}'\right)\right).
\end{split}
\end{equation}
Note that we have that, from \eqref{equ:Brownian_inv},  for all $ {u}\in \mathcal{U}_{N,\alpha}$,
 \begin{align}\label{equ:Brownian_inv_space}
 {u}_* \overline{F}_k((\vec{x},v );d(\vec{x}',v'))=\overline{F}_k((\vec{x},L_{u}  v);d(\vec{x}',v')).
\end{align}
Thus,
\begin{align*}
\int_{\mathcal{V}_{N,\alpha }} d\nu \left( v  \right)\,   {u}_*\overline{F}_{ k }\left(\left( \vec{x} ,  {v} \right);d\left(\vec{x}',  {v}'\right)\right)=&\int_{\mathcal{V}_{N,\alpha }} d\nu \left( v  \right)\,  \overline{F}_{ k }\left(\left( \vec{x} ,  L_u {v} \right);d\left(\vec{x}',  {v}'\right)\right)\\
=&\int_{\mathcal{V}_{N,\alpha}} d\nu \left( v  \right)\,  \overline{F}_{ k }\left(\left( \vec{x} , {v} \right);d\left(\vec{x}',  {v}'\right)\right)
\end{align*} 
where we used $u_*\nu=\nu$ in the second equality. and the proof of the first case is finished.
\end{proof}
\begin{proof}[Proof of Lemma~\ref{clm:keyspace}, case 2]

Suppose that $\alpha^{(k)}>\alpha^{(k+1 )} $ and $t^{(k)}=t^{(k+1 )}$. To simplify the notation, we denote $t^{(k)}=t^{(k+1 )}=t$, $\alpha^{(k )}=\alpha$ and $\alpha^{(k+1)}=\alpha'$. Moreover, we write $A$ for $A_{(k)}$, $A'$ for $A_{(k+1)}$, $\nu$ for $\nu_k$, $\nu'$ for $\nu_{k+1}$ and $ \overline{F}_{k}\left(\left( \vec{x},  v \right);d\left(\vec{x}',  {v}'\right)\right)$ for the Markov kernel. We aim to show that
 \begin{equation}\label{keyspace4}
\begin{split}
&\int_{\mathcal{V}_{N,\alpha}} d\nu \left( v  \right)\,  \overline{F}_{ k }\left(\left( \vec{x} ,  {v} \right);d\left(\vec{x}',  {v}'\right)\right)=\overline{G}_{k}\left(\vec{x} ;d\vec{x}'\right)d\nu' \left( v' \right).
\end{split}
\end{equation}  
It is straightforward to see that \eqref{keyspace4} is equivalent to \eqref{keyspace}. From Lemma~\ref{lem:Haar}, it suffices to show that for all $ \tilde{u}\in \mathcal{U}_{N,\alpha'}$,
 \begin{equation}\label{keyspace5}
\begin{split}
&\int_{\mathcal{V}_{N,\alpha }} d\nu \left( v  \right)\,  \tilde{u}_*\overline{F}_{ k }\left(\left( \vec{x} ,  {v} \right);d\left(\vec{x}',  {v}'\right)\right)=\int_{\mathcal{V}_{N,\alpha }} d\nu \left( v  \right)\,  \overline{F}_{ k }\left(\left( \vec{x} ,  {v} \right);d\left(\vec{x}',  {v}'\right)\right).
\end{split}
\end{equation}
From \eqref{def:A} and \eqref{def:Ak-space}, $A'$ is obtained by eliminating the last $\alpha-\alpha'$ rows in $A$. That is,
\begin{align}\label{equ:A-3-case2-1}
 \left[ \begin{array}{c}
A' \\
{C}_{k}
\end{array}
\right]= A   ,\ {C}_{k}\ \textup{is a}\ (\alpha-\alpha')\times N\ \textup{matrix}.  
\end{align} 
Here we write the last $\alpha-\alpha'$ rows of $A$ as $C_k$ to emphasize that they are inconsequential for the dynamics from $A$ to $A'$. In particular, the following argument \textbf{does not} rely on the distribution of $C_k$. Let
\begin{equation}\label{equ:A-3-case2-2}
 A=U_1\Sigma U_2^{*}\ \textup{  and}\  A'=U'_1\Sigma' (U'_2)^{*} 
\end{equation} 
be the singular value decomposition of $A$ and $A'$ respectively. Fix an arbitrary $ \tilde{u}=( \tilde{U}_1, \tilde{U}_2)\in \mathcal{U}_{N,\alpha'}=\mathcal{U}(N+\alpha')\times\mathcal{U}(N)$. Let
\begin{align*}
\hat{U}_1=\left[ \begin{array}{cc}
 \tilde{U}_1 & 0\\
0 & I_{\alpha-\alpha'}
\end{array} \right]\in\mathcal{U}(N+\alpha),
\end{align*}
and $\hat{u}=( \hat{U}_1, \tilde{U}_2)\in \mathcal{U}_{N,\alpha}.$ Multiplying \eqref{equ:A-3-case2-1} by $\hat{U}_1$ from the left, by $ \tilde{U}^{*}_2$ from the right and using \eqref{equ:A-3-case2-2}, we get  
 \begin{align}\label{equ:A-3-case2-3}
 \left[ \begin{array}{c}
\left( \tilde{U}_1U'_1 \right)\Sigma' \left( \tilde{U}_2U'_
2 \right)^{*} \\
{C}_{k}\tilde{U}_2^{*}
\end{array}
\right]= \left(\hat{U}_1U_1 \right)\Sigma \left(\tilde{U}_2U_
2 \right)^{*} .  
\end{align} 
From \eqref{equ:A-3-case2-3}, we derive
 \begin{align}\label{equ:Brownian_inv_space2}
 \tilde{u}_* \overline{F}_k((\vec{x},v );d(\vec{x}',v'))=\overline{F}_k((\vec{x},L_{\hat{u}}  v);d(\vec{x}',v')).
\end{align}
As a result,
 \begin{equation*} 
\begin{split}
 \int_{\mathcal{V}_{N,\alpha }} d\nu \left( v  \right)\,   {u}_*\overline{F}_{ k }\left(\left( \vec{x} ,  {v} \right);d\left(\vec{x}',  {v}'\right)\right)=&\int_{\mathcal{V}_{N,\alpha }} d\nu \left( v  \right)\,  \overline{F}_{ k }\left(\left( \vec{x} ,  L_{\hat{u}} {v} \right);d\left(\vec{x}',  {v}'\right)\right)\\
=&\int_{\mathcal{V}_{N,\alpha }} d\nu \left( v  \right)\,  \overline{F}_{ k }\left(\left( \vec{x} , {v} \right);d\left(\vec{x}',  {v}'\right)\right).
\end{split}
\end{equation*}
We have used $\hat{u}_*\nu=\nu$ in the second equality. This proves \eqref{keyspace5} and hence \eqref{keyspace}. The proof for case 2 is finished.
\end{proof}

\begin{appendix}

\section{Bessel Function}\label{sec:B}
In Section \ref{sec:B.0}, we collect basic properties of the Bessel function $J_\alpha(z)$ and introduce the Hankel transformation. Then we prove Lemma~\ref{lem:Qintegral} and Lemma~\ref{lem:Qintegral-s} in Sections~\ref{sec:B.1} and \ref{sec:B.2} respectively. 
\subsection{Bessel Function and Hankel transform}\label{sec:B.0}
Throughout this section, we assume $\alpha\in\mathbb{N}_0$ even though many of the results hold for a wider range of $\alpha$. Interested readers can find those generalizations in the references provided with the statements.
\vspace{\baselineskip}

From \cite[(9.1.10), page 360]{AS}, the Bessel function $J_\alpha(z)$ has the following expansion
\begin{equation}\label{equ:Jexpansion}
J_\alpha(z)=\left( \frac{z}{2} \right)^\alpha\sum_{k=0}^\infty \frac{(-4^{-1}z^2)^k}{k!\Gamma(\alpha+k+1)}.
\end{equation}
In particular, for $g_\alpha(z)\coloneqq z^{-\alpha/2}J_{\alpha}(2z^{1/2})$ defined in \eqref{def:galpha}, we have the expansion
\begin{align}\label{equ:gexpansion}
g_\alpha(z)=\sum_{k=0}^\infty\frac{(-z)^k}{k!\Gamma(\alpha+k+1)}.
\end{align}
It is easy to check that $g_\alpha(z)$ is an entire function.
\vspace{\baselineskip}
 
Next, we record the asymptotics of $J_\alpha(z)$ for $z$ close to zero or infinity. From \eqref{equ:Jexpansion}, as $z$ goes to zero, the asymptotics of $J_\alpha(z)$ is 
\begin{align}\label{equ:asymptoticzero}
J_\alpha(z)=\frac{(2^{-1}z)^\alpha}{\Gamma(\alpha+1)}(1+O(z^2)). 
\end{align}
As $z$ goes to infinity, the asymptotics of $J_\alpha(z)$ is given by \cite[(9.2.1), page 364]{AS}
\begin{align}\label{equ:asymptoticinfinity}
J_\alpha(z)=\sqrt{\frac{2}{\pi z}}\left( \cos(z-2^{-1}\alpha\pi-4^{-1}\pi)+O(z^{-1}) \right).
\end{align}

Now we introduce the \textit{Hankel transform}. The Hankel transform of order $\alpha$ of a function $f(u)$ is defined by \cite[page 214]{T} 
\begin{align}\label{def:Hankel}
H_\alpha[f](z)\coloneqq \int_0^\infty \sqrt{zu}J_\alpha(zu)f(u)\, du,\ z\in \mathbb{R}_+=(0,\infty).
\end{align}
$H_\alpha$ is an automorphism on $L^2(\mathbb{R}_+)$ \cite[page 221, Theorem 129]{T}. That is, for any $f(u),g(u)\in L^2(\mathbb{R}_+)$, the Parseval identity holds. 
\begin{equation}\label{equ:Parseval}
\int_0^\infty dz\, H_\alpha [f](z)H_\alpha [g](z)   =\int_0^\infty du\, f(u)g(u).
\end{equation}

\subsection{Proof of Lemma~\ref{lem:Qintegral}}\label{sec:B.1}
In this section, we prove Lemma~\ref{lem:Qintegral}. We start by recalling relevant definitions.
\begin{align}
T_{\alpha}((t,x);(s,y))\coloneqq  &(s-t)^{-1}(y/x)^{\alpha/2}e^{-(x+y)/(s-t)} I_\alpha \left( \frac{2\sqrt{xy} }{s-t} \right),\label{def:T-app}\\
W_t((\alpha,x);(\beta,y))\coloneqq &\frac{1}{\Gamma(\beta- \alpha)}\frac{(y-x)^{\beta- \alpha-1}}{t^{ \beta- \alpha }}e^{-(y-x)/t} \mathbbm{1}(x\leq y).\label{def:W-app}
\end{align}
\begin{equation}\label{equ:Qform-app}
Q((\alpha,t,x);(\beta,s,y))=\left\{ \begin{array}{cc}
T_\alpha((t,x);(s,y)), &\textup{if}\ \alpha=\beta\ \textup{and}\ t<s,\\
W_t((\alpha,x);(\beta,y) ), &\textup{if}\ \alpha<\beta\ \textup{and}\ t=s,\\
\int_0^\infty dz\, T_\alpha((t,x);(s,z))W_s((\alpha,z);(\beta,y) ), &\textup{if}\ \alpha<\beta\ \textup{and}\ t<s.
\end{array} \right.
\end{equation}

We record Lemma~\ref{lem:Qintegral} as Lemma~\ref{lem:Qintegral-app} below.
\begin{lemma}\label{lem:Qintegral-app}
Let $(\alpha,t)\prec_{\textup{t}} (\beta,s)$ be time-like pairs in $ \mathbb{N}_0\times (0,\infty)$ as in Definition \ref{def:timelike} and $x,y\in (0,\infty)$. It holds that
\begin{multline}\label{equ:Qintegral-app}
Q((\alpha,t,x);(\beta,s,y))
\\= (st)^{-(\beta-\alpha)/2} e^{-y/s+x/t} x^{-\alpha/2}y^{ \beta/2}\int_0^\infty du\,  e^{-(s-t)u}u^{-(\beta-\alpha)/2}J_{\alpha}(2\sqrt{t^{-1}s xu})J_{\beta}(2\sqrt{s^{-1}t yu}).
\end{multline} 
\end{lemma}
From \eqref{equ:Qform-app}, we split \eqref{equ:Qintegral-app} into three cases.
\begin{lemma}\label{lem:T}
For any $t<s$ with $t,s\in (0,\infty) $, $\alpha\in\mathbb{N}_0$ and $x,y\in (0,\infty)$, it holds that
\begin{align}\label{equ:TQ}
{T}_{\alpha}((t,x);(s,y) )=& e^{-y/s+x/t}(y/x)^{\alpha/2} \int_0^\infty du\,  e^{-(s-t)u} J_{\alpha}(2\sqrt{t^{-1}s xu})J_{\alpha}(2\sqrt{s^{-1}t yu}).
\end{align}
\end{lemma}
\begin{lemma}\label{lem:W}
For any $t\in (0,\infty)$, $\alpha<\beta$ with $\alpha,\beta\in\mathbb{N}_0$ and $x,y\in (0,\infty)$, it holds that
\begin{align}\label{equ:WQ}
{W}_t((\alpha,x);(\beta,y))=&t^{-(\beta-\alpha) } e^{-(y-x)/t } y^{\beta/2}x^{-\alpha/2}   \int_0^\infty  du\, u^{-(\beta-\alpha)/2}J_{\alpha}(2\sqrt{  xu})J_{\beta}(2\sqrt{  yu}) .
\end{align}
\end{lemma}
\begin{lemma}\label{lem:Q}
For any $t<s$ with $t,s\in (0,\infty)$, $ \alpha<\beta$ with $\alpha,\beta\in\mathbb{N}_0$ and $x,y\in (0,\infty)$, it holds that
\begin{multline*}
\int_0^\infty dz\, T_\alpha((t,x);(s,z) )W_s((\alpha,z);(\beta,y) )\\=(st)^{-(\beta-\alpha)/2} e^{-y/s+x/t} x^{-\alpha/2}y^{ \beta/2}
 \int_0^\infty du\,  e^{-(s-t)u}u^{-(\beta-\alpha)/2}J_{\alpha}(2\sqrt{t^{-1}s xu})J_{\beta}(2\sqrt{s^{-1}t yu}).
\end{multline*}
\end{lemma}

\begin{proof}[Proof of Lemma~\ref{lem:T}]
Using a change of variables,
\begin{align*}
u=\bar{x}^2,\ s-t=\bar{\rho}^2,\ 2\sqrt{t^{-1}sx}=\bar{\alpha},\ 2\sqrt{s^{-1}ty}=\bar{\beta},\ \alpha=\bar{p},
\end{align*} 
we have
\begin{align*}
\int_0^\infty  du\, e^{-(s-t)u} J_{\alpha}(2\sqrt{t^{-1}s xu})J_{\alpha}(2\sqrt{s^{-1}t yu}) =&\int_0^\infty d\bar{x}\, 2\bar{x} e^{-\bar{\rho}^2\bar{x}^2}J_{\bar{p}}(\bar{\alpha}\bar{x})J_{\bar{p}}(\bar{\beta}\bar{x}).
\end{align*}
In view of \cite[(6.633-2), page 707]{GR},
\begin{multline*}
\int_0^\infty d\bar{x}\, 2\bar{x} e^{-\bar{\rho}^2\bar{x}^2}J_{\bar{p}}(\bar{\alpha}\bar{x})J_{\bar{p}}(\bar{\beta}\bar{x})=\frac{1}{\bar{\rho}^2}\exp\left( -\frac{\bar{\alpha}^2+\bar{\beta}^2}{4\bar{\rho}^2} \right)I_{\bar{p}}\left( \frac{\bar{\alpha}\bar{\beta}}{2\bar{\rho}^2} \right)\\
=(s-t)^{-1}e^{y/s-x/t}e^{-(x+y)/(s-t)}I_\alpha\left(\frac{2\sqrt{xy}}{s-t}\right) .
\end{multline*}
As a result,
\begin{multline*}
 e^{-y/s+x/t}(y/x)^{\alpha/2} \int_0^\infty du\,  e^{-(s-t)u} J_{\alpha}(2\sqrt{t^{-1}s xu})J_{\alpha}(2\sqrt{s^{-1}t yu})\\
 =(s-t)^{-1}(y/x)^{\alpha/2}e^{-(x+y)/(s-t)} I_\alpha \left( \frac{2\sqrt{xy} }{s-t} \right)=T_{\alpha}((t,x);(s,y)).
\end{multline*}
This finishes the proof of Lemma~\ref{lem:T}.
\end{proof}

\begin{proof}[Proof of Lemma~\ref{lem:W}]
The proof is similar, first using a change of variables,
\begin{align*}
u=\bar{t}^2,\ \alpha=\bar{\mu},\ \beta=\bar{\nu}+1,\ 2\sqrt{x}=\bar{\beta}, 2\sqrt{y}=\bar{\alpha}.
\end{align*} 
we have
\begin{align*}
 \int_0^\infty  du\, u^{-(\beta-\alpha)/2}J_{\alpha}(2\sqrt{  xu})J_{\beta}(2\sqrt{  yu}) =&\int_0^\infty  d\bar{t}\, 2 \bar{t}^{-(\bar{\nu}-\bar{\mu})}J_{\bar{\nu}+1}(\bar{\alpha}\bar{t})J_{\bar{\mu}}(\bar{\beta}\bar{t}) .
\end{align*}
By \cite[(6.575-1), page 683]{GR},
\begin{align*}
\int_0^\infty  d\bar{t}\, 2 \bar{t}^{-(\bar{\nu}-\bar{\mu})}J_{\bar{\nu}+1}(\bar{\alpha}\bar{t})J_{\bar{\mu}}(\bar{\beta}\bar{t}) =&\mathbbm{1}(\bar{\beta}\leq \bar{\alpha}) \frac{(\bar{\alpha}^2-\bar{\beta}^2)^{\bar{\nu}-\bar{\mu}}\bar{\beta}^{\bar{\mu}}}{2^{\bar{\nu}-\bar{\mu}-1}\bar{\alpha}^{\bar{\nu}+1}\Gamma(\bar{\nu}-\bar{\mu}+1)}\\
 =&\mathbbm{1}(x\leq y) x^{\alpha/2} y^{-\beta/2} \frac{(y-x)^{\beta-\alpha-1}}{\Gamma(\beta-\alpha)}.
\end{align*}
Thus,
\begin{multline*}
 t^{-(\beta-\alpha) } e^{-(y-x)/t } y^{\beta/2}x^{-\alpha/2}   \int_0^\infty  du\, u^{-(\beta-\alpha)/2}J_{\alpha}(2\sqrt{  xu})J_{\beta}(2\sqrt{  yu})\\
 =\frac{1}{\Gamma(\beta- \alpha)}\frac{(y-x)^{\beta- \alpha-1}}{t^{ \beta- \alpha }}e^{-(y-x)/t} \mathbbm{1}(x\leq y)=W_t((\alpha,x);(\beta,y)  ).
\end{multline*}
This finishes the proof of Lemma~\ref{lem:W}.
\end{proof}
To prove Lemma~\ref{lem:Q}, we express $T_\alpha((t,x);(s,y))$ and $W_t((\alpha,x);(\beta,y))$ as Hankel transforms and then apply the Parseval identity \eqref{equ:Parseval}. For $t<s$ with $t,s\in (0,\infty)$, $\alpha<\beta$ with $\alpha,\beta \in\mathbb{N}_0$ and $x, y\in (0,\infty)$, define the functions
\begin{align*}
G_{\alpha,t,s,x}(u)&\coloneqq \exp\left( -\frac{s(s-t)}{4t}u^2 \right) u^{1/2} J_{\alpha}( t^{-1}s x^{1/2}u),\\
F_{\alpha,\beta,y}(u)&\coloneqq u^{-(\beta-\alpha)+1/2} J_{\beta}( y^{1/2}u),
\end{align*}
 
\begin{lemma}\label{lem:Hankel}
Fix $\alpha<\beta$ with $\alpha,\beta\in\mathbb{N}_0 $ and $t<s$ with $t,s\in (0,\infty) $. For any $x,y\in (0,\infty)$, we have
\begin{align*}
T_\alpha((t,x),(s,y) )&=2^{-1}(s/t)e^{-y/s+x/t}x^{-\alpha/2}y^{\alpha/2-1/4} H_{\alpha}\left[ G_{\alpha,t,s,x} \right](y^{1/2}),\\
W_t((\alpha,x);(\beta,y))&=2^{\beta-\alpha-1}t^{-(\beta-\alpha)}e^{-(y-x)/t}x^{-\alpha/2-1/4}y^{\beta/2} H_{\alpha}\left[F_{\alpha,\beta,y} \right](x^{1/2}).
\end{align*}
Here $H_\alpha$ is the Hankel transform defined in \eqref{def:Hankel}.
\end{lemma}
\begin{proof}
Through a direct computation,
\begin{align*}
H_{\alpha}\left[ G_{\alpha,t,s,x} \right](y^{1/2})=2(t/s)y^{1/4}\int_0^\infty du\, e^{-(s-t)u}J_{\alpha}(2\sqrt{t^{-1}sxu})J_{\alpha}(2\sqrt{s^{-1}tyu}).
\end{align*}
Therefore,
\begin{multline*}
2^{-1}(s/t)e^{-y/s+x/t}x^{-\alpha/2}y^{\alpha/2-1/4} H_{\alpha}\left[ G_{\alpha,t,s,x} \right](y^{1/2})\\
=e^{-y/s+x/t}(y/x)^{\alpha/2}\int_0^\infty du\, e^{-(s-t)u}J_{\alpha}(2\sqrt{t^{-1}sxu})J_{\alpha}(2\sqrt{s^{-1}tyu})=T_\alpha ((t,x),(s,y)).
\end{multline*}
We used \eqref{equ:TQ} in the second equality.
\vspace{\baselineskip}

By another direct computation,
\begin{align*}
H_{\alpha}\left[ F_{\alpha,\beta,y} \right](x^{1/2})=2^{-(\beta-\alpha)+1} x^{1/4} \int_0^\infty du\, u^{-(\beta-\alpha)/2} J_{\alpha}(2\sqrt{ xu})J_{\beta}(2\sqrt{yu}).
\end{align*}
which gives
\begin{multline*}
2^{\beta-\alpha-1}t^{-(\beta-\alpha)}e^{-(y-x)/t}x^{-\alpha/2-1/4}y^{\beta/2} H_{\alpha}\left[F_{\alpha,\beta,t,y} \right](x^{1/2})\\
= t^{-(\beta-\alpha)}e^{-(y-x)/t} x^{-\alpha/2 }y^{\beta/2}\int_0^\infty du\, u^{-(\beta-\alpha)/2} J_{\alpha}(2\sqrt{ xu})J_{\beta}(2\sqrt{yu})=W_t((\alpha,x);(\beta,y) ).
\end{multline*}
We used \eqref{equ:WQ} in the second equality.
\end{proof}

\begin{proof}[Proof of Lemma~\ref{lem:Q}]
Fix $t<s$ with $t,s\in (0,\infty) $, $ \alpha<\beta$ with $\alpha,\beta \in\mathbb{N}_0$ and $x,y\in (0,\infty)$. For simplicity, we write $G(u) $ for $G_{\alpha,t,s,x}(u)$ and $F(u)$ for $F_{\alpha,\beta,y}(u)$. From Lemma~\ref{lem:Hankel},
\begin{align*}
\int_0^\infty dz\, T_\alpha((t,x),(s,z))W_s((\alpha,z),(\beta,y)) 
\end{align*}
equals
\begin{align*}
2^{\beta-\alpha-2}(s/t) s^{-(\beta-\alpha)}e^{-y/s+x/t}x^{-\alpha/2}y^{\beta/2} \int_0^\infty dz\, z^{-1/2} H_\alpha[G](z^{1/2})H_\alpha[F](z^{1/2}).
\end{align*}
Through a change of variable, the above equals
\begin{align*}
2^{\beta-\alpha-1}(s/t) s^{-(\beta-\alpha)}e^{-y/s+x/t}x^{-\alpha/2}y^{\beta/2} \int_0^\infty dz\,  H_\alpha[G](z )H_\alpha[F](z ).
\end{align*}

Suppose that $G(u)$ and $F(u)$ are in $L^2(\mathbb{R}_+)$. Applying the Parseval identity \eqref{equ:Parseval}, we get
\begin{align*}
\int_{0}^\infty dz\,  H_\alpha[G](z)H_\alpha[F](z)=\int_0^\infty du\, G(u)F(u)  .
\end{align*}
By a direct computation, 
\begin{multline*}
\int_0^\infty du\, G(u)F(u) \\=  2^{-(\beta-\alpha)+1} \left(  t/s \right)^{-(\beta-\alpha)/2+1} \int_0^\infty du\, e^{-(s-t)u}u^{-(\beta-\alpha)/2}J_{\alpha}(2\sqrt{t^{-1}s xu})J_{\beta}(2\sqrt{s^{-1}t yu}).
\end{multline*}
As a result,
\begin{multline*}
\int_0^\infty dz\, T_\alpha(x,z;s-t)W_s(z,y;\beta-\alpha)
\\=(st)^{-(\beta-\alpha)/2} e^{-y/s+x/t} x^{-\alpha/2}y^{ \beta/2}
\int_0^\infty du\,  e^{-(s-t)u}u^{-(\beta-\alpha)/2}J_{\alpha}(2\sqrt{t^{-1}s xu})J_{\beta}(2\sqrt{s^{-1}t yu}).
\end{multline*}

It remains to check that $G(u)$ and $F(u)$ are in $L^2(\mathbb{R}_+)$. 
Using $s-t>0,\beta-\alpha\geq 1$ and the asymptotics in \eqref{equ:asymptoticzero} and \eqref{equ:asymptoticinfinity}, it is straightforward to show that
\begin{align*}
G(u)=&\exp\left( -\frac{s(s-t)}{4t}u^2 \right) u^{1/2} J_{\alpha}( t^{-1}s x^{1/2}u),\ F(u)=u^{-(\beta-\alpha)+1/2} J_{\beta}( y^{1/2}u) 
\end{align*}
are both in $L^2(\mathbb{R}_+)$. This finishes the proof.
\end{proof}

\subsection{Proof of Lemma~\ref{lem:Qintegral-s}}\label{sec:B.2}
In this section, we prove Lemma~\ref{lem:Qintegral-s}. We start by recalling relevant definitions.
\begin{align}
\overline{T}_{\alpha}((t,x);(s,y) )\coloneqq &(s-t)^{-1}(y/x)^{\alpha/2}e^{-(x+y)/(s-t)} I_\alpha \left( \frac{2\sqrt{xy} }{s-t} \right),\label{def:T-s-app}\\
\overline{W}_t((\alpha,x);(\beta,y) )\coloneqq &\frac{y^{\beta}x^{-\alpha}(x-y)^{\alpha-\beta-1}}{\Gamma(\alpha-\beta)}\mathbbm{1}(y\leq x).\label{def:W-s-app}
\end{align}
\begin{equation}\label{equ:Qform-s-app}
\overline{Q} ((\alpha,t,x);(\beta,s,y))=\left\{ \begin{array}{cc}
\overline{T}_\alpha((t,x);(s,y)), &\textup{if}\ \alpha=\beta\ \textup{and}\ t<s,\\
\overline{W}_t((\alpha,x);(\beta,y) ), &\textup{if}\ \alpha>\beta\ \textup{and}\ t=s,\\
\int_0^\infty dz\, \overline{W}_t((\alpha,x);(\beta,z) )\overline{T}_\beta((t,z);(s,y)), &\textup{if}\ \alpha>\beta\ \textup{and}\ t<s.
\end{array} \right.
\end{equation}

We record Lemma~\ref{lem:Qintegral-s} as Lemma~\ref{lem:Qintegral-s-app} below.
\begin{lemma}\label{lem:Qintegral-s-app}
Let $(\alpha,t)\prec_{\textup{s}} (\beta,s)$ be space-like pairs in $ \mathbb{N}_0\times (0,\infty)$ as in Definition \ref{def:spacelike} and $x,y\in (0,\infty)$. It holds that
\begin{equation}\label{equ:Qintegral-s-app}
\begin{split}
\overline{Q}((\alpha,t,x);(\beta,s,y))= &x^{-\alpha/2}y^{ \beta/2}\int_0^\infty du\,  e^{-(s-t)u}u^{-(\alpha-\beta)/2}J_{\alpha}(2\sqrt{ xu})J_{\beta}(2\sqrt{  yu}).
\end{split}
\end{equation} 
\end{lemma}
From \eqref{equ:Qform-s-app}, we split \eqref{equ:Qintegral-s-app} into three cases.
\begin{lemma}\label{lem:T-s}
For any $t<s$ with $t,s\in (0,\infty) $, $\alpha\in\mathbb{N}_0$ and $x,y\in (0,\infty)$, it holds that
\begin{align}\label{equ:TQ-s}
\overline{T}_{\alpha}((t,x);(s,y) )=&  (y/x)^{\alpha/2} \int_0^\infty du\,  e^{-(s-t)u} J_{\alpha}(2\sqrt{ xu})J_{\alpha}(2\sqrt{ yu}).
\end{align}
\end{lemma}
\begin{lemma}\label{lem:W-s}
For any $t>0$, $\alpha>\beta$ with $\alpha\,\beta \in\mathbb{N}_0$ and $x,y\in (0,\infty)$, it holds that
\begin{align}\label{equ:WQ-s}
\overline{W}_t((\alpha,x);(\beta,y) )=& y^{\beta/2}x^{-\alpha/2}   \int_0^\infty  du\, u^{-(\alpha-\beta)/2}J_{\alpha}(2\sqrt{  xu})J_{\beta}(2\sqrt{  yu}) .
\end{align}
\end{lemma}
\begin{lemma}\label{lem:Q-s}
For any $t<s$ with $t,s\in (0,\infty) $, $\alpha>\beta$ with $\alpha,\beta \in\mathbb{N}_0$ and $x,y\in (0,\infty)$, it holds that
\begin{equation*}	
\int_0^\infty dz\, \overline{W}_t((\alpha,x);(\beta,z) )\overline{T}_\beta((t,z);(s,y))=x^{-\alpha/2}y^{ \beta/2}\int_0^\infty du\,  e^{-(s-t)u}u^{-(\alpha-\beta)/2}J_{\alpha}(2\sqrt{ xu})J_{\beta}(2\sqrt{  yu}).
\end{equation*}
\end{lemma}

\begin{proof}[Proof of Lemma~\ref{lem:T-s}]
Using a change of variables,
\begin{align*}
u=\bar{x}^2,\ s-t=\bar{\rho}^2,\ 2\sqrt{x}=\bar{\alpha},\ 2\sqrt{y}=\bar{\beta},\ \alpha=\bar{p},
\end{align*} 
we have
\begin{align*}
\int_0^\infty du\,  e^{-(s-t)u} J_{\alpha}(2\sqrt{ xu})J_{\alpha}(2\sqrt{ yu}) =&\int_0^\infty d\bar{x}\, 2\bar{x} e^{-\bar{\rho}^2\bar{x}^2}J_{\bar{p}}(\bar{\alpha}\bar{x})J_{\bar{p}}(\bar{\beta}\bar{x}).
\end{align*}
In view of \cite[(6.633-2), page 707]{GR},
\begin{multline*}
\int_0^\infty d\bar{x}\, 2\bar{x} e^{-\bar{\rho}^2\bar{x}^2}J_{\bar{p}}(\bar{\alpha}\bar{x})J_{\bar{p}}(\bar{\beta}\bar{x})=\frac{1}{\bar{\rho}^2}\exp\left( -\frac{\bar{\alpha}^2+\bar{\beta}^2}{4\bar{\rho}^2} \right)I_{\bar{p}}\left( \frac{\bar{\alpha}\bar{\beta}}{2\bar{\rho}^2} \right)\\
=(s-t)^{-1}e^{-(x+y)/(s-t)}I_\alpha\left(\frac{2\sqrt{xy}}{s-t}\right) .
\end{multline*}
As a result,
\begin{multline*}
 (y/x)^{\alpha/2} \int_0^\infty du\,  e^{-(s-t)u} J_{\alpha}(2\sqrt{ xu})J_{\alpha}(2\sqrt{ yu})\\
 =(s-t)^{-1}(y/x)^{\alpha/2}e^{-(x+y)/(s-t)} I_\alpha \left( \frac{2\sqrt{xy} }{s-t} \right)=\overline{T}_{\alpha}((t,x);(s,y)).
\end{multline*}
This finishes the proof of Lemma~\ref{lem:T-s}. 
\end{proof}

\begin{proof}[Proof of Lemma~\ref{lem:W-s}]
As in the proof of Lemma~\ref{lem:W}, we have
\begin{align*}
\int_0^\infty du\, u^{-(\alpha-\beta)/2}J_\alpha(2\sqrt{xu})J_\beta(2\sqrt{yu})=\mathbbm{1}(y\leq x)x^{-\alpha/2}y^{\beta/2}\frac{(x-y)^{\alpha-\beta-1}}{\Gamma(\alpha-\beta)}.
\end{align*}
Therefore,
\begin{multline*}
y^{\beta/2}x^{-\alpha/2}   \int_0^\infty  du\, u^{-(\alpha-\beta)/2}J_{\alpha}(2\sqrt{  xu})J_{\beta}(2\sqrt{  yu}) \\=\mathbbm{1}(y\leq x)\frac{y^{\beta}x^{-\alpha} (x-y)^{\alpha-\beta-1}}{\Gamma(\alpha-\beta)}=\overline{W}_t((\alpha,x);(\beta,y) ).
\end{multline*}
This finishes the proof of Lemma~\ref{lem:W-s}. 
\end{proof} 
 
In the next lemma, we express $\overline{T}_\alpha((t,x);(s,y))$ and $\overline{W}_t((\alpha,x);(\beta,y) )$ as Hankel transforms. Then we use the Parseval identity \eqref{equ:Parseval} to prove Lemma~\ref{lem:Q-s}. For $t<s$ with $t,s\in (0,\infty)$, $\alpha>\beta$ with $\alpha,\beta\in\mathbb{N}_0$ and $x, y\in (0,\infty)$, define the functions
\begin{align*}
\overline{ G}_{\alpha,t,s,y}(u)&\coloneqq \exp\left( -\frac{ (s-t)}{4 }u^2 \right) u^{1/2} J_{\alpha}(  y^{1/2}u),\\
\overline{F}_{\alpha,\beta,x}(u)&\coloneqq u^{-(\alpha-\beta)+1/2} J_{\alpha}( x^{1/2}u),
\end{align*} 
 
\begin{lemma}\label{lem:Hankel-s}
\begin{align*}
\overline{T}_\alpha((t,x);(s,y))&=2^{-1}x^{-\alpha/2-1/4}y^{\alpha/2 }H_\alpha[\overline{G}_{\alpha,t,s,y}](x^{1/2}),\\
\overline{W}_t((\alpha,x);(\beta,y) )&=2^{\alpha-\beta-1}x^{-\alpha/2}y^{\beta/2-1/4}H_\beta[\overline{F}_{\alpha,\beta,x}](y^{1/2}).
\end{align*}
\end{lemma}
\begin{proof}
Through a direct computation,
\begin{align*}
H_\alpha[\overline{G}_{\alpha,t,s,y}](x^{1/2})=2x^{1/4}\int_0^\infty du\, e^{-(s-t)u}J_\alpha(2\sqrt{xu})J_\alpha(2\sqrt{yu}).
\end{align*}
Therefore,
\begin{multline*}
2^{-1}x^{-\alpha/2-1/4}y^{\alpha/2 }H_\alpha[\overline{G}_{\alpha,t,s,y}](x^{1/2})\\
=  (y/x)^{\beta/2 }\int_0^\infty du\, e^{-(s-t)u}J_\alpha(2\sqrt{xu})J_\alpha(2\sqrt{yu})=\overline{T}_\alpha((t,x);(s,y)).
\end{multline*}
We used \eqref{equ:TQ-s} in the second equality.
\vspace{\baselineskip}

Through a direct computation,
\begin{align*}
H_\beta[\overline{F}_{\alpha,\beta,t,x}](y^{1/2})=2^{-(\alpha-\beta)+1}y^{1/4}\int_0^\infty du\, u^{-(\alpha-\beta)/2} J_{\alpha}(2\sqrt{xu})J_{\beta}(2\sqrt{yu}).
\end{align*}
Therefore,
\begin{multline*}
2^{\alpha-\beta-1}x^{-\alpha/2}y^{\beta/2-1/4}H_\beta[\overline{F}_{\alpha,\beta,t,x}](y^{1/2})\\
= x^{-\alpha/2}y^{\beta/2 } \int_0^\infty du\, u^{-(\alpha-\beta)/2} J_{\alpha}(2\sqrt{xu})J_{\beta}(2\sqrt{yu})=\overline{W}_t((\alpha,x);(\beta,y) ).
\end{multline*}
We used \eqref{equ:WQ-s} in the second equality.
\end{proof}

\begin{proof}[Proof of Lemma~\ref{lem:Q-s}]
Fix $t<s$ with $t,s\in (0,\infty) $, $\alpha>\beta$ with $\alpha,\beta\in\mathbb{N}_0$ and $x,y\in (0,\infty)$. For simplicity, we write $\overline{G}(u) $ for $\overline{G}_{\beta,t,s,y}(u)$ and $\overline{F}(u)$ for $\overline{F}_{\alpha,\beta,x}(u)$. From Lemma~\ref{lem:Hankel},
\begin{align*}
\int_0^\infty dz\,  \overline{W}_t((\alpha,x);(\beta,z) ) \overline{T}_\beta((t,z);(s,y))= 2^{\beta-\alpha-2}  x^{-\alpha/2}y^{\beta/2} \int_0^\infty dz\, z^{-1/2} H_\beta[\overline{G}](z^{1/2})H_\beta[\overline{F}](z^{1/2}).
\end{align*}
Through a change of variable, the above equals
\begin{align*}
2^{\beta-\alpha-1}  x^{-\alpha/2}y^{\beta/2} \int_0^\infty dz\,  H_\beta[\overline{G}](z)H_\beta[\overline{F}](z).
\end{align*}

Suppose that $\overline{G}(u)$ and $\overline{F}(u)$ are in $L^2(\mathbb{R}_+)$. Applying the Parseval identity \eqref{equ:Parseval}, we get
\begin{align*}
\int_{0}^\infty dz\,  H_\beta[\overline{G}](z)H_\beta[\overline{F}](z)=\int_0^\infty du\, \overline{G}(u)\overline{F}(u)  .
\end{align*}
By a direct computation, 
\begin{align*}
\int_0^\infty du\, \overline{G}(u)\overline{F}(u) =  2^{-(\beta-\alpha)+1} \int_0^\infty du\, e^{-(s-t)u}u^{-(\alpha-\beta)/2}J_{\alpha}(2\sqrt{ xu})J_{\beta}(2\sqrt{  yu}).
\end{align*}
As a result,
\begin{align*}
\int_0^\infty dz\,  \overline{W}_t((\alpha,x);(\beta,z) ) \overline{T}_\beta((t,z);(s,y)) =& x^{-\alpha/2}y^{ \beta/2}\int_0^\infty du\, e^{-(s-t)u}u^{-(\alpha-\beta)/2}J_{\alpha}(2\sqrt{ xu})J_{\beta}(2\sqrt{  yu}).
\end{align*}

It remains to check that $\overline{G}(u)$ and $\overline{F}(u)$ are in $L^2(\mathbb{R}_+)$. Using $s-t>0,\alpha-\beta\geq 1$ and the asymptotics in \eqref{equ:asymptoticzero} and \eqref{equ:asymptoticinfinity}, it is straightforward to show that
\begin{align*}
\overline{ G}(u)\coloneqq &\exp\left( -\frac{ (s-t)}{4 }u^2 \right) u^{1/2} J_{\beta}(  y^{1/2}u),\ \overline{F} (u)\coloneqq u^{-(\alpha-\beta)+1/2} J_{\alpha}( x^{1/2}u),
\end{align*} 
are both in $L^2(\mathbb{R}_+)$. This finishes the proof. 
\end{proof}

\end{appendix}
\bibliography{BesselField0919.bib}
\bibliographystyle{alpha}
\end{document}